\documentclass[reqno,12pt]{article}

\usepackage{amssymb,amsmath}
\usepackage{amsthm}
\usepackage{enumerate}
\usepackage{url}
\usepackage{booktabs}
\usepackage{makeidx}         % allows index generation
\usepackage{graphicx}        % standard LaTeX graphics toolq
\usepackage{color}
\usepackage{multicol}        % used for the two-column index
\usepackage[bottom]{footmisc}% places footnotes at page bottom

\newtheorem{theorem}{Theorem}
\newtheorem{proposition}[theorem]{Proposition}
\newtheorem{lemma}[theorem]{Lemma}
\newtheorem{definition}[theorem]{Definition}
\newtheorem{corollary}[theorem]{Corollary}
\newtheorem{remark}[theorem]{Remark}
\newtheorem{example}[theorem]{Example}

\newcommand\fractional[1]{\left\{#1\right\}}

\newcommand\floor[1]{\lfloor{#1}\rfloor}%\newcommand{\p}{{\mathfrak{p}}}

\newcommand{\uu}{{\mathfrak{u}}}
\newcommand{\lp}{{\mathfrak{l}}}
\newcommand{\C}{{\mathbb C}}
\newcommand{\N}{{\mathbb N}}

\newcommand{\Q}{{\mathbb Q}}
\newcommand{\R}{{\mathbb R}}

\newcommand{\Z}{{\mathbb Z}}

\newcommand{\CA}{{\mathcal A}}

\newcommand{\CC}{{\mathcal C}}

\newcommand{\CF}{{\mathcal F}}

\newcommand{\CH}{{\mathcal H}}

\newcommand{\CP}{{\mathcal P}}

\newcommand{\CS}{{\mathcal S}}
\newcommand{\CT}{{\mathcal T}}

\newcommand{\CW}{{\mathcal W}}

\newcommand{\lin}{\operatorname{lin}}
\renewcommand{\mod}{\operatorname{mod}}

\newcommand{\Res}{\operatorname{Res}}

\newcommand{\spec}{\operatorname{Spec}}

\newcommand{\tr}{\operatorname{Tr}}

\renewcommand{\a}{{\mathfrak{a}}}

\renewcommand{\c}{{\mathfrak{c}}}

\newcommand{\g}{{\mathfrak{g}}}
\renewcommand{\k}{{\mathfrak{k}}}
\renewcommand{\l}{{\mathfrak{l}}}

\newcommand{\z}{{\mathfrak{z}}}

\renewcommand{\t}{{\mathfrak{t}}}
\renewcommand{\u}{{\mathfrak{u}}}

\newcommand{\la}{\langle}
\newcommand{\ra}{\rangle}

\begin{document}

\title{ Multiplicity of compact group representations  and applications to Kronecker coefficients}
\author{Velleda Baldoni,  Mich\`ele Vergne}

\date{}

%*******************************************************************
%*******************************************************************

\maketitle

\begin{abstract}These notes are an expanded version of a talk given by the second   author.
Our main  interest is focused on the challenging problem of computing Kronecker coefficients.
We decided, at the beginning, to take a very general approach to the problem of studying  multiplicity functions,
and  we survey the various aspects of the theory that comes into play, giving  a detailed bibliography to orient the reader.
Nonetheless the  main general theorems involving multiplicities functions (convexity, quasi-polynomial behavior, Jeffrey-Kirwan residues) are stated without proofs.
Then, we present in detail our approach to the  computational problem, giving explicit formulae, and outlining an algorithm that  calculate many interesting examples, some of which  appear in the literature also in connection with Hilbert series.
\end{abstract}

\tableofcontents

\section{Introduction}

\bigskip

Let $K$ be a compact connected Lie group and let $\hat K$ be the set of classes of irreducible finite dimensional representations of $K$.
Let $V$ be a representation of $K$, which is a direct sum (possibly infinite) of irreducible finite dimensional representations of $K$  with finite multiplicities.
We write $$V=\oplus_{\pi\in \hat K} m_K(\pi)V^K_ \pi$$
where $V^K_ \pi$ is the irreducible representation of $K$ parameterized by $\pi.$ We call the function $\pi\to m_K(\pi)$ on $\hat K$ the multiplicity function of  the representation $V$.
The study of the multiplicity function $m_K(\pi)$ is important in representation theory, invariant theory, quantum information theory.

When the representation space $V$ is constructed by " geometric quantization" of a Hamiltonian manifold $M$, the moment map $\Phi$ on $M$ gives us a geometric interpretation of the multiplicity function of the representation $V$.
This is the famous $[Q,R]=0$ theorem, obtained by Meinrenken-Sjamaar \cite{M-S}.
This has consequences on the qualitative properties of the function $\pi\to m_K(\pi)$ that we will recall here
in two  examples, which are the paradigms  of geometric quantization:
\begin{itemize}
\item The quantization of a symplectic vector space under a linear action of $K$.
\item  The quantization of $T^*K$, the cotangent bundle of $K$.
\end{itemize}

In these two basic examples, we give  a direct  construction of the corresponding representation space $V$.
 We will not justify that the representation space $V$  is ``the quantized space" of $M$, but we will recall results on the multiplicity function   of the representation of $K$ in
$V$ in terms of the moment map on $M$ which confirm the fact that $V$ is the ``the quantized space" of the Hamiltonian space $M$.

\begin{example}Quantization of a symplectic vector space \end{example}
 Consider $M$  a $2 n$ dimensional symplectic vector space ($M=\R^{2n}$)  with a linear symplectic action  of the compact group  $K.$
We choose a Hermitian structure on $M$ (so we identify $M$ to $\C^n$)  such that $K$ acts unitarily on $M$.
We denote by $Sym^k(M)$ the complex space of $k$-symmetric tensors on  the complex space $M$.
Consider
 $$V=\oplus_{k=0}^{\infty}Sym^k(M),$$ the space of polynomial functions on $M^*$.

 The space $V$ may be completed as a Hilbert space,  by choosing a Gaussian inner product, and the completion of $V$ is  the Fock space.
In the philosophy of quantization, this Fock space  is the quantized space associated to the dual representation of $K$ in the symplectic space $M^*$.

If $K$ contains the homotheties, we then
have $$V=\oplus_{\pi\in \hat K} m_K(\pi)V_\pi^K$$ where $m_K(\pi)$ is finite.
When $K$ is abelian,  $m_K(\pi)$ is a  partition function.

Remark that the knowledge of the function $\pi\to m_K(\pi)$, for $\pi$  trivial on the semi-simple part of $K$, allows to compute the Hilbert series
of the ring of invariant polynomials on $M^*$ under the semi-simple group $[K,K]$.

\begin{example}Quantization of the cotangent space $T^*G$\end{example}\label{ex:cotangent}
Consider   a compact Lie group $G$.
Let us define $V=R(G)$ to be the subspace of $C^{\infty}(G)$ generated by the coefficients
$\langle g u_1,u_2\rangle$ of finite dimensional representations of $G.$ The space $L^2(G)$ is the Hilbert completion of $V$. In the philosophy of quantization, the space $L^2(G)$ ``is" the quantized space associated to the symplectic space $T^*G$.

By Peter-Weyl theorem, under the action of $G\times G$
by left and right translations:
$$V=\oplus_{\lambda\in \hat G} V_\lambda^G\otimes (V_\lambda^G)^*.$$

Let $K$ be a subgroup of $G$, and consider the subgroup $G\times K$ of $G\times G$. Under the action of $G\times K,$
$$V=\oplus_{\lambda\in \hat G,\mu\in \hat K} m_{G,K}(\lambda,\mu)V^G _\lambda\otimes (V_\mu^K)^*$$
where  $m_{G,K}(\lambda,\mu)$ is the multiplicity of the representation $\mu$ in the restriction of $\lambda$ to $K$.
This is the so called branching coefficient.

\bigskip

In both of these examples, we will recall some of the qualitative properties of the multiplicity function, in particular its
piecewise quasi-polynomial behavior on closed cones.
These examples being particular cases of geometric quantization
of Hamiltonian manifolds,  the moment cone, its decomposition in cones of quasi-polynomiality, and the determination of its faces,
 gives us already  information on the behavior of the multiplicity function.

We will then  recall our method to compute the multiplicity function $m_K(\pi)$, based on computations of partition functions via multi-dimensional
residues. We apply it to the challenging example of computation of Kronecker coefficients.
In particular, we obtain an algorithm which enables us to compute the dilated Kronecker coefficients associated to
Young diagrams  with $n$ rows (but  `` any number" of columns and any shape) in polynomial time when $n$ is fixed.
We implemented our  algorithm  as a simple Maple program. In particular, it computes
the dilated Kronecker coefficient $g(k\alpha,k\beta,k \gamma)$ in reasonable time (reasonable=less than 20 minutes)
for $3$ rows. We obtain a quasi-polynomial of degree $11$ for general $\alpha,\beta,\gamma$ (with coefficients periodic functions of period at most $12$) see Example \ref{C333}.
When $\alpha,\beta,\gamma$ are special, the degree might be much smaller.
For example,  the dilated Kronecker coefficient
$m(k)=g(k[1,1,1],k[1,1,1],k[1,1,1])$ corresponds to the Hilbert series
of the ring of invariants of $SL(3)\times SL(3)\times SL(3)$ in $\C^3\otimes \C^3\otimes \C^3$.
  An efficient way to represent   periodic functions   is by using  step-polynomials (\cite{Verdo}).
In this representation, $m(k)$ is given
    by the following quasi-polynomial
     {\small{$$1-\frac{3}{2}\fractional{\frac{1}{3}k}+\frac{3}{2}\fractional{\frac{1}{3}k}^2-\frac{3}{2}\fractional{\frac{1}{2}k}-\fractional{\frac{3}{4}k}^2+\fractional{\frac{3}{4}k}\fractional{\frac{1}{2}k}+\fractional{\frac{1}{2}k}^2+$$$$\left(\frac{1}{4}-\frac{1}{4}\fractional{\frac{1}{2}k}\right)k+\frac{1}{48}k^2$$}}
    Here  for $s\in\R$, the function $\fractional{s}=s-\floor{s} \in [0,1)$
    where     $\floor{s}$ denotes the largest integer smaller or equal
    to~$s$. It is easy to check that (fortunately)
    this result of our algorithm  agrees with Kac's determination \cite{Kac}
    of the ring of invariants  $[Sym(\C^3\otimes \C^3\otimes \C^3)]^{SL(3)\times SL(3)\times SL(3)}$, which is freely generated with generators in degree $2,3,4$ (see \cite{MAN}).
Indeed $$\sum_k m(k)t^k=\frac{1}{(1-t^2)(1-t^3)(1-t^4)}.$$

Another example is $$m(k)=g(k[3,3,3,3],k[4,4,4],k[4,4,4])=\dim [Sym^{12 k}(\C^4\otimes \C^3\otimes \C^3)]^{SL(4)\times SL(3)\times SL(3)}.$$
Our algorithm gives the Hilbert series    $$\sum_k m(k)t^k={\frac {1+{t}^{9}}{ \left( 1-{t}^{2} \right) ^{2} \left(1- {t}^{4}
 \right)  \left( 1-t\right)  \left(1- {t}^{3} \right) }}.$$

We will resume  with the study of Kronecker coefficients  in  Section \ref{application} giving all the details.
For example, we can compute easily the Hilbert series for $3,4,5$ qubits. The result for the case of $5$ qbits appears also in \cite{LUTHI} with a discrepancy due to a  small misprint  (see \ref{application}).

Our article, focused towards computational problems,  is strongly inspired by the  article of Christandl-Doran-Walter \cite{C-D-W} and we thank them for introducing us to the subject of quantum computing.
Our computational method differ at some points from their method, and we will discuss the differences. It is
based on  Jeffrey-Kirwan residues.
This allows us to produce a symbolic function
valid inside the cone of quasi-polynomiality, in particular along a line.

\bigskip
We believe that this survey might be of interest. The eventual reader that just want to
read  a discussion of general results in Hamiltonian geometry
and some applications of these results to combinatorics, without being  addicted to  computing,  might skip the technical parts,
 namely, Subsection \ref{tools}
inside Section \ref{linearrep}
as well as  Subsection  \ref{branch.general}
and Subsection \ref {singular} inside Section \ref{Branching Rules}, and might go directly to
Section \ref{section:kro}
 skiping again the technical subsection \ref{thealgoKro}.

 \section{Notations}\label{not}
 Let  $K$ be a compact connected Lie group and  $T_K$ a maximal torus of  $K$. We denote by $\k $ and $\t_\k$ the corresponding Lie algebras.
Denote by $\CW_\k$ the Weyl group, and $w\to \epsilon(w)=\det_{\t_\k} w$ its sign representation.

 The weight lattice $\Lambda_K$ of $T$ is a lattice in $i\t_\k^*$.
 If $\lambda\in \Lambda_K\subset i\t_\k^*$,  it determines a one dimensional representation  of $T_K$
 by $t\to e^{\langle \lambda,X\rangle}, $ with $t=\exp X,  \ X\in \t.$ As $\lambda$ takes imaginary values on $\t_\k$, $e^{\langle \lambda,X\rangle}$ is of modulus $1$.

We denote by $\Gamma_K\subset i\t_\k$ the dual lattice of $\Gamma_K$: if $\lambda\in \Lambda_K$, $\gamma\in \Gamma_K$,
then $\la \lambda,\gamma\ra$ is an integer.

Let $\Delta_\k\subset i\t^*$ be the root system for $\k$ with respect to $\mathfrak t_\k.$
If $\alpha\in \Delta_\k$, its coroot $H_\alpha$ is in $i\t_\k$, and $\la \alpha,H_\alpha\ra=2$.
Define $ \hat K$
  to be  the set of irreducible finite dimensional classes of complex  representations of  $K$.
Fix $\Delta_\k^+$, a positive system for $\Delta_\k$, and let
$$i\t^*_{\k,\geq 0}=\{\xi, \langle \xi, H_\alpha\rangle \geq 0, \alpha \in \Delta_\k^+\}$$
be the corresponding positive Weyl chamber.
Let $\rho_\k=\frac{1}{2}\sum_{\alpha\in \Delta_\k^+} \alpha$.

We denote by $\Lambda_{K,\geq 0}$ the  "cone" of dominant weights, that is the set $\Lambda_K\cap i\t^*_{\k,\geq 0}.$
An element $\lambda \in \Lambda_{K,\geq 0}$ is called a dominant weight.

When  the group $K$ is understood, we denote $\Lambda_K$ simply by $\Lambda$, $T_K$ by $T$,  $\t_\k$ by $\t$, etc.

The following example is the only example which will be needed when discussing Kronecker coefficients.

\begin{example} \label{Un} \end{example} We consider the case $K=U(n)$ and $T\subset K$ the torus consisting of the diagonal matrices.
Then  the Lie algebra $\k$ consists of the $n\times n$ anti Hermitian matrices  and $i\k$ is the space of   Hermitian matrices.
If we  identify $\k$ and $\k^*$ via the linear form $Tr(AB)$,
then
$\t=\t^*$   is the set of diagonal anti hermitian matrices.
Thus  the positive Weyl chamber  is
$i\t^*_{\geq 0}=\{\xi=[\xi_1,\xi_2,\ldots, \xi_n]\}$ with $\xi_j\in \R$ and
$\xi_1\geq \xi_2\geq \cdots \geq \xi_n$ where $\xi$ represents the hermitian matrix with diagonal entries $\xi_j$.
Denote by $|\xi|=\sum_i \xi_i$, the trace of  the matrix corresponding to $\xi$.

An element $X\in i\k^*$ is conjugated to a unique element $\xi=[\xi_1,\xi_2,\ldots, \xi_n]$
in $i\t^*_{\geq 0}$, where the list
$\xi_1,\xi_2,\ldots, \xi_n$ is the list  of  eigenvalues  (with their multiplicities) of $X$ ordered decreasingly.
We denote $\xi$  by $\spec(X)$.

 The ``cone" of dominant weights is
$\Lambda_{U(n),\geq 0}=\{\lambda=[\lambda_1,\lambda_2,\ldots, \lambda_n]\}$ with $\lambda_j\in \Z$ and
$\lambda_1\geq \lambda_2\geq \cdots \geq \lambda_n.$

If  $\lambda \in\Lambda_{U(n),\geq 0}$ is such that  $\lambda_n\geq 0$, it
indexes a finite dimensional irreducible polynomial representation of $GL(n,\C)$.
The corresponding subset of  $\Lambda_{U(n),\geq 0}$ will be denoted by $P\Lambda_{U(n),\geq 0}$.
If $\lambda\in P\Lambda_{U(n),\geq 0}$, we also identify  $\lambda$  to a Young diagram with $n$ rows.
The content of the corresponding diagram is the number of its boxes, that is $k=|\lambda|$.
The dominant weight $[k,k,\ldots,k]$ correspond to a rectangular   Young diagram  with $k$ columns and $n$ rows, and index the
one-dimensional representation $\det(g)^k$ of $U(n)$.

Assume now  $N\geq n,$ then there is a natural injection from $P\Lambda_{U(n),\geq 0}$ to $P\Lambda_{U(N),\geq 0}$  obtained just by adding more zeros
on the right of the sequence $\lambda.$ We denote by $\tilde{\lambda}$ the  new sequence so obtained. $\Box$

\bigskip

We  parameterize $\hat K$ by the set of elements  $(\pi_\lambda, V^K_\lambda)$ where $\pi_\lambda$ denotes the irreducible finite dimensional representation of highest weight  $\lambda \in \Lambda_{K,\geq 0}$ and $V^K_\lambda$ is the finite dimensional space on which $\pi_\lambda$ acts.
If there is no ambiguity on the group $K$,  we may write simply $V_\lambda$.
If $T$ is a torus, we also write $e^{\mu}$ for the one dimensional representation of $T$ associated to a weight $\mu$ of $T$.
The contragredient representation $V_\lambda^*$ of $K$ is indexed by the dominant weight
 $\lambda^*=-w_0(\lambda)$, where $w_0$ is the longest element in $\CW_\k$.

Let $\pi$ be a representation of  $K$  acting on a complex vector space $V$.
Assume $V$ is a direct sum (possibly infinite) of irreducible finite dimensional complex representations $\pi_\lambda$ of $K$, each of them occurring with finite multiplicities, and let
$$V=\oplus_{\lambda\in \hat K} m_K(\lambda) V_\lambda^K.$$
Assume that the restriction of the representation of  $K$ to its  maximal torus $T$   is also with finite multiplicities $m_T(\mu)$.
We have
\begin{equation}\label{eq:multKmultT}
m_{K}(\lambda)=\sum_{w \in \CW_\k} \epsilon(w)m_{T}(\lambda+\rho_\k -w(\rho_\k)).
\end{equation}

This is a consequence of the denominator formula:
\begin{equation}\label{eq:denom}
\prod_{\alpha\in \Delta_\k^+}(1-e^{\alpha})=\sum_{w \in \CW_\k} \epsilon(w)e^{\rho_\k -w(\rho_\k)}.
\end{equation}

If $X\in \k$, we denote by
 $K(X)$ the subgroup of $K$ stabilizing $X$.
  If $K$ is compact connected, then $K(X)$ is a compact connected subgroup of $K$.

If $\CH$ is a Hermitian vector space, define $\CH_{pure}=\{v\in \CH, \langle v, v\rangle=1\}$, the set of elements of $\CH$ of norm $1$.
Such an element is called a pure state.
A density matrix is a semi-positive definite Hermitian matrix with trace $1$.
A pure state in the space $\C^2$ is called a qubit, and
a pure state in $\CH=(\C^2)^{\otimes N}$  called a $N$ qubit.

\bigskip

For us, a cone $C$ in a real vector space $E$  is a closed subset of $E$ containing $0$,  invariant by positive homotheties.
We consider here only polyhedral cones, that is $\xi\in C$  if and and only if $\xi$ satisfies a certain number of  linear inequations
$\la X_a,\xi\ra\geq 0$, with $X_a\in E^*$.
When we say cone, it will always mean a polyhedral cone.
Most of the time, $E$ will be equipped with a lattice $L.$
Consider the dual lattice $L^*$ in $E^*.$ We say that $C$ is a rational cone if
the $X_a$ are in $L^*.$
In other words, if $E=\R^n$, and $L=\Z^n$, a rational cone is a cone defined by inequations with integral coefficients.

We will say that a cone $C$ is solid if $C$ has non empty interior.
If $F$ is a finite set of vectors in $E$, we denote by $Cone(F)$ the cone generated by $F$.

An affine cone is a translate $s+C$ of a cone.
An open cone will mean the interior of a polyhedral cone, that is a set determined by strict linear inequations
$\la X_a,\xi\ra > 0$.
If $C$ is a polyhedral cone in a vector space,
 we say that $C=\cup C_a$ is a cone decomposition of $C$  if the  $C_a$ are closed polyhedral cones of dimension equal to the dimension of $C$, and  if, when $a\neq b$, the intersection
$C_a\cap C_b$ is contained in the boundary of the cone $C_a$ and $C_b$.

A face $F$ of a cone $C$ generates a linear space that we call $lin(F)$.
A facet of $C$ is a face of codimension $1$ in $C$.

A polytope is a closed compact convex subset of $E$ determined by linear affine inequations
$\la X_a,\xi\ra\geq c_a$, with $X_a\in E^*$, $c_a\in \R$.
If $X_a\in L^*$, and $c_a\in \Q$, the polytope is rational.

\section {Basic examples for multiplicities}
\begin{example} {The Littlewood-Richardson coefficients}\end{example}

We consider two irreducible representations  $(\pi_\lambda,V_\lambda)$ and  $(\pi_\mu,V_\mu)$  of $K.$  Their tensor product
$V=V_\lambda\otimes V_\mu $ is a $ K \times K$ representation  with action defined by:
$(\pi_\lambda \otimes \pi_\mu )(k_1,k_2)=\pi_\lambda(k_1) \otimes \pi_\mu (k_2).$
The group $K$ acts diagonally on $V$ via $\pi(k)=\pi_\lambda(k) \otimes \pi_\nu(k)$, so we may consider  the representation $\pi=\pi_\lambda \otimes \pi_\nu$ restricted to $K.$
We can write  the classical formula
  $$(V_\lambda \otimes V_\mu)_{|K}=\oplus_{\nu \in \Lambda_{\geq 0}}c_{\lambda,\mu}^{\nu} V_\nu$$
where  $  c_{\lambda,\mu}^{\nu}$  is the multiplicity of $\pi_\nu$ in $ (\pi_\lambda \otimes \pi_{\mu})_{|K}$.

 The numbers $c_{\lambda,\mu}^{\nu}$ are called {\it {Littlewood- Richardson coefficients.}}

The function $k\to c_{k\lambda,k\mu}^{k\nu}$ of $k\in \{0,1,2,\ldots\}$ is called the dilated
Littlewood-Richardson coefficient.
It follows (see Theorem \ref{theo:mGK}) from the $[Q,R]=0$ theorem that
the function $k\to c_{k\lambda,k\mu}^{k\nu}$ is given by a quasi-polynomial formula for $k\geq 1$
(polynomial in the case  of $U(n)$, see Proposition \ref{polyhorn})).

We recall that Cochet, \cite{C1} , \cite{C2}, has given an algorithm to compute
the dilated Littlewood-Richardson coefficient for all classical root systems.
This algorithm is available on \cite{Vergnewebpage}  $\Box$

\begin{example} {The Kronecker coefficients}\label{exa:Kron}\end{example}
Let $N=n_1n_2\cdots n_s$ where $n_1,\, n_2,\ldots ,n_s$ are positive integers and write $\C^N=\C^{n_1} \otimes \C^{n_2} \otimes \cdots  \otimes \C^{n_s} $.  Consider the   action of the group $K=U(n_1)\times \cdots \times U(n_s)$ on the complex vector space $\C^{N}$, where $(k_1,k_2,\ldots,k_s)\in K$ acts by
$k_1\otimes k_2\otimes \cdots\otimes k_s$ on $\C^{n_1} \otimes \C^{n_2} \otimes \cdots  \otimes \C^{n_s} $.
Consider the space  $V=\oplus_{k=0}^{\infty}V^k$, where $V^k=Sym^k(\C^N)$ is the space of symmetric tensors of degree $k$.
Write $$V=\oplus g(\lambda_1,\lambda_2,\ldots, \lambda_s) V^{U(n_1)}_{\lambda_1}\otimes \cdots \otimes V^{U(n_s)}_{\lambda_s}.$$
Here  $\lambda_j$ are polynomial representations of $U(n_j)$ and are indexed by Young diagrams with $n_j$ rows.

Considering the action of the center, we see that all diagrams $\lambda_j$ occurring in $V^k$
 have content $k$, so that they also index irreducible representations
$\pi_{\lambda_j}$ of the
symmetric group $\Sigma_k$.
By Schur duality, $ g(\lambda_1,\lambda_2,\ldots, \lambda_s) $ is the multiplicity of the trivial representation of $\Sigma_k$ in
$\pi_{\lambda_1}\otimes \cdots \otimes \pi_{\lambda_s}$.
The numbers $ g(\lambda_1,\lambda_2,\ldots ,\lambda_s)$ are called   {\it{Kronecker coefficients.}}
The function $k\to g(k\lambda_1,k\lambda_2,\ldots ,k\lambda_s)$ of $k\in \{0,1,2,\ldots\}$ is called the dilated
Kronecker coefficient.
It follows again from the $[Q,R]=0$ theorem that
the function $k\to g(k\lambda_1,k\lambda_2,\ldots ,k\lambda_s)$ is given by a quasi-polynomial formula for $k\geq 1$, (Proposition \ref{dilatedKro}).

Denote by $M=n_2n_3 \cdots  n_s$. In computing Kronecker coefficients,
we may assume $n_1\leq M$, and that $n_1$ is the maximum of the $n_i$.
Indeed, if $n_1\geq M$, the multiplicities
$g(\lambda_1,\lambda_2,\ldots, \lambda_s)$ stabilize in the sense that
$g(\lambda_1,\lambda_2,\ldots, \lambda_s)$ is non zero only if $\lambda_1$
is obtained from
an element in   $ P\Lambda_{U(M),\geq 0}$, by adding more zeroes on its right,
and multiplicities coincide.
Moreover, we have
  $$g(\lambda_1,\lambda_2,\ldots,  \lambda_s)=g(\tilde\lambda_1,\tilde \lambda_2,\ldots, \tilde \lambda_s).$$
Thus it is sufficient to study Kronecker coefficients in the case where
$n_1\leq M=n_2 n_3\cdots n_s$  and where $n_i$ is the number of rows of the tableau corresponding to $\lambda_i$.

To describe an algorithm to compute  the dilated Kronecker coefficient is the main computational  objective of this article.
We will  state some of the previous  results and  our results in Section \ref{section:kro}.
Our algorithm uses a  branching rule from $U(n_2 n_3\cdots n_s)$ to
$U(n_2)\times U(n_3)\times \cdots \times U(n_s)$   in order to reduce (slightly) the size of the problem.
Our Maple program is available on \cite{Vergnewebpage} $\Box$%{\bf End of example}

\section {Linear representation of $K$ in a Hermitian vector space $\CH$}\label{linearrep}

Let $\CH$ be a finite dimensional Hermitian vector space provided with a representation of $K$ by unitary transformations.

Assume that $K$ contains the subgroup of homotheties $\{e^{i\theta}Id_\CH\}$.
Consider $V=Sym(\CH)$, the space of symmetric tensors, so we have
$$Sym(\CH)=\oplus_{\lambda\in \hat K} m_K^{\CH}(\lambda) V_\lambda^K$$
where $m_K^{\CH}(\lambda)$ is finite.

It is clear that if $m_K^{\CH}(\alpha)$ and $m_K^{\CH}(\beta)$ are non zero, then
$m_K^{\CH}(\alpha+\beta)$ is non zero.
Indeed the product of two  non zero vectors $f,g$ in
$Sym(\CH)$ is non zero, and if $f,g$ are highest weight vectors of weights $\alpha,\beta$, the product is a highest weight vector of weight
$\alpha+\beta$. So the support of the multiplicity function $m_K^{\CH}(\alpha)$
is a ``discrete cone" (that is a semi-group). We will relate this discrete cone  to the moment map and to the
 Kirwan  cone.

\subsection{Examples of decomposition of $Sym(\CH)$}

\begin{example}\label{basicex}\end{example}

Let $\CH=\C^n$
and let $K=S^1$ acting  by the homothety
$$(z_1,z_2,\ldots, z_n) \to (e^{i\theta}z_1,e^{i\theta}z_2,\ldots, e^{i\theta}z_n).$$

Then  $$V=Sym(\CH)=\oplus_{k=0}^{\infty} \binom{n-1+k }{n-1} e^{ik\theta}$$
since  $dim(Sym^k(\CH))=\binom{n-1+k }{n-1} .$

\noindent So the multiplicity function $k\to dim(Sym^k(\CH))$ is a polynomial function of $k$.

\begin{example}\label{knapsack} The Knapsack \end{example}

Again, let $\CH=\C^n$
and let $K=S^1$ acting  by  $$(z_1,z_2,\ldots, z_n) \to (e^{iA_1\theta}z_1,e^{iA_2\theta}z_2,\ldots, e^{iA_n\theta}z_n),$$
where now the $A_i$ are any positive integers.

Then $Sym(\CH)=\sum_k m(k) e^{ik\theta}$
where $m(k)$ is the number of solutions in non negative integers $x_i$ of the knapsack equation
$$A_1x_1+A_2 x_2+\cdots+A_n x_n=k.$$

The computation of  the function $m(k)$ is  an  "intractable" problem, as illustrated in the lecture \url{
https://www.youtube.com/watch?v=2IbJf4oXOxU&feature=youtu.be} by P. Van Hentenryck.
See however \cite{BBDDKV} for results on its highest coefficients.

\begin{example} Cauchy formula\label{exa:Cauchy formula}
\end{example}
Let $N,  n$ be positive integers, and assume $N\geq n$.  Let  $\lambda=[\lambda_1,\ldots,\lambda_n] $  be a  sequence of weakly decreasing non negative integers with $\lambda_n\geq 0$. We consider $\lambda$ as an element of  $\Lambda_{U(n),\geq 0}$, that is a dominant polynomial weight for $U(n)$.
To $\lambda\in \Lambda_{U(n),\geq 0}$ is  associated an irreducible representation of $U(n)$ that we have denoted by $V_{\lambda}^{U(n)}.$
Recall that if   $N\geq n,$ then there is a natural injection $\lambda \to \tilde \lambda$  from $P\Lambda_{U(n),\geq 0}$ to $P\Lambda_{U(N),\geq 0}$  obtained just by adding more zeros
on the right of the sequence $\lambda.$
The decomposition of $Sym(\C^{n}\otimes \C^N)$ with respect to $U(n)\times U(N)$  (see \cite{KP}, page 63) is given by {\it {Cauchy formula}}:

\begin{equation}  Sym(\C^{n}\otimes \C^N)=\oplus_{\lambda \in P\Lambda_{U(n),\geq 0}}V_\lambda^{U(n)}\otimes V_{\tilde\lambda}^{U(N)}. \end{equation}
\bigskip

\begin{example} Clebsch-Gordan coefficients \label{exa:Clebsch} \end{example}

Consider the representation $\Pi$ of $K= U(d) \times U(d) \times U(d)$ on
$\mathcal H = \mathfrak{gl}(d) \oplus \mathfrak{gl}(d)$ given by
\[ \Pi(g, h, k) (A,B) = (g A k^{-1}, h B k^{-1}), \] where $\mathfrak{gl}(d)$ is the Hilbert space of complex $d \times d$ matrices equipped with the trace inner product $\la A,B \ra:= \tr A B^*$.
This action factorizes through the center $Z=S^1$ of $U(d)$ embedded in $U(d)\times U(d)\times U(d)$ as a diagonal subgroup
$(z{\rm Id} ,z{\rm Id} ,z{\rm Id})$.
We can write
$$Sym(\CH)=\oplus_{\lambda, \mu, \nu } c_{\lambda,\mu}^{\nu} V_\lambda \otimes V_\mu \otimes V_\nu^*.$$

Here $\lambda,\mu,\nu$ vary in
$\Lambda_{U(d),\geq 0}$ and
 $c_{\lambda,\mu}^{\nu}$ is the multiplicity of the representation $V_\nu$ in
$V_\lambda\otimes V_\mu$.
For $c_{\lambda,\mu}^{\nu}$ to be non zero, we need $|\lambda|+|\mu|=|\nu|$, as seen by considering the action of the center.

\subsection{The moment cone}
In the following, the compact group  $K$  will be fixed, and we denote simply by $T$ its maximal torus, $\t$ its Lie algebra, etc., as we stated in the Section \ref{not}

 Consider the moment map $\CH\to i\k^*$ given by
$$\Phi_K(v)(X)=\langle Xv,v\rangle.$$
Here $X\in\k$, and we have denoted by $v\to Xv$ the infinitesimal action of $X\in \k$ in $V$ by a anti-hermitian transformation, so
$\langle Xv,v\rangle$ is purely imaginary.

\begin{remark}\end{remark}
We consider $\CH$ as a symplectic manifold, with symplectic form
$\frac{1}{-2i}\la dv,dv \ra$
(if $\CH=\C^n$ with coordinates $z_k$, this is the form $\frac{-1}{2i}\sum_k dz_k \overline{dz_k}$),
then $\frac{1}{2i}\Phi_K$ is the moment map for the action of $K$ in  $\CH$ in the sense of Hamiltonian geometry.

$\Box$

\bigskip

We consider $i\t^*_{\geq 0}$ as a subset of $i\k^*$.

\begin{definition}

Define

$\bullet$  $C_K(\CH)=\Phi_K(\CH)\cap i\t^*_{\geq 0}$

and

$\bullet$
$\Delta_K(\CH)=\Phi_K(\CH_{pure})\cap i\t^*_{\geq 0}.$

\end{definition}

Assume that $K$ contains the homotheties
and let $J\in i\t$ so that the infinitesimal action of $J$ is the identity on $\CH$.
Then a pure state $v$ is such that $\Phi_K(v)(J)=1$.

Recall the following theorem, which is a particular case of Kirwan theorem \cite{Kirwan.84.bis}
 (a proof of this theorem, following closely Mumford argument, \cite{Mum},  can be found in \cite{Be}).

\begin{theorem}
$\bullet$ The set $C_K(\CH)=\Phi_K(\CH)\cap i\t^*_{\geq 0}$ is a  rational polyhedral cone.
We call $C_K(\CH)$ the Kirwan cone.

$\bullet$ The set $\Delta_K(\CH)$ is a rational polytope.
We call $\Delta_K(\CH)$ the Kirwan polytope.

\end{theorem}

Thus there exists a finite number of elements $X_a\in \Gamma$ such that
$$C_K(\CH)=\{\xi\in i\t^*_{\geq 0}| \langle X_a,\xi\rangle \geq 0\}.$$
We say that the inequations $\langle X_a,\xi\rangle \geq 0$ are the inequations of the cone $C_K(\CH)$.
We may normalize $X_a$ to be a primitive element in $\Gamma$, the dual lattice to $\Lambda$.
The set $\Delta_K(\CH)$ is the intersection of $C_K(\CH)$ with the affine hyperplane $\la J,\xi\ra=1$
and $$C_K(\CH)=\R_{\geq 0} \Delta_K(\CH)$$ is the cone over the Kirwan polytope.

It is in general quite difficult to determine the explicit inequations of the cone $C_K(\CH)$.
An algorithm to describe the inequations of this cone, based on Ressayre's notion of dominant pairs, \cite{Ressayre-inventiones},
is given in Vergne-Walter, \cite{V-W}.
Let us first give the emblematic  example of the Horn cone.

\begin{example} Horn problem \end{example}
Consider $3$ Hermitian matrices $X,Y,Z$. Let $\spec(X),\spec( Y),\spec( Z)$ denote
 the list of eigenvalues of the Hermitian matrices $X,Y,Z$.
 The Horn inequations describe the range of the triple  $(\spec(X),\spec(Y),\spec(Z))$, when the $3$ matrices $X,Y,Z$ are constrained by the relation $X+Y+Z=0$.
This problem is related to the moment map for Example \ref{exa:Clebsch} as follows.

Consider the representation of $K=U(d) \times U(d) \times U(d)$ on
$\mathcal H = \mathfrak{gl}(d) \oplus \mathfrak{gl}(d).$
The moment map  is
\[ \Phi_K(A, B) = (A A^*, B B^*, -A^* A - B^* B). \]
Since any non-negative Hermitian matrix can be written in the form $A^* A$
and since the spectra of $A A^*$ and $A^* A$ are equal, the moment cone is equal to
\[ C_K(\CH) := \{ (\spec (X), \spec (Y), \spec (Z)) : X, Y \geq 0, Z \leq 0, ~ X + Y + Z = 0 \}, \]
As proved in \cite{Klyachko}, \cite{Knu-Tao01} (see also \cite{Belk}), the equations of $C_K(\CH)$ are given by the inductive system of inequalities
conjectured by Horn \cite{Horn62}.
These inequations are of the following form. Let $I,J,K$ be subsets of $[1,2,\ldots,d]$, all three of them  of cardinal $r< d$.
Let $E_I$ be the diagonal Hermitian matrix with $r$   eigenvalues $1$ in positions $I$, and others being $0$.
  Then  the triple $(E_I,E_J,E_K)$ gives rise to the inequation
 \begin{equation}\label{eqHorn}
 Tr(E_I D_1)+  Tr(E_JD_2)+ Tr(E_K D_3)\leq 0
 \end{equation}
   on triples $(D_1,D_2,D_3)$ of Hermitian diagonal matrices.

   Horn defined inductively, for every $r< d$, a set
$H(r,d)$ of  triples  $(I,J,K)$ of subsets of cardinal $r$ of  $[1,2,\ldots,d]$.
Then
 $C_K(\CH)$ is described by the inequations (\ref{eqHorn}) above, for
 all $(I,J,K)\in H(r,d)$ and all $r< d$, and the equation $Tr(D_1)+Tr(D_2)+Tr(D_3)=0.$

\bigskip

We now consider examples related to Kronecker coefficients.

\begin{example} (Bipartite case) \label{exa:Cauchy formula2}.
\end{example}
We return to Example \ref{exa:Cauchy formula}.
 Consider $\CH=\C^n\otimes \C^N$ with the action of $U(n)\times U(N)$.
   Using the Hermitian inner product,  we identify $A\in \CH$ to a matrix $A:\C^n\to \C^N$.
   Then the moment map  is given by
$$\Phi_K(A)=[AA^*, A^*A]$$ with value Hermitian matrices of size $n$ and size $N$ respectively.
The matrices $AA^*$
and $A^*A$ have the same non zero eigenvalues.
Assume that $N\geq n$.
Define
$$Pi\t^*_{{\mathfrak u}(n),\geq 0}=\{\xi=[\xi_1,\ldots,\xi_n]; \xi_1\geq \xi_2\geq \cdots \geq \xi_n\geq 0\},$$
and consider  $Pi\t^*_{{\mathfrak u}(n),\geq 0}$ as a subset of the positive Weyl chamber $i\t^*_{\geq 0}$ for $U(n)$.
There is a natural injection from $Pi\t^*_{{\mathfrak u}(n),\geq 0}$  to $Pi\t^*_{{\mathfrak u}(N),\geq 0}$   obtained just by adding more zeros
on the right of the sequence $\xi.$ We denote by $\tilde{\xi}$ the  new sequence so obtained.
Then we see that
the Kirwan cone $C_K(\CH)$ is the "diagonal"
$(\xi,\tilde \xi)$ with $\xi\in  Pi\t^*_{{\mathfrak u}(n),\geq 0}.$

$\Box$

In the above example, we see that the
cone  $C_K(\CH)$ may have empty interior in $i\t^*_{\geq 0}$.
The following general result holds.

\begin{lemma}
The cone $C_K(\CH)$ is a solid cone if and only if
there exists $v\in \CH$ so that the stabilizer $K_v$ of $v$
is a finite group.
\end{lemma}

Return to the Kronecker case.

Consider $\CH=\C^{n_1}\otimes \C^{n_2}\otimes \cdots\otimes \C^{n_s}$
with action of $K=U(n_1)\times \cdots \times U(n_s)$ on $\CH$.

We may assume $n_1\geq n_2\geq \cdots \geq  n_s$, and let $M=n_2n_3\cdots n_s$.
The corresponding Kirwan cone $C_K(\CH)$ is a subset of
$$Pi\t^*_{{\mathfrak u}(n_1),\geq 0} \oplus Pi\t^*_{{\mathfrak u}(n_2),\geq 0}\oplus \cdots\oplus
Pi\t^*_{{\mathfrak u}(n_s),\geq 0}.$$
Then if $n_1\geq M$, this cone stabilizes in the sense
that $C_K(\CH)$ coincides with  the cone associated with
 the sequence $(M, n_2,n_3,\ldots, n_s)$ embedded by sending
 $Pi\t^*_{{\mathfrak u}(M),\geq 0}$ to $Pi\t^*_{{\mathfrak u}(n_1),\geq 0}$
 by the map $\xi\to \tilde \xi$.
Furthermore, considering the action of the center, we see that $C_K(\CH)$ is contained in the  subspace
$E=\{(y_1,\ldots,y_s)\}$ of $i\t^*_{\k}$ defined by  the $s-1$ linear equalities ${\rm Tr}(y_1)={\rm Tr}(y_2)=\cdots ={\rm Tr}(y_s)$.

For later use, we recall the following result (see \cite{V-W}).
\begin{proposition}\label{pro:solidVW}

Let $\CH=\C^{n_1}\otimes \C^{n_2}\otimes \cdots\otimes \C^{n_s}$
with action of $K=U(n_1)\times \cdots \times U(n_s)$ on $\CH$.
Assume that all $n_a$ are greater or equal to $2$, and that $s$ is greater or equal to $3$.
Assume also that $n_1=max(n_i)$ and $n_1\leq n_2\cdots n_s$.
Then $C_K(\CH)$ is solid in the subspace $E$.
\end{proposition}

A few examples of  $(n_1,n_2,\ldots, n_s)$ where the inequations of the cone
$C_K(\CH)$ are known are:

$(2,2,\ldots, 2)$  (with any number  $N$ of $2$: the cone of $N$-qubits,), (Higuchi--Sudbery-Sultz \cite{H-Su-Sz}).

$(3,3,3)$ (Franz \cite{MF}),

 $(4,2,2)$ (Briand \cite{BOR2}, Bravyi \cite{BR}),

$(6,3,2)$, $(9,3,3)$ (Klyachko \cite{Klyachko}),

$(4,4,4)$ and $(12,3,2,2)$ (Vergne-Walter, \cite{V-W})
(this last case being incorrect in Klyachko).

The general case seems for the moment out of reach.

\begin{example}\label{ex:bravyis}\end{example}
Let us write explicitly  Higuchi-Sudbery-Szulc  description of the cone of $N$-qubits (see also \cite{BR}).
Consider $\lambda_1=[\lambda_1^1,\lambda_2^2],\ldots, \lambda_N=[\lambda_1^N,\lambda_2^N]$ a sequence of $N$ elements
of $Pi\t^*_{\u(2),\geq 0}$ ( that is $\lambda_j^1\geq \lambda_j^2\geq 0$).
We assume that ${\rm Tr}(\lambda_1)={\rm Tr}(\lambda_2)=\cdots={\rm Tr}
(\lambda_N)$.
Then $(\lambda_1,\ldots,\lambda_N)\in C_{U(2)\times \cdots\times U(2)}(\C^2\otimes \cdots \otimes \C^2)$
if and only if
$$\lambda_j^2\leq \sum_{k\neq j} \lambda_k^2,$$ for any $j=1,2,\ldots, N$.

$\Box$

\begin{example} {Quantum marginals}\end{example}
Let $A,B,C$ be integers. If $v \in \C^A\otimes \C^B\otimes \C^C$ is a pure state,  its quantum marginals are defined as follows.
Consider $v$ as an operator $v: \C^A \rightarrow \C^B \otimes \C^C$  and define $\rho_A(v)=v^* v$.
 Define $\rho_B(v),\rho_C(v)$  similarly.
 Then $\rho_A(v), \rho_B(v), \rho_C(v)$ are  densities matrices of size $A,B,C$ respectively. They are called the quantum marginals of $v$.

 Consider $\CH= \C^A\otimes \C^B\otimes \C^C$, with action of $K=U(A)\times U(B) \times U(C)$.
 The moment map $\Phi_K: \C^A\otimes \C^B\otimes \C^C\to i\u(A)^*\oplus i\u(B)^*\oplus i\u(C)^*$
 is $$\Phi_K(v)=(\rho_A(v),\rho_B(v),\rho_C(v)).$$

 Given $3$ densities matrices $\rho_A,\rho_B,\rho_C$, the quantum marginal problem is to determine if
 there exists a pure state $v\in \C^A \otimes \C^B \otimes \C^C$ with quantum marginals $\rho_A,\rho_B,\rho_C$.
 We see that this pure state exists if and only if  $(\rho_A,\rho_B,\rho_C)$ is in the image of the moment map $\Phi_K$.
 It is thus necessary and sufficient that the triple of spectrums $(\spec(\rho_A),\spec(\rho_B),\spec(\rho_C))$ of $(\rho_A,\rho_B,\rho_C)$
 satisfy a certain number of linear inequalities.
 Unfortunately, there is not a good understanding of what are these inequalities, except in the few low dimensional cases quoted previously.

\subsection{Duistermaat-Heckman measure on the moment cone}

\bigskip
For simplicity, we assume that the moment cone $C_K(\CH)$ intersects the interior of the Weyl chamber. Of course, this is the case
when $C_K(\CH)$ is solid.
We also assume (we can always reduce easily to this case) that the kernel of the homomorphism $K\to U(\CH)$ is trivial.

\begin{definition}
Define
$$r=\dim_\C (\CH)- |\Delta_\k^+|-\dim C_K(\CH).$$
\end{definition}
Then, if the cone $C_K(\CH)$ is solid,
$$r=\dim_\C (\CH)- |\Delta_\k^+|-\dim \t.$$

Consider the Lebesgue measure $dv$ on $\CH$ (considered as a real vector space).
If $\xi\in i\k^*$, there is a natural measure $\beta_\xi$ on the coadjoint orbit $K\xi$ of $\xi$
determined by the symplectic structure of $K\xi$ (see \cite{Kiri}).
 Consider the Duistermaat-Heckman measure $DH^{\CH}_K(\xi)$ supported on the polyhedral cone $C_K(\CH)$,  determined by
$$\int_V f(\Phi_K(v)) dv=\int_{\xi\in C_K(\CH)} \big(\int_{K\xi}fd\beta_\xi\big) DH^{\CH}_K(\xi).$$
Here $f$ is a compactly supported continuous function on $i\k^*$.
In short, we divide the push forward $(\Phi_K)_*(dv)$ of the Lebesgue measure $dv$ on $\CH$  by the Kirillov measure
of the orbits in the image.
This measure is strictly positive on the relative interior of $C_K(\CH)$.

 Consider the Lebesgue measure $d\xi$ on the vector space
$lin(C_K(\CH))$
spanned by $C_K(\CH)$. Here we normalize $d\xi$ so that it gives mass $1$ to   a fundamental domain of $\Lambda\cap lin(C_k(\CH))$.

The following theorem follows from  Duistermaat-Heckman \cite{Dui-Hec}.
\begin{theorem} \label{multDH}
There exists a cone decomposition
$C_K(\CH)=\cup_{a\in F} \c_a$,  in {\bf closed} polyhedral cones $\c_a$,
and for each $a$, there exists a  homogeneous polynomial function
$d_{K,a}^{\CH}$ of degree $r$ on  $lin(C_K(\CH))$ such that
$$DH^{\CH}_K(\xi)=d_{K,a}^{\CH}(\xi)d\xi$$ if
$\xi\in \c_a$.
\end{theorem}

We denote by $dh^{\CH}_K$ the function on $C_K(\CH)=\cup \c_a$ such that $dh_K^{\CH}=d_{K,a}^{\CH}$ on $\c_a$.
So we obtain    a piecewise polynomial function $dh^{\CH}_K$  with support $C_K(\CH)$ and continuous on  $C_K(\CH)$.
We call the $\c_a$ {\bf  chambers of   polynomiality } (for  the  Duistermaat-Heckman measure).
It is already quite difficult to determine the cone $C_K(\CH)$, so even more so to determine the chambers of polynomiality $\c_a$.

The function $dh^{\CH}_K(\xi)$ is related to the volume of the reduced fiber of $\CH$.

\begin{definition}
Let $\xi\in i\k^*$, we define the reduced fiber  of $\CH$ at $\xi$  by
$\CH_{red,\xi}=\Phi_K^{-1}(K\xi)/K$.
\end{definition}
In other words, the reduced space at $\xi$ consists in all orbits $Kv$ of $K$ in $\CH$
 projecting on the orbit $K\xi$ under the moment map.
Via the identification of $\CH_{red,\xi}$ with the GIT  (geometric invariant theory) quotient of semi-stable orbits under $K_\C$,
 the reduced space can be provided with a structure of projective variety (see \cite{git}).

If all orbits of $K$ in  $\Phi_K^{-1}(K\xi)$ have the same dimension,
the reduced space is an orbifold. We say that $\xi$ is a quasi-regular value.
In particular, if the map $k\to kv$ is injective for all $v\in \CH$ projecting on $\xi$, we say that the action of $K$ on
  $\Phi_K^{-1}(K\xi)$ is free. In this case $\xi$ is a regular value of $\Phi_K$, and
 $\CH_{red,\xi}$ is a smooth manifold.
 Furthermore,  $\CH_{red,\xi}$ inherits a symplectic structure $\Omega_\xi$ from the symplectic
 form $\Omega$ on $\CH$: the restriction of $\Omega$ to
$\Phi_K^{-1}(K\xi)$ is the pull back of a form $\Omega_\xi$ on $\CH_{red,\xi}$.
For $\xi$ in the relative interior of the Kirwan cone,  and a quasi-regular value, then
the reduced space $\CH_{red,\xi}$ is of dimension $r$.
Duistermaat-Heckman theorem implies that
$dh^{\CH}_K(\xi)$ is the volume of the symplectic space $\CH_{red,\xi}$.
So when we restrict the function $dh^{\CH}_K$ to the line $t\xi$ ($\xi$ in the interior of $C_K(\CH)$ and quasi-regular), the function $t\to dh^{\CH}_K(t\xi)$ is
homogeneous of degree  equal to $r$.
This is also true without the quasi-regularity assumption.
\begin{lemma}
If $\xi$ is in the relative interior of the cone $C_K(\CH)$,
then $dh^{\CH}_K(t\xi)=t^{r} dh^{\CH}_K(\xi).$
\end{lemma}

So the homogeneous degree on a line stays constant for interior lines.
But typically, the function $dh^{\CH}_K$ vanishes on the boundary, except if it is constant on $C_K(\CH)$.

\begin{example}(The Bloch sphere)\end{example}
Let us consider a very simple example.

Let $T$ be a two dimensional torus, with Lie algebra $\t=\R J_1\oplus \R J_2$.
Let $\t^*$ with dual basis $J^1,J^2$.
We consider the diagonal action of   $T$  on $\CH=\C^2$, by
$\exp(\theta_1 J_1+\theta_2 J_2)(z_1,z_2)=(e^{i\theta_1} z_1,e^{i\theta_2} z_2)$.
Thus the weights $\phi_1,\phi_2$ of the action of $T$ on $\C^2$ are
$\phi_1=iJ^1, \phi_2=i J^2$.

 The  space of pure states divided by the action of $(e^{i\theta},e^{i\theta})$ is the Bloch sphere.
The Kirwan polytope is the segment $[\phi_1,\phi_2] $ in $\R \phi_1\oplus \R \phi_2=i\t^*$.
The Kirwan cone is the cone $\R_{\geq 0}\phi_1 \oplus\R_{\geq 0}\phi_2$.
The reduced fibers are points, and the Duistermaat Heckman measure is the characteristic function on this cone.
This is just Archimede result for which the area on the sphere between two latitudes depends only of the difference of their heights on the $z$ axes as one can see looking at the picture Figure
\ref{Kirwansphere}.
\begin{figure}[hh]
\centering
\includegraphics[height=1.5 in]{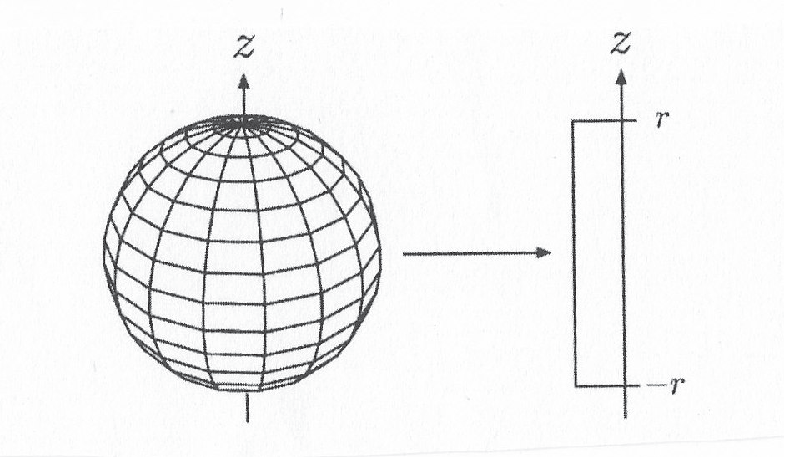}
\caption{The Kirwan polytope}
\label{Kirwansphere}
%\end{center}
\end{figure}

\subsection{Moment map and multiplicities}

\bigskip
We continue  (for simplicity) to assume that $C_K(\CH)$  intersects the interior of the Weyl chamber
and that the kernel of the homomorphism $K\to U(\CH)$ is trivial.

Consider the decomposition
$$Sym(\CH)=\oplus_{\lambda\in \hat K} m_K^{\CH}(\lambda) V_\lambda^K.$$

The cone $C_K(\CH)$  is related to the multiplicities through the following basic result, which is a particular case of Mumford theorem
\cite{Mum}(a proof of Mumford theorem, following closely Mumford argument, can be found in \cite{Be}).
\begin{proposition}
We have $m_K^{\CH}(\lambda)=0$ if $\lambda\notin C_K(\CH)$.
Conversely, if $\lambda$ is a  dominant weight belonging to
$C_K(\CH)$, there exists an integer $k>0$ such that
$m_K^{\CH}(k\lambda)$  is  non zero.
\end{proposition}

Thus the support of the function $m_K^{\CH}(\lambda)$ is contained in the Kirwan polyhedron
$C_K(\CH)$ and its asymptotic support is exactly the cone $C_K(\CH)$.

\begin{example} (The Cauchy formula, next) \end{example}

Return to the example of the Cauchy formula.
 Let $\CH=\C^n\otimes \C^N$ under the action of $U(n)\times U(N)$.
 The Kirwan cone has been determined in Example \ref{exa:Cauchy formula2}.
We see that  the multiplicity function  determined in Example \ref{exa:Cauchy formula} is supported exactly on the set
$\Lambda\cap C_K(\CH)$ (and with multiplicity $1$). $\Box$

\bigskip

It is interesting to understand, for a given $\lambda\in C_K(\CH)\cap \Lambda_{\geq 0}$, what is the smallest positive $k$ such that
$m_K^{\CH}(k\lambda)\neq 0$. We call this $k$ the saturation factor. We will give one example,  Example \ref{firstnonzero},
 where $k=2.$ We would have liked to find larger $k$, but we could not succeed.

\bigskip

We now describe the nature of the function $m_K^{\CH}(\lambda)$ on $C_K(\CH)$.
For this, we need to recall the definition of a periodic polynomial function.

Consider  a  real vector space $E$ with a lattice $L.$
We can think of $E=\R^d$ and $L=\Z^d.$
Let $\mathcal C(L)$ be the space of functions on the lattice  $L$.
 The restriction to $L$ of a polynomial function on $E$ will be called a polynomial function on $L.$
Given an integer $q$,  a function on $L/qL$ will be called a periodic function on $L$ of period $q$. If $q=1$ we just have a constant function.
\begin{definition}A periodic polynomial function on $L$ is a function on
$L$ which  is a linear combination of products of
polynomial functions  with periodic functions.
\end{definition}

We also say that  a periodic polynomial function $p$ is a quasi polynomial function.

\noindent The space of periodic polynomial functions is graded:
if $p(\lambda)=\sum_{i,j} c_i(\lambda) p_j(\lambda)$ where the polynomials $p_j$ are homogeneous of degree $k$ and the functions
 $c_i$ periodic, we say that $p$ is homogeneous of degree $k$.
 As for polynomials, we say that $p$ is of degree $k$ if $p$ is a sum of homogeneous terms of degree less or equal to $k$, and the term of degree $k$ is non zero.
If all functions $c_i(\lambda)$ are of period $q$, we say that $p(\lambda)$ is of period $q$.
\begin{example}\end{example}
$$m(k)=\frac{1}{2}k^2+k+\frac{3}{4}+\frac{1}{4}(-1)^k$$
is a periodic polynomial function
of $k\in \Z$ and of degree $2$.

$\Box$

In other words, if  $p(\lambda)$ is of period $q$, for any $\lambda_0$, the function $\lambda \to p(\lambda_0+q\lambda)$ is a polynomial
function on $L$. So we can represent a periodic polynomial function as a family of polynomials indexed by $L/qL$.
If $q$ is very large, the above description is not efficient (the numbers of cosets being quite large). Nonetheless we will give an example of this description in Section \ref{HS}.

In practice, $p$ will be naturally obtained as a sum of quasi-polynomial functions $p_1,p_2,\ldots,p_u$ of
periods $q_1,q_2,\ldots, q_{u}$. So $p$ is of period $q$ where $q$ is the least common multiple of $q_1,q_2,\ldots,q_u$.
However, it is already more efficient to keep $p$ as represented as $\sum p_i$, the number of cosets needed to describe
each $p_i$ being $q_i$, and $\sum q_i$ is usually much smaller that $q$.
We thus will say that the set of periods of the quasi-polynomial function $p$  is the set  $\{q_1,q_2,\ldots,q_{u}\}$.

	An efficient way to represent  periodic functions   by step-polynomials is given in \cite{Verdo}, \cite{BBDKV}, as in the example we gave in the introduction.

\bigskip

In the case of one variable, we can give the following characterization of quasi-polynomial functions $p(k)$.
If the function $p(k)$ is quasi-polynomial,
its generating series $\sum_{k=0}^{\infty} p(k) t^k$ is the Taylor expansion for $|t|<1$ of a rational function
$R(t)=\frac{P(t)}{\prod_{i=1}^s(1-t^{a_i})}$, where the $a_i$ are integers, and $P(t)$ a polynomial in $t$
of degree strictly less than $\sum_i a_i$.
The correspondence is  as follows.
Consider   a quasi polynomial $p(k)$ of period $q$, equal to $0$ on all cosets except the coset $f+q\Z$, with $0\leq f< q$.
Write  the polynomial function   $j\to p(f+qj)$  of degree $R$ in terms of binomials:
$p(f+qj)=\sum_{n=0}^R a(n) \binom{j+n}{n}$.
Then $$\sum_{j=0}^{\infty} p(f+qj) t^{f+qj}=t^f\sum_{n=0}^Ra(n)\frac{1}{(1-t^q)^{n+1}}.$$

(Given a rational function $R(t)=\frac{1}{\prod_{i=1}^s(1-t^{a_i})}$,
an algorithm (in polynomial time if  the number  $s$  of factors is fixed)
to compute explicitly the corresponding quasipolynomial  function $p(k)$  such that $\sum_k p(k)t^k=R(t)$ is given in  LattE integrale \cite{allourfriends}.)

\bigskip

As in the case of polynomial functions, to test if two quasi polynomials functions are equal, it
is sufficient to test it on a sufficiently large subset.

\begin{lemma}\label{lem:equaquasipoly}
Let $p_1,p_2$ be two quasi polynomial functions.
If there exists a  open cone $\tau$, such that $p_1,p_2$ agree on  a translate $s+\tau$
of $\tau$, then $p_1=p_2$.
\end{lemma}

Recall that $$r=\dim_\C (\CH)- |\Delta_\k^+|-\dim C_K(\CH).$$

The following theorem  can be considered as a quantum analogue of
Duistermaat-Heckman theorem,(Theorem \ref{multDH}.)

\begin{theorem} \label{theo:multH}
Consider the decomposition of the cone
$C_K(\CH)=\cup_a \c_a$,  in the {\bf closed} cones of polynomiality $\c_a$,
then,      for each $a$, there exists a quasi polynomial function
$p_{K,a}^{\CH}$ of degree $r$  on  the lattice $\Lambda$ such that  if
$\lambda\in \c_a\cap \Lambda$

$$ m_K^{\CH}(\lambda)=p_{K,a}^{\CH}(\lambda).$$ \end{theorem}

This theorem  is  Meinrenken-Sjamaar theorem
for the particular case of the action of $K$ in the projective space associated to $\CH$.
In this work, the function $m_K^{\CH}(\lambda)$ is related to the Kawasaki-Riemann-Roch number \cite{Kawasaki81}  (suitably defined) of the reduced symplectic space
 $\CH_{red,\lambda}=\Phi_K^{-1}(K\lambda)/K$.
One may look  in  \cite{Par-Ve} for a different proof.

Assume the cone $C_K(\CH)$ is solid, and let $\CH^{fin}$ be the open subset of $\CH$  consisting of elements with finite stabilizers.
Let $\c_a$ be a cone of polynomiality, choose a regular value $\xi\in \c_a$ of the moment map,
  and let $s_a^i$ be the orders of the stabilizers  $K_v$, for
 $v\in \Phi_K^{-1}(\xi)$.
 Then the set of periods  of the quasi-polynomial function $p_{K,a}^{\CH}$ is  contained in the set $\{s_a^i\}$.
 In particular, if the action of $K$ on $\CH$ has trivial generic stabilizer, the reduced spaces are smooth for regular values, and
 all the functions $p_{K,a}^{\CH}$ are polynomials.

Finding the set of periods is already quite non trivial.
It is related to the computation of the Picard group of the reduced spaces \cite{Kumar-Prasad}.
In the examples we study here, we compute  a set containing the set of periods by brutal force, as we consider here relatively small values
of the rank of $K$ and the dimension of $\CH$.

So, for any  dominant weight  $\lambda$  belonging to
$C_K(\CH)$,
the function $k\to m^{\CH}(k\lambda)$ is of the form:
$m^{\CH}_K(k\lambda)=\sum_{i=0}^{N}c_i(k) k^i$
where $c_i(k)$ are periodic functions of $k$.
This formula is valid {\bf for any} $k\geq 0$ (so $c_0(0)=1$).
The highest degree term for which this function is non zero
will be called the degree of  the quasi polynomial function $m_K^{\CH}(k\lambda)$.
We discuss this degree in the next subsection.

If $\lambda$ is not in $C_K(\CH)$, the function $k\to m^{\CH}_K(k\lambda)$ is just equal to $0$, except for $k=0$,   and conversely, if this function is not zero, then
$\lambda\in C_K(\CH)$.

An interesting particular case is when the center of $\k$ is one dimensional,
 $\chi$ a weight of $T$  vanishing  on $[\k,\k]\cap \t$, and such that $\chi(J)=1$.
 So $\chi$ indexes a one dimensional representation of $K$, and
 $m_K^{\CH}(k\chi)=\dim( [Sym^k(\CH)]^{[K,K]})$.
The corresponding generating series
$\sum_{k=0}^{\infty}m_K^{\CH}(k\chi)t^k$
is the Hilbert series $HS(t)$
of the ring of invariant polynomials  on $\CH$ under $[K,K]$.
So, the degree of the function $k\to m_K^{\CH}(k\chi)$ is  the maximal number of algebraically independent invariants.
There is a non trivial invariant $P$ (different from a constant) if and only if  $0$ belongs to the Kirwan polytope of the projective space $P(\CH)$:
equivalently, if the line $\R_{\geq 0}\chi$ is an edge of the cone $C_K(\CH)$. This is one of the first instance of the $[Q,R]=0$ theorem, and
this basic case follows from Mumford description \cite{git} of the GIT quotient.
 The function  $k\to m_K^{\CH}(k\chi)$ is a quasi polynomial on the full positive line $k\geq 0$
(not only for $k$ sufficiently large).
It is in general difficult to decide if   $0$ belongs to the Kirwan polytope of the projective space $P(\CH)$, and even more so to determine   the degree of the
 quasi-polynomial function $ k\to m_K^{\CH}(k\chi)$.

\bigskip

The two following propositions
are particular cases of Meinrenken-Sjamaar result.
\begin{proposition}\label{polyhorn}
Let $(\lambda,\mu,\nu)$  be a triple of dominant weights for $U(d)$, belonging to the Horn cone.
The dilated Clebsh-Gordon   coefficient  $k\to c_{k\lambda,k\mu}^{k\nu}$ is a polynomial function of $k$ for $k\geq 0$.
\end{proposition}
(If $(\lambda,\mu,\nu)$ is not in  the Horn cone, this function is identically $0$, except if $k=0$).

Indeed, we have already seen the quasi-polynomial nature of $k\to c_{k\lambda,k\mu}^{k\nu}$ for $k\geq 0$.
Now it is easy to see that the generic stabilizer of the action of $U(d)\times U(d) \times U(d)$
on $\CH$ is the center of $U(d)$ (embedded diagonally).
 Here is an explicit proof. Let $D_1,D_2$ be two positive definite Hermitian matrices, with distinct eigenvalues.
 We assume $D_1$ diagonalizable in the basis $e_1,e_2,\ldots, e_d$,
 while $D_2$ diagonalizable in a basis containing $\sum_i e_i$.
 Consider $A=D_1^{1/2}$, $B=D_2^{1/2}$ in $\mathfrak{gl}(d)$ and $(g,h,k)\in U(d)\times U(d)\times U(d)$ such that
 $gAk^{-1} =A, hBk^{-1} =B$.
 We obtain $A A^*=D_1=gD_1g^{-1}$, and $BB^*=D_2=hD_2h^{-1}.$
 Thus $g$ commutes with $D_1$, so is diagonalizable in the standard basis. So $g$ commutes with
 $A$, and $gAk^{-1}=A$ implies $g=k$. Similarly $h=k$, and $h$ is diagonalizable in the basis diagonalizing $D_2$.
 So we have $g=h=k$. But $g$ being diagonalizable in the basis $e_i$, and in a basis containing $\sum e_i$, the equation
 $g(e_1+e_2+\cdots+e_d)=a(e_1+e_2+\cdots+e_d)$ implies that all eigenvalues of $g$ are equal.

A proof of Proposition \ref{polyhorn} by more combinatorial methods  is given in \cite{De}.

\bigskip

We now consider the action of $K=U(n_1)\times U(n_2)\times\cdots \times U(n_s)$ in $\CH=\C^{n_{1}} \otimes \cdots \otimes \C^{{n_s}}$.

\begin{proposition}\label{dilatedKro}
Let $\lambda_i\in P\Lambda_{U(n_i),\geq 0}$ be polynomial dominant weights for $U(n_i)$ .
 We assume that $(\lambda_1,\ldots,\lambda_s)$ belongs to the Kronecker cone $C_K(\CH)$.
 The dilated Kronecker  coefficient  $k\to g(k\lambda_1,\ldots,k\lambda_s)$ is a quasi-polynomial function of $k$ for $k\geq 0$.
\end{proposition}
This result is  asserted in \cite{Mulmuley}.

(Remark that if  $(\lambda_1,\ldots,\lambda_s)$ is not in  $C_K(\CH)$, the dilated Kronecker coefficient is identically $0$, except if $k=0$).

\bigskip

 In the Kronecker case, we do not know the set of periods  of the dilated Kronecker coefficient.
 \begin{example} \end{example}
 For the case $n_1=n_2=n_3=3$, we find quasi polynomials with set of periods $\{1,2,3,4\}$ (see \ref{application}),
leading to polynomial behavior on cosets $f+12\Z$.

For the $4$ qubits  case $n_1=n_2=n_3=n_4=2$, we find quasi polynomials with set of periods $\{1,2,3,4\}$ (see \ref{application}),
leading to polynomial behavior on cosets $f+6\Z$.

For the $5$ qubits  case $n_1=n_2=n_3=n_4=n_5=2$, we find quasi polynomials with set of periods $\{1,2,3,4,5\}$ (see \ref{application}),
leading to polynomial behavior on cosets $f+60\Z$.

$\Box$

It would be  desirable  to describe
$m_K^{\CH}(\mu)$ as the number of integral points in a polytope.
This result would imply directly the quasi polynomial behavior of $m_K^{\CH}(k\mu).$
For Clebsh-Gordan coefficients,  this is the hive polytope of
Knutson-Tao \cite{Knu-Tao01}.
The corresponding computation is implemented in \cite{DLMA}. For Kronecker coefficients, no general result is known.

\bigskip

Let us end this section by recalling the behavior of multiplicities on the faces of $C_K(\CH)$.

Consider a face $F$  of the  cone $C_K(\CH)$. If $F$ contains elements of
$C_K(\CH)\cap i\t^*_{>0}$, we then say that $F$ is a regular face.
If $F$ is a regular face, let $T_F$ be the torus with Lie algebra $\lin(F)^{perp}$.
 Let $K_F$  be the centralizer of $T_F$, with Lie algebra $\k_F$, and system of positive roots
  $\Delta_{\k_F}^+$.
  Let $\CH(T_F)$ be the subspace of $\CH$ stable by $T_F$.
 This is a Hamiltonian space for $K_F$.

 The following reduction formula holds.

\begin{theorem}\label{th:restriction}
Let $F$ be a regular face of $C_K(\CH)$.
Then for any $\lambda\in F$, we have
$$m_K^{\CH}(\lambda)=m_{K_F}^{\CH(T_F)}(\lambda).$$

\end{theorem}
Again, this theorem is a  corollary of the $[Q,R]=0$ general theorem of Meinrenken-Sjamaar,
(see also  \cite{Par-Ve}).

Return to the example of the Horn cone. Let us consider a Horn triple $(I,J,K)$ such that the equation
(\ref{eqHorn}) determines a facet $F_{I,J,K}$ of the Horn cone. Let $I^c,J^c,K^c$ the complement of
$I,J,K$ in $[1,\ldots, d]$.
Then the centralizer of the element $(E_I,E_J,E_K)$ is
$\left (U(I)\times U(I^c)\right )\times \left (U(J)\times U(J^c)\right ) \times \left (U(K)\times U(K^c)\right )$.
Here we denoted by $U(I)$, the unitary group acting on $\C_I =\oplus_{i\in I} \C e_i$.

So when $\lambda, \mu,\nu$ belong to the facet $F_{I,J,K}$, the corresponding Clebsh-Gordan coefficient
is the product of the Clebsch-Gordan $c_{\lambda_I,\mu_J}^{\nu_K}$
coefficient relative to the group $U(r)$ with the
Clebsch-Gordan $c_{\lambda_{I^c},\mu_{J^c}}^{\nu_{K^c}}$
coefficient relative to the group $U(d-r)$.
So Theorem \ref{th:restriction}  implies the factorization theorem
of \cite{King-Tollu-Toumazet}, proved by combinatorial means.

\subsection{Comments on degrees }
We continue to assume that the moment cone $C_K(\CH)$ intersects the interior of the Weyl chamber and
 that the kernel of the homomorphism $K\to U(\CH)$ is trivial.

Recall the definition of
$$r=\dim_\C (\CH)- |\Delta_\k^+|-\dim C_K(\CH).$$

Consider $\c_a$ a cone for polynomiality and the two associated functions $d_{K,a}^{\CH}, p_{K,a}^{\CH}$.
The function $d_{K,a}^{\CH}$ is a polynomial function on $i\t^*$, while $p_{K,a}^{\CH}$ is a quasi polynomial function on $\Lambda$.
Then the two functions $d_{K,a}^{\CH}$ and $p_{K,a}^{\CH}$ are both of degree $r$.
In the case where $K$ is abelian,  (partition functions), the   term  of $p_{K,a}^{\CH}$  of highest degree $r$ is polynomial.
In general, this is not usually true.  The   term  of $p_{K,a}^{\CH}$  of highest degree $r$ might  be a quasi polynomial and not a polynomial.
For example, $r$ might be equal to $0$, however the functions
$p_{K,a}^{\CH}$ might be  periodic function of period $q>1$.

The following example was communicated to us by M.Walter.
\begin{example}\end{example}\label{ex:Walter}
Let us consider $U(2)$  and its torus $T$.
Let $i\t^*$ with basis $\epsilon_1=[1,0],\epsilon_2=[0,1]$
and the lattice of weights $\Lambda=\{\lambda=[\lambda_1, \lambda_2], \lambda_1,\lambda_2\in \Z\}$.

Consider $K=U(2)/\{\pm 1\}$ and its torus $T_K$.
Then the lattice of weights is $\Lambda_K=\{\lambda=[\lambda_1,\lambda_2]; \lambda_1,\lambda_2\in \Z; \lambda_1+\lambda_2 \ {\rm even}\}$.

 Consider the $3$-dimensional irreducible representation of  $K$ on $\CH=\C^3$ (with trivial kernel), with weights $[2,0],[1,1],[0,2]$.

 The Kirwan cone $C_T(\CH)$ is $\R_{\geq 0}\epsilon_1\oplus \R_{\geq 0}\epsilon_2$.
 Let $\c_1$ be the closed cone generated by $\epsilon_1,\epsilon_1+\epsilon_2$, and $\c_2$ generated by
 $\epsilon_1+\epsilon_2,\epsilon_2$.

 The Duistermaat-Heckman function for $T_K$ is given
by
$$dh^{\CH}_{T_K}(\xi_1\epsilon_1+\xi_2\epsilon_2)=\frac{1}{2}\xi_2$$
 on $\c_1,$
 $$dh^{\CH}_{T_K}(\xi_1 \epsilon_1+\xi_2\epsilon_2)=\frac{1}{2}\xi_1$$
 on $\c_2.$

 The multiplicity function for $T_K$ on $\Lambda_K$ is (with $\lambda_1+\lambda_2\in 2\Z$) is

$$m_{T_K}^{\CH}(\lambda_1 \epsilon_1+\lambda_2\epsilon_2)=\frac{1}{2}\lambda_2+\frac{3}{4}+(-1)^{\lambda_2}\frac{1}{4} $$
 on $\c_1,$

  $$m_{T_K}^{\CH}(\lambda_1 \epsilon_1+\lambda_2\epsilon_2)=\frac{1}{2}\lambda_1+\frac{3}{4}+(-1)^{\lambda_1} \frac{1}{4}$$
 on $\c_2.$

The (closed) Weyl chamber is $\c_1$ and we have $C_K(\CH)=\c_1$.

The Duistermaat-Heckman function is
$$dh^{\CH}_{K}(\xi_1\epsilon_1+\xi_2\epsilon_2)=1.$$

The multiplicity function is
$$m_{K}^{\CH}(\lambda_1 \epsilon_1+\lambda_2\epsilon_2)=\frac{1}{2}(1+(-1)^{\lambda_1}).$$
In particular,  the highest degree term (degree $0$) of the multiplicity function is not a polynomial.
Only when $\lambda_1$ is even, we obtain that  the highest degree term of $m_K^{\CH}$ (which is the constant function $1$)
coincide with the Duistermaat-Heckman polynomial. $\Box$

\bigskip

The degree $r$ is equal to $0$ if and only if multiplicities are bounded.
This implies that all the fibers $\Phi^{-1}(\xi)$ of the moment map are homogeneous spaces (the reduced fiber is a point).
In particular, if  the $K$ multiplicities of the representation $Sym(\CH)$ are bounded,  they take only the values $0$ or $1$.

We now consider a  face $F$ of $C_K(\CH)$.
The restriction to $F$ of the function $m_K^{\CH}$ is again a piecewise quasi-polynomial function.
 On $F\cap \c_a$, the multiplicity is given by the restriction of
 $p_{K,a}^{\CH}$ to $F\cap \c_a$.
 The degree of $p_{K,a}^{\CH}|_{lin(F)}$ (considered as a quasi-polynomial function $\lambda$ on the lattice $lin(F)\cap \Lambda$) might drop.
 If $F$ is a regular face, we can compute the degree using Theorem \ref{th:restriction}. Indeed
 \begin{lemma}
 If $F$ is a regular face,  then the degree of the function $m_K^{\CH}$ restricted to $F$ is the same on each cone
 $\c_a\cap F$ and is equal to
 $r_F=\dim_\C \CH(T_F)-\dim(F)-|\Delta^+_{\k_F}|$.
  \end{lemma}

 If $F$ is not a regular face, it is not eas§y to describe the degree of the restriction by such a simple formula.
In particular if $0$ is in the Kirwan polytope of $P(\CH)$,
the degree on the corresponding edge of the Kirwan cone is difficult to compute.

If $F$ is a regular face,  the degree of the multiplicity function restricted to any  $\c_a\cap F$ depends only of the linear span of $F$.
Let $w$ be an element of the Weyl group.
Although the Kirwan cone is not stable by the action of $\CW_\k$, we might have two faces $F_1$ and $F_2$,
so that their linear span is conjugated by $w$ (but of course not the intersection of this linear span with $C_K(\CH)$).
We then obtain  that the degree  of the multiplicity function is the same on $F_1,F_2$, see Example \ref{exa:632}.

\bigskip

Consider now  a line $\{k\lambda, k=0,1,2,\ldots\}$ contained in $C_K(\CH)$.
It is clear that the  degree of the quasi-polynomial function  $k\to m_K^{\CH}(k\lambda)$
is less or equal to $r$.
The reduced space $\Phi^{-1}(K\lambda)/K$, which may be singular,
is provided with a structure of projective variety. Thus we can compute in principle its dimension $r_\lambda$.
 More precisely, let us consider the open set $P^0$  of orbits of maximal dimension in
$P=\Phi^{-1}(K\lambda)$.
Then $P^0/K$ is an orbifold, and then $r_\lambda=\dim P^{0}/K$.
The following result is also a consequence of Meinrenken-Sjamaar.

\begin{theorem} \label{theo:beforemax}
The degree of the quasi-polynomial function  $k\to m_K^{\CH}(k\lambda)$
is equal to $r_\lambda$.
\end{theorem}

However, it is usually difficult to compute explicitly $r_\lambda$ using this geometric theorem.
We can do it when $\lambda$ is  in the interior of the Weyl chamber.
\begin{lemma}\label{max}
Let $\lambda$ be in the interior of the Weyl chamber.
Consider the minimal face $F$ of $C_K(\CH)$ containing $\lambda$.
Then the  degree of the quasi-polynomial function  $k\to m_K^{\CH}(k\lambda)$
is  equal to $r_F=\dim_\C \CH(T_F)-\dim(F)-|\Delta^+_{\k_F}|$.
\end{lemma}

Thus the behavior of the function $m_K^{\CH}$ is strikingly different from the behavior of the Duistermaat-Heckman measure which typically vanishes identically on the faces. So for a $\lambda$ which is not in the interior of the cone $C_K(\CH)$, the degree of the function
$dh^{\CH}_K(t\lambda)$ and of the function $m_K^{\CH}(k\lambda)$ are usually different.

A point $\lambda$ is called (weakly) stable, if the degree of $k\to m_K^{\CH}(k\lambda)$ is $0$. In other words,
the function $m_K^{\CH}(k\lambda)$ is bounded on the ray with generator $\lambda$.
 In this case the function $m_K^{\CH}(k\lambda)$ takes only values $0$ or $1$.
So this degree can be equal to $0$ only on the boundary of $\CH$, except in the case where multiplicities are bounded.
Meinrenken-Sjamaar result (or in fact, here GIT theory)  implies the following lemma.
\begin{lemma}
A point $\lambda\in C_K(\CH)$  is (weakly) stable
 if and only if $\Phi^{-1}(K\lambda)$ is  a $K$-orbit.
  \end{lemma}

A point  $\lambda$ is called a stable point, if $m_K^{\CH}(\lambda)=1$,
and if for any $\gamma \in \Lambda_{\geq 0}$, the function
$k\to  m_K^{\CH}(\gamma+k\lambda)$ is an increasing and bounded function of $k$.

The following result, conjectured by Stembridge \cite{Stem},  has been obtained recently by  \cite{St-S}.

\begin{proposition}
A weakly stable point is stable.
\end{proposition}

 L. Manivel \cite{MAN1}  has determined some  faces of the Kirwan cone $C_K(\CH)$
consisting of stable points in terms of compatible imbeddings,
and has described the corresponding stabilized multiplicity
$\lim_{k\to \infty} m_K^{\CH}(\gamma+k\lambda)$ in geometric terms.
Explicit descriptions in the case of the Kronecker cone
are given in \cite{MAN}.

Assume that the cone $C_K(\CH)$ intersect $\t^*_{>0}$.
If the inequations $X_a\geq 0$
of the Kirwan cone $C_K(\CH)$ are known,
the preceding discussion allows us to compute all stable points by elementary combinatorics,
except on the boundary of the Weyl chamber, see Example \ref{exa:632}.

\subsection {The  3-qubits example}\label{exa3qubits}

Before explaining our method to obtain  actual computations of the multiplicity function,
let us give a complete simple example,  \cite{BOR1}, \cite{BOR2}.
We  followed the exposition of \cite{C-D-W}.

We consider the action of
$U(2)\times U(2)\times U(2)$
on $\CH=\C^2\otimes \C^2\otimes \C^2.$

The representation of $U(2)\times U(2)\times U(2)$ in $Sym(\CH)$ decomposes as
$$Sym(\CH)=\oplus_{\lambda,\mu,\nu} g(\lambda,\mu,\nu) V_\lambda^{U(2)}\otimes V_\mu^{U(2)}\otimes V_\nu^{U(2)}$$
 over  polynomial irreducible
 representations
$\lambda,\mu,\nu$ of
$U(2)\times U(2)\times U(2)$.
They are indexed as
 $\lambda=[\lambda_1,\lambda_2], \mu=[\mu_1,\mu_2],  \nu=[\nu_1,\nu_2] $
with $\lambda_i, \mu_i, \nu_i\in \Z, i=1,2$, $\lambda_1\geq \lambda_2\geq 0$,
$\mu_1\geq \mu_2\geq 0$,
$\nu_1\geq \nu_2\geq 0.$

Though the dimension of a Cartan subalgebra of $U(2)\times U(2)\times U(2)$ is  $6$, we can restrict our attention to the subset
 $E \subset i\t^* $ defined by the equations
$$E=\{(\alpha,\beta,\gamma); \alpha_1+\alpha_2=\beta_1+\beta_2=\gamma_1+\gamma_2\}.$$
Here $\alpha=[\alpha_1,\alpha_2]$, with $\alpha_1,\alpha_2$ real numbers, and  $\alpha_1\geq \alpha_2$, etc..

Indeed, considering the action of the center,
it is easy that the Kirwan cone $C_K(\CH)$ is contained in $E$.
More precisely, the two dimensional subgroup
$$S=\{\left(
      \begin{array}{cc}
        t_1 & 0 \\
        0 & t_1 \\
      \end{array}
    \right), \left(
      \begin{array}{cc}
        t_2 & 0 \\
        0 & t_2 \\
      \end{array}
    \right), \left(
      \begin{array}{cc}
        t_3 & 0 \\
        0 & t_3 \\
      \end{array}
    \right);|t_i|=1;  t_1 t_2 t_3=1\}$$
    of $U(2)\times U(2)\times U(2)$ is contained in its center  and acts trivially on $\CH$.
    So we might consider $K=(U(2)\times U(2)\times U(2))/S$ acting on $\CH$, with Cartan subalgebra $\t$.
    Then $i\t^*$ is isomorphic to $E$.

Higuchi-Sudbery-Szulc equations \cite{H-Su-Sz}  (see  Example \ref{ex:bravyis}) for the Kirwan cone in $E$
are
$$C_K(\CH)=\left \{(\alpha,\beta,\gamma)\in E \ | \   \begin{array}{lll} \alpha_1\geq \alpha_2\geq 0, &\beta_1\geq \beta_2\geq 0,&\gamma_1\geq \gamma_2\geq 0,\\\alpha_2\leq \gamma_2+\beta_2,&\beta_2\leq \gamma_2+\alpha_2,&\gamma_2\leq \alpha_2+\beta_2\end{array} \right\}.$$

The image of pure states in $\CH$  is contained in the three dimensional affine space:
$$E_1=\{(\alpha,\beta,\gamma); \alpha_1+\alpha_2=\beta_1+\beta_2=\gamma_1+\gamma_2=1\}.$$

The group $\Sigma_3$ of permutations of $\alpha,\beta,\gamma$ acts on $E$.

We parameterize $E$ by $\R^4$ by associating to
$[[\alpha_1,\alpha_2], [\beta_1,\beta_2],[\gamma_1,\gamma_2]]$
the point $(t, \alpha_1-\alpha_2,\beta_1-\beta_2,\gamma_1-\gamma_2)$ in $\R^4$,
with $t= \alpha_1+\alpha_2= \beta_1+\beta_2=\gamma_1+\gamma_2$.
The Weyl chamber in
$\R^4=\{(t,x_1,x_2,x_3)\}$
is $x_i\geq 0$, and
Bravy's equations
become
$$\{x_1\geq 0,x_2\geq 0, x_3\geq 0, t+x_1\geq x_2+x_3,t+x_2\geq x_3+x_1, t+x_3\geq x_1+x_2\}.$$

As $E_1$ is isomorphic to $\{(1,x_1,x_2,x_3)\}$,  we can picture the Kirwan polytope  $\Delta(\CH)$
in $\R^+\times \R^+\times \R^+.$ It is the polytope  with $5$ vertices
$v_1,v_2,v_3,v_4,v_5$ (see Figure  \ref{Kirwan3qu}).
\begin{figure}[h]
\centering
\includegraphics[height=1.5 in]{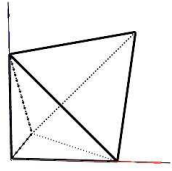}
\caption{Kirwan polytope for three qubits}
\label{Kirwan3qu}
\end{figure}

For example, the point $v_1:=(0,0,0)$
corresponds to  the triple $v_{min}$ in $E_1$, with  $$v_{min}=[[1/2,1/2],[1/2,1/2],[1/2,1/2]],$$
 and the point $v_2=(1,1,1)$ corresponds to the triple
  $$v_{max}=[[1,0],[1,0],[1,0]].$$.
The other vertices $v_3,v_4,v_5$ of  $\Delta(\CH)$  are
$v_3=[ 1, 0,  0], v_4= [0, 1,  0], v_5=[0,0, 1]$, corresponding to the
the triple $[[1,0],[1/2,1/2],[1/2,1/2]]$, and its permutations by $\Sigma_3$.

Thus the Kirwan polytope is an union of two  tetrahedra glued over the triangle $T$ with vertices
$[v_3,v_4,v_5]$.
Consider the center $c=(1/3,1/3,1/3)$
of this triangle, corresponding to the triple
$[[2/3,1/3],[2/3,1/3],[2/3,1/3]]$,
 and take the subdivision
of the triangle $T$ in $3$ triangles, $T_1,T_2,T_3$.
Then we have $6$ cones of polynomiality: the $3$ cones with basis the $3$ tetrahedra which are the convex hulls of
$(T_i,v_{min})$ or the $3$ cones with basis
$(T_i,v_{max})$.

Up to permutations by $\Sigma_3$, we obtain two cones.
Writing $\kappa=\alpha_1+\alpha_2=\beta_1+\beta_2=\gamma_1+\gamma_2,$

$$\CC_1=\{\alpha,\beta,\gamma\  | \  0\leq \alpha_1-\alpha_2
\leq min(\beta_1-\beta_2,\gamma_1-\gamma_2) \  {\text {and}}\
\alpha_1-\alpha_2+\beta_1-\beta_2+\gamma_1-\gamma_2\leq \kappa\}$$
and
$$\CC_2=\{\alpha,\beta,\gamma; 0\leq \alpha_1-\alpha_2
\leq min(\beta_1-\beta_2,\gamma_1-\gamma_2);
\alpha_1-\alpha_2+\beta_1-\beta_2+\gamma_1-\gamma_2\geq \kappa\}$$

Figure   \ref{Kirwan3quter}
represents  the cones $\CC_1$ or $\CC_2$ intersected with the hyperplane
$E_1$ of $E$.

\begin{figure}[h]\centering
\includegraphics[height=1.5 in]{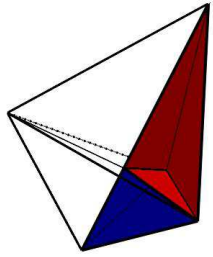}\caption{Cones $\CC_1$ and $\CC_2$ inside the Kirwan polytope}
\label{Kirwan3quter}
\end{figure}

The multiplicities on the cone $C_K(\CH)$
 have been described by Briand-Orellana-Rosas, \cite{BOR1}, \cite{BOR2}.
On each cone $g(\lambda,\mu,\nu)$ is a  periodic polynomial of degree $1$ and period $2$.
Set $k=\lambda_1+\lambda_2=\nu_1+\nu_2=\mu_1+\mu_2.$
Then

$$g(\lambda,\mu, \nu)=$$$$\left\{\begin{array} {ll} \frac{1}{2}(\lambda_1-\lambda_2)+
\frac{1}{4}\, \left( -1 \right) ^{\lambda_{{1}}+\mu_{{1}}+\nu_{{1}}}+\frac{1}{4}\, \left( -1 \right) ^{\lambda_{
{2}}+\mu_{{1}}+\nu_{{1}}}+\frac{1}{2}& \ {\text {if}} \  (\lambda,\mu, \nu) \in \CC_1\\
\lambda_{{1}}-\frac{1}{2}\,\mu_{{1}}+\frac{1}{2}\,\lambda_{{2}}-\frac{1}{2}\,\nu_{{1}}+\frac{1}{4}
\, \left( -1 \right) ^{\mu_{{1}}+\lambda_{{2}}+\nu_{{1}}}+\frac{3}{4}
 & \   {\text {if}} \  (\lambda,\mu, \nu) \in \CC_2\end{array}\right.$$
In particular  for example $v_{min}=[[1,1],[1,1],[1,1]]$ is in $\CC_1$  while $v_{max}=[[1,0],[1,0],[1,0]]$ is in $\CC_2.$

In this example, the multiplicity function takes values $0$ or $1$ on the boundary of $C_K(\CH)$.
In particular, we see that
$g([k,k],[k,k],[k,k])=\frac{1}{2}+(-1)^k\frac{1}{2}$, so that $g([k,k],[k,k],[k,k])$ is $1$, or $0$, if $k$ is even or odd, and
$$\sum_{k=0}^{\infty} g([k,k],[k,k],[k,k])t^k=\frac{1}{1-t^2},$$ as follows from the study of the Hilbert series of the ring of invariant
polynomials on
$\C^2\otimes \C^2\otimes \C^2$ under the action of $SL(2)\times SL(2)\times SL(2)$.

We conclude this example with Figure \ref{DHM3qu} that shows the Duistermaat-Heckman measure for three qubits. The drawing is along  the line
$[[1/2,1/2],[1/2,1/2],[1/2,1/2]]$ to $[[1,0],[1,0],[1,0]]$
between the bottom  vertex $v_{min}$ and the top vertex $v_{max}$
 of the Kirwan polytope up to the top vertex It grows linearly between
$v_{min}$ and $c$, then decreases from $c$ to $v_{max}$.
\begin{figure}[h]
\centering
\includegraphics[height=2 in]{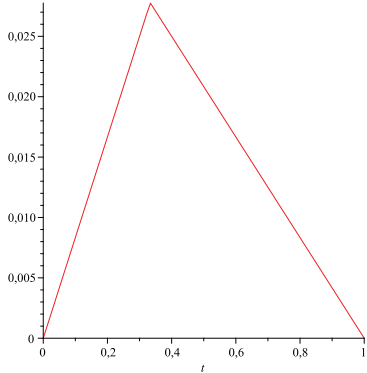}
\caption{Duistermaat-Heckman measure for three qubits: drawing  along  the line $[[1/2,1/2],[1/2,1/2],[1/2,1/2]]$ to $[[1,0],[1,0],[1,0]]$
}
\label{DHM3qu}
\end{figure}

Here is the discrete version of the multiplicity function.
We  compute
$g([6 k+s,6 k-s], [6 k+s,6 k-s],[6 k+s,6 k-s])$ for $s$ from $0$ to $6$.
Here is the answer:

                            $$[1, 1, 3, 2, 2, 1, 1],\hspace{1cm} {\rm for}\, k=1 $$
                   $$[1, 1, 3, 3, 5, 4, 4, 3, 3, 2, 2, 1, 1], \hspace{1cm} {\rm for}\,
                    k=2$$
          $$[1, 1, 3, 3, 5, 5, 7, 6, 6, 5, 5, 4, 4, 3, 3, 2, 2, 1, 1], \hspace{1cm}
         {\rm for} \,k=3$$

  It shows clearly the  quasi polynomial  behavior of the multiplicity on cones of polynomiality.

Our program  compute symbolically  the multiplicity function on the domains of polynomiality.

\subsection{Multiplicities and Partition functions }
We are now going to explicit the quasi polynomial functions determining
the Duistermat-Heckman  measure for the space $\CH$, and the multiplicity function
for $Sym(\CH)$
 whose existence is stated in Theorem \ref{multDH} and Theorem \ref{theo:multH}.

For computing multiplicities,
 the connection is made trough partition functions,
 whose computation is achieved  (using techniques developed in  \cite {BBCV}), via iterated residue of rational functions with poles on arrangement of hyperplanes.
 We give a simple example, explaining the philosophy  of the method, and we compare it with the algorithm
 of Guoce Xin \cite{Xin}  in Subsection \ref{sub:verysimple}.
 The reader may want first to read this subsection.
 We state   formulae  Theorem \ref{volumeT} and Theorem \ref{theo:DHK} for the  Duistermat-Heckman  measure and  Theorem \ref{multT}, and Theorem \ref{theo:multK} for the multiplicity.
Even if we could follow the same pattern,
 we will take a different approach for computing Kronecker coefficients.
  We will compute them as  a byproduct of the computation of branching coefficients, as explained  in detail in Section \ref{section:kro}.

To start with,  when $K=T$ is abelian, computation of multiplicities is the same problem than computing  the number of integral points in a (rational)
 polytope, while
 computation of Duistermaat-Heckman function is the same problem than computing  the volume  of a polytope.
In turn volumes  and   number of integral points in polytopes can be computed, as we just said,
as iterated residue of rational functions with poles on arrangement of hyperplanes. The precise statements are collected in Theorem \ref{volumeT} and Theorem \ref{multT}.

\bigskip

We write
$$Sym(\CH)=\oplus_{\mu\in \hat T} m_T^{\CH}(\mu) e^{\mu}.$$

Let $\Psi\subset i\t^*$ be the list of the weights for the action of $T$
on $\CH$ counted with multiplicities:
 $$\Psi=[\psi_1,\psi_2,\ldots,\psi_N].$$

We assume that the cone $Cone(\Psi)$ generated by  $\Psi$ is a pointed cone: $Cone(\Psi)\cap -Cone(\Psi)=\{0\}$, and that the lattice
of weights $\Lambda$ is generated by $\Psi$.

Let $\CP_\Psi$ be the function on $\Lambda$  that computes the number of ways we can write
$\mu\in \Lambda$ as $\sum x_i\psi_i$ with $x_i$ nonnegative integers.
The function $\CP_\Psi(\mu)$ is called the Kostant partition function (with respect to $\Psi$).
It is thus immediate to see that
$$m_T^{\CH}(\mu)=\CP_\Psi(\mu).$$

The moment map  $\Phi_T : \CH \rightarrow i\t^*$ is given by $$\Phi_T(\sum_{a=1}^N z_a e_a)=\sum |z_a|^2 \psi_a.$$
Here we have used an orthonormal basis $e_1,e_2,\ldots, e_N$ of $\CH$ where $T$ acts diagonally with weights
$[\psi_1,\ldots, \psi_N]$.
So the cone $C_T(\CH)$ is just  the cone
$Cone(\Psi)$ generated by the list $\Psi$ of weights.
Assume for simplicity that $\Psi$ generates $i\t^*$.
 It is thus a cone with non empty interior in $i\t^*$.

 Let $y\in i\t^*$.
 Define the polytope $$\Pi_\Psi(y)=\{[x_1,\ldots,x_N] \in \R^N, x_i\geq 0, \sum_{a=1}^N x_a\psi_a=y\}.$$
The Duistermaat-Heckman function $dh^{\CH}_T(y)$ is  the volume of the polytope $\Pi_\Psi(y)$
  and $m_T^{\CH}(\mu)$ ($\mu\in \Lambda$)  is the number of integral points in the polytope $\Pi_\Psi(\mu)$.

Define $C_T^{reg}(\CH)$ to be  the open subset of the cone  $C_T(\CH)$, where we removed from $C_T(\CH)$ the union of the boundaries of  the cones generated by all subsets of $\Psi$.
Let us write $C_T^{reg}(\CH)=\cup Y_a$ where $Y_a$ are the connected components  of $C_T^{reg}(\CH).$
Define $\c_a=\overline Y_a$, so we obtain a cone decomposition
$$C_T(\CH)=\cup_a \c_a.$$
Here all cones $\c_a$ are {\bf closed} and  have non empty interiors.

The following result explicit Theorem \ref{multDH} in this setting and the well  known behavior of the partition function .
\begin{proposition}
\begin{itemize}
\item The function  $dh^{\CH}_T(y)$ is given by a homogeneous polynomial function $d_{a}^{\Psi}$  on $\c_a.$
 The degree of $d_{a}^{\Psi}$ is equal to $|\Psi|-\dim \t$.
\item  There exists a quasi polynomial function $p_a^\Psi$ on $\Lambda$ such that
 $$\CP_\Psi(\mu)=p_a^\Psi(\mu)$$  on  the closed cone $\c_a.$
 The degree of $p_a^{\Psi}$ is equal to $|\Psi|-\dim \t$ and its highest degree term is equal to $d_{a}^{\Psi}$.

  \item In particular $m_T^\CH(\mu)=p_a^\Psi(\mu)$ on  the closed cone $\c_a.$
\end{itemize}
\end{proposition}

\bigskip

Let $K$ be a compact Lie group acting on $\CH$ by unitary transformations, and let us assume that
the action of its maximal Cartan subgroup $T$ has finite multiplicity. In other words, the weights of $T$ in $\CH$ span a pointed cone.
Then
$$Sym(\CH)=\oplus _{\lambda} m_K^{\CH}(\lambda) V_\lambda^K$$
and we have, for $\lambda\in \Lambda_{\geq 0}$,
$$m_K^{\CH}(\lambda)=\sum_{w\in \CW_\k} \epsilon(w) \CP_\Psi(\lambda+\rho_\k-w\rho_\k).$$

The cone $C_K(\CH)$ is contained in $C_T(\CH)\cap i\t^*_{\geq 0}$ but usually smaller.
Let us give some simple examples.

\begin{example}   \end{example}
Let us consider the standard  action of $U(2)$ on $\C^2$ with weights $\epsilon_1=[1,0], \epsilon_2=[0,1]$.
The moment cone for $T$ is thus $\R_{\geq 0}\epsilon_1\oplus \R_{\geq 0}\epsilon_2$.
The Weyl chamber is $\t^*_{\geq 0}=\{\xi_1 \epsilon_1+\xi_2 \epsilon_2; \xi_1\geq \xi_2\}$.
We have $\Phi_K(z)=z z^*$, if we consider $z\in \C^2$ as a map $\C^2 \to \C$. Thus $\Phi_K(z)$ is a Hermitian non negative matrix of rank $1$. We thus have $\Delta_K(\CH)=\R_{\geq 0}\epsilon_1$.
In particular, the cone $C_T(\CH)$ is solid in $\t^*$, while
$C_K(\CH)$ is not solid. $\Box$

\bigskip

Even when both $C_T(\CH)$ and $C_K(\CH
)$  are solid, they can be different.
\begin{example}  $C_T(\CH)$ for the $3$-qubits.\end{example}

Return for example to the case of $3$ qubits \ref{exa3qubits} and we keep the same notations.
The  weights of the action of $T$ in $\CH=\C^2\otimes \C^2\otimes \C^2$ on the weight vectors
$e_{\pm}\otimes e_{\pm}\otimes e_{\pm}$ with $\C^2=\C e_+\oplus \C e_-$
are $[[1,0],|1,0],[1,0]]$ and its permutation by the Weyl group $\Sigma_2\times \Sigma_2\times \Sigma_2$  of $K$.
In the parameters $\R^4$, the weights are
$[1,\epsilon_1,\epsilon_2,\epsilon_3]$ with $\epsilon_i\in\{1,-1\}$.
So we see that $C_T(\CH)$ is the cone over the cube
        $C$
        with vertices $[\epsilon_1,\epsilon_2,\epsilon_3]$ with $\epsilon_i\in \{-1,1\}$.
        The intersection of $C_T(\CH)$ with the positive Weyl chamber
        is the cone over $C$ intersected with the positive quadrant.
        This is again a cube with vertices
         $[\nu_1,\nu_2,\nu_3]$ with $\nu_i\in \{0,1\}.$

Figure \ref{CKT} shows the Kirwan cone $C_K(\CH)$ inside $C_T(\CH)\cap i\t^*_{\geq 0}$.
We see that the points $(1,1,0)$ (and permutations) are vertices of
$C_T(\CH)\cap i\t^*_{\geq 0}$, but they do not belong to $C_K(\CH)$.
The Kirwan polytope $C_K(\CH)$ is obtained by removing from $C_T(\CH)$  a "neighborhood" of these vertices, namely
the simplices with vertices
$(1,1,0), (1,1,1), (1,0,0),(0,1,0)$ (and  its permutations).

\begin{figure}[h]\centering
\includegraphics[height=1.5 in]{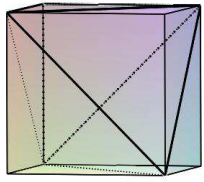}\caption{$K$-Kirwan polytope for three qubits  inside $T$-Kirwan polytope}
\label{CKT}
\end{figure}
$\Box$

\bigskip

 Assume (to simplify) that $K$ acts on $\CH$ with generic finite stabilizer. Then $C_K(\CH)$ is a cone in $i\t^*_{\geq 0}$ with non empty interior and   it follows from the formula
 for Duistermaat-Heckman measure that the  Kirwan cone $C_K(\CH)$ is the union of  the solid cones $\c_a\cap i\t^*_{\geq 0}$
contained in it.
As seen in the preceding example, it is not easy to understand which ones have this property.

We denote by $\partial_{-\alpha}f(y)=\frac{d}{dt} f(y-t\alpha)\vert_{t=0}$ the derivative  of a function $f$ on $i\t^*$ in the direction $-\alpha$.
Let $(\prod_{\alpha>0} \partial_{-\alpha})$ the product of derivatives with respect to all negative roots in $\Delta_\k$.
Consider the cone decomposition $C_T(\CH)=\cup_a \c_a$.

\begin{theorem}\label{theo:reductoT}
Assume $C_K(\CH)$ is a solid cone.

Let $C_T(\CH)=\cup_a \c_a$.
Consider $\c_a$ such that $\c_a\cap i\t^*_{\geq 0}\subset C_K(\CH)$.
\begin{itemize}
%$\bullet$
\item  On $\c_a\cap  i\t^*_{\geq 0}$
we have $$dh^{\CH}_K=(\prod_{\alpha>0} \partial_{-\alpha})\cdot d_{a}^{\Psi}$$
\item Let $p_a^{\Psi}$ be the quasi polynomial function on $\Lambda$ such that
$\CP_\Psi(\mu)=p_a^{\Psi}(\mu)$ on $\c_a.$
Then, on $\c_a\cap  \Lambda_{\geq 0}$,
 we have
$$m_K^{\CH}(\lambda)=\sum_{w\in \CW_\k} \epsilon(w) p_a^{\Psi}(\lambda+\rho_\k-w\rho_\k).$$
\end{itemize}
\end{theorem}

\begin{remark} \end{remark}
We see that the degree of  the polynomial function $dh^{\CH}_K$, and $m_K^{\CH}$ on each $\c_a$
is at most equal to the degree of the corresponding functions for $T$ minus the number $|\Delta_\k^+|$ of positive roots.
Our algorithm will indeed be easier for $K$ than for $T$.

  $\Box$
  \bigskip

\begin{proof}
Consider first the open subset $\tau_Q$  of
$\c_a\cap  i\t^*_{\geq 0}$, at distance $Q$ of its boundary.
We assume $Q$ greater than the norm of $\|\rho_\k-w\rho_\k\|$ for all $w\in \CW_\k$.
If we use  the formula $$m_K^{\CH}(\lambda)=\sum_{w\in \CW_\k} \epsilon(w) \CP_\Psi(\lambda+\rho_\k-w\rho_k)$$
and $\lambda$ is in $\tau_Q$, all points $\lambda+\rho_\k-w\rho_k$ are in $\c_a$ and we may replace $\CP_\Psi$ by the function $p_a^{\Psi}$.
We thus obtain on $\tau_Q$ the expression
$$m_K^{\CH}(\lambda)=\sum_{w\in \CW_\k} \epsilon(w) p_a^\Psi(\lambda+\rho_\k-w\rho_k).$$
However the function $\lambda\to\sum_{w\in \CW_\k} \epsilon(w) p_a^{\Psi}(\lambda+\rho_\k-w\rho_\k)$ on the right side of the above equation
 is a quasi polynomial function of $\lambda$ {\bf} on $\c_a\cap (i\t^*_{\geq 0})$.
By Meinrenken-Sjamaar, Theorem \ref {theo:multH}, we know that
$m_K^{\CH}(\lambda) $ coincide with a quasi polynomial function on the closed cone $\c_a\cap \Lambda_{\geq 0}$, under our assumption that
$\c_a\cap i\t^*_{\geq 0}\subset C_K(\CH)$.

Using Lemma \ref{lem:equaquasipoly}, we obtain the equality on $\c_a\cap \Lambda_{\geq 0}$.
\end{proof}

\subsection{Back to degrees on faces}

Assume that the cone $C_K(\CH)$ is solid.
We rewrite the formulae for the degrees $r_F$ on faces of the multiplicities, using $\Psi$.
 Let $F$ be a facet of $C_K(\CH)$ determined by an equation $X=0$ intersecting $i\t^*_{>0}$
(that is  $X$ is not proportional to a fundamental coroot $H_\alpha$).
Then we compute $\CH(X)$, the subspace of $\CH$ annihilated by $X$.
Its complex dimension is the cardinal of $\Psi_0$, where $\Psi_0$ is the sublist of $\Psi$ composed of the $\psi_j$ such that
$\psi_j(X)=0$.
The system $\Delta_0^+$ of positive roots for $\k(X)$ is the sublist of $\Delta_\k^+$ composed of the roots $\alpha_j$ such that
$\alpha_j(X)=0$. So we obtain the formula

\begin{lemma}
On the facet $F$, the degree of the multiplicity function is
$|\Psi_0|-|\Delta_0^+|- (\dim \t-1).$
\end{lemma}

\subsection{Tools for computations}\label{tools}
From now on, we focus on the computation of $m_K^{\CH}(\lambda).$

\subsubsection{Residue calculations: the philosophy}
Consider the case of  $n$ dimensional torus  $T=\{(u_1,u_2,\ldots, u_n); |u_i|=1\} $ acting on $\CH=\C^N$.
We denote by $u\to u^{\lambda}=\prod_{i} u_i^{\lambda_i}$ the character of $T$ indexed by $\lambda\in \Lambda=\Z^n.$
Assume that all weights $\psi_i$ of the action of $T$  on $\CH$ have non negative coordinates.
Almost by definition, the
computation of $m_T^{\CH}(\lambda)$
is the computation
of the Fourier series of the function
$\CS(u)=\prod_{i}\frac{1}{1- u^{\psi_i}}$, on $|u_i|<1$.
So we obtain, choosing $\epsilon_i$ small positive real numbers,

\begin{equation}\label{eq:dumbresidue}
m_T^{\CH}(\lambda)=\int_{|u_j|=\epsilon_j} u^{-\lambda}\prod_{j=1}^N\frac{1}{1- u^{\psi_j}}\prod_{j} \frac{ du_j}{2i \pi u_j}.
\end{equation}

We choose an order of integration, and the residue theorem in one variable allows to integrate the variable $u_1$, but we obtain a priori
 $N-1$ new integrals in $n-1$ variables, by considering all poles  in $u_1$ of the functions  $(1-u^{\phi_j})$.
Not all of these poles are in the interior of the circle $|u_1|=\epsilon_1$,
then we have to keep track of the branchings. At the end we obtain a very {\it {large}} number of products of one dimensional residues computations.

The essence of Jeffrey-Kirwan residue is to compute a priori the
paths which contributes.
In fact, it is easy to see that the function
$\prod_{j=1}^N\frac{1}{1- u^{\psi_j}}$ can be written as a sum of functions $f_\alpha$, where each $f_\alpha$ depend of
 $n$ independent variables. Here $\alpha$ varies in a large set $\CT$.
For example
$$\frac{1}{(1-u_1) (1-u_1 u_2)(1-u_2)}=\frac{1}{(1-u_1u_2)^2(1-u_1)}+\frac{u_2}{(1-u_1u_2)^2(1-u_2)}$$
and we can separately use the independent variables $(u_1, u_1 u_2)$ or $(u_2,u_1u_2)$.
In the residue computation  at $\lambda$, only a certain $\CT(\lambda)$ of these functions $f_\alpha$ will contribute.
So we compute a priori
what is the subset $\CT(\lambda)$ and we are reduced to product of residues in one variables.
This is the object of the computation of adapted Orlik Solomon bases.
We give now the technical details of the resulting formula.

\subsubsection{Definitions of regular elements and topes}

 Consider a finite set $F$ of non zero vectors  in a real vector space $E$.
\begin{definition}\label{defiadmissible}
We say that a hyperplane $H\subset E$ is $F$-admissible if $H$ is generated by elements of $F$.
We denote by $\CA(F)$ the set of admissible hyperplanes.
\end{definition}

Assume that $F$ generates a lattice $L$ in $E$.
Let $\sigma$ be a subset of $F$, generating $E$ as a vector space.
We say that $\sigma$ is a basis of $F$.
We denote by $d_\sigma$ the smallest integer such that
$d_\sigma L$ is contained in the lattice generated by $\sigma$.
\begin{definition}

Consider the set $periods(F)=\{d_\sigma\}$ where $\sigma$ runs over the basis of $F$.
  Let  $q(F)$ is the least common multiple of the integers  in $periods (F)$
We say that $q(F)$ is the index of the finite set $F$ (with respect to
$L$).
\end{definition}

\begin{definition}

Consider a finite set $\CF$ of hyperplanes in a real vector space $E$.
We say that $\xi\in E$ is $\CF$ regular, if $\xi$ is not on any hyperplane belonging to $\CF$.

Let $E_{reg\CF}$ be the set of $\CF$-regular elements.
A connected component $\tau$ of $E_{reg\CF}$ will be called a tope (with respect to $\CF$).
If $\xi$ is $\CF$-regular, we denote by $\tau(\xi)$ the unique tope containing $\xi$.
\end{definition}

Return to the situation where $\Psi\subset i\t^*$ is the set of weights of $T$ in $\CH$.
 When $\CF$ is the family $\CA(\Psi)$ of $\Psi$-admissible hyperplanes, we also say that $\xi$ is $\Psi$-regular instead of
 $\CF$-regular.
A tope for the family $\CA(\Psi)$ will be called a $\Psi$-tope.

\subsubsection{Iterated residues and Orlik-Solomon  bases}
A list ${\overrightarrow{\sigma}}=[\psi_{i_1},\psi_{i_2,}\ldots, \psi_{i_r}]$ of elements of $\Psi$ will be called a \textbf{ordered basis} if the
elements $\psi_{i_k}$  form a basis of $i\t^*$.
Let ${\overrightarrow{\mathfrak{B}}}(\Psi)$ be  the set of ordered bases.

\noindent An ordered base $=[\psi_{i_1},\psi_{i_2,}\ldots, \psi_{i_r}]$ is an \textbf{Orlik Solomon base}, {\it {$OS$}} in short, if   for each $1\leq l\leq r,$ there is no $j< i_l$ such that the elements $\psi_j,\psi_{i_l},\ldots, \psi_{i_r}$ are linearly dependent. Denote by $\mathcal{OS}(\Psi)$  the set of $OS$ bases.
\begin{definition}\label{OS}

If $\tau$ is a $\Psi$-tope,
  we denote by $\mathcal{OS}(\Psi,\tau)=\{\sigma \in \mathcal{OS}(\Psi), \ \tau \subset Cone(\sigma)$. Here $Cone(\sigma)=\sum \R_{\geq 0} \psi_{i_j}.$

\end{definition}

The set $\mathcal{OS}(\Psi,\tau)$ is called  the set of \textbf{OS adapted bases to $\tau$}.

We will show in Section \ref{computeOS} how to compute  adapted bases. For an account of the theory cf. \cite{BBCV}.

We now define the notion of iterated residue.
If $f$ is a meromorphic function in one variable $z$, consider its Laurent series $\sum_n a_n z^n$ at $z=0$.
The coefficient of $z^{-1}$ is denoted by $\Res_{z=0}f$.

Let $\overrightarrow{\sigma}=[\alpha_1,\alpha_2,\ldots, \alpha_r]$ be an ordered basis of $i\t^*$.
For $z\in \t_\C$, let $z_j=\langle z,\alpha_j\rangle$, and  express a meromorphic function $f(z)$ on $\t_\C$ with poles on a union
of hyperplanes  as a function  $f(z)=f(z_1,z_2,\ldots,
z_r),
$
(in particular $f$ may have poles on $z_j=0$).

Consider the associated iterated residue functional
defined by:
\begin{equation} \label{defres}  Res_{{\overrightarrow{\sigma}}}(f(z)):=
 \Res_{z_1=0}(\Res_{z_2=0}\cdots(\Res_{z_r=0}f(z_1,z_2,\ldots,z_r))\cdots).
 \end{equation}

\subsubsection{ Multiplicity and Duistermaat-Heckman  function for the torus $T$}
 Fix a cone of polynomiality $\c_a$, then $\c_a$ is the union of the closures of the topes $\tau$ contained in $\c_a$.

Topes were also called small chambers in \cite{BaldoniVergne}.

We might need several topes (small chambers) so that the union of their closures is $\c_a$.
Figure \ref{chambervshyper} shows an example for the chambers complex versus the topes complex for the cone associated to a root system of type $A_3.$

\begin{figure}[h]%[htpb]
\begin{center}
    \includegraphics[scale=0.5]{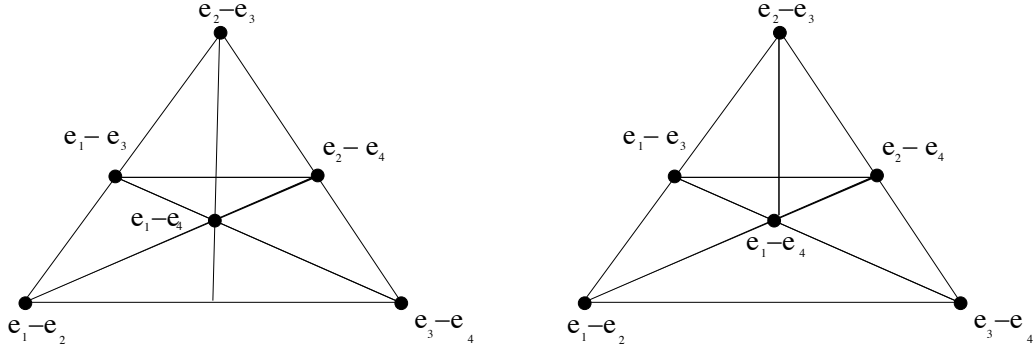}
\caption{8 topes (left) versus 7 chambers (right)}
\label{chambervshyper}
\end{center}
\end{figure}

Let $\xi\in i\t^*$.
Define the following function of $z\in \t_\C$.
 $$s_T^{\Psi}(\xi,z)=e^{\langle \xi,z\rangle }\frac{1}{\prod_{\psi\in \Psi} \langle\psi,z\rangle }.$$

The function $z\to s_T^{\Psi}(\xi,z)$
 has poles on the union of the hyperplanes $\langle\psi,z\rangle =0$.
An iterated residue of the function $z\to s_T^{\Psi}(\xi,z)$ depends  of $\xi$  through the Taylor series of the function
$e^{\langle \xi,z\rangle }$ at $z=0$.
So if we consider $\xi$ as a variable,  we obtain as a result of performing an iterated residue a polynomial function of $\xi$.

\begin{definition}

 Let $ \tau$  be a $\Psi$-tope.
Define
$$d_\tau^{\Psi}(\xi)=\sum_{\sigma\in \mathcal{OS}(\Psi,\tau)} Res_{{\sigma}}s_T^{\Psi}(\xi,z).$$

\end{definition}
Thus $d_\tau^{\Psi}$ is a polynomial function of $\xi$.

\begin{theorem}\label{volumeT}
Let $\tau \subset i\t^*$ be a $\Psi$-tope, and
 $\xi\in \overline\tau$.
 If $\tau$ is contained in $Cone(\Psi)$,
$$dh^{\CH}_T(\xi)=d_\tau^{\Psi}(\xi).$$
\end{theorem}

\bigskip

Let $\mu\in \Lambda$, and for  $z\in \t_\C$,
define :
$$S_T^{\Psi}(\mu,z)=e^{\langle\mu,z\rangle }\frac{1}{\prod_{\psi\in \Psi} (1-e^{-\langle\psi,z\rangle })}.$$

Let $\Gamma$ be the dual lattice to $\Lambda$. Let $q:=q(\Psi)$ be the index of $\Psi$.
So
if $\sigma$ is a basis in $i\t^*$ consisting of elements of $\Psi$ then
$q \Lambda \subset \sum _{\psi \in \sigma}\Z \psi .$

If $\gamma\in \Gamma$, and we apply an iterated residue to the function
$z\to S_T^{\Psi}(\mu, z+\frac{2i\pi}{q} \gamma)$, we obtain a quasi-polynomial function of $\mu$.
Indeed, it depends of $\mu$ through the Taylor series at $z=0$ of
$e^{\langle\mu,z+\frac{2i\pi}{q} \gamma\rangle}=e^{\langle\mu,\frac{2i\pi}{q} \gamma\rangle} e^{\langle\mu,z\rangle}$,
and $e^{\langle\mu,\frac{2i\pi}{q} \gamma\rangle}$ is a periodic function of $\mu$ of period $q$.

\begin{definition}\label{defmultT}
Let $ \tau$  be a $\Psi$-tope.
Define
$$p_\tau^\Psi(\mu)=\sum_{\sigma\in \mathcal{OS}(\Psi,\tau)}
\sum_{\gamma\in \Gamma/q\Gamma} Res_\sigma S_T^{\Psi}(\mu, z+\frac{2i\pi}{q} \gamma).$$
\end{definition}

\begin{theorem}\label{multT} (Szenes-Vergne)
 Let $\tau \subset i\t^*$ be a $\Psi$-tope and  $\mu\in \overline \tau$.
 If  $\tau$ is contained in $Cone(\Psi)$, then for any $\mu\in \overline\tau\cap \Lambda$,
$$m_T^{\CH}(\mu)=p_\tau^\Psi(\mu).$$
\end{theorem}

\begin{remark}\end{remark}
We denote by $u\to u^{\nu}$ the character of $T_\C$ associated to $\nu\in \Lambda$.
Consider the function $F(u)=\frac{1}{\prod_{\psi\in \Psi}(1-u^{\psi})}$.
It is clear that $m_T^{\CH}(\mu)$ is the Fourier coefficient of the function $F(u)$, expanded in the domain
$|u^{\psi}|<1$, for all $\psi\in \Psi$.
Thus $m_T^{\CH}(\mu)=\int_{|u|=r} u^{-\mu}F(u) \frac{du}{u}$.
Here $r$ is a small positive number.

Using the change of variables $u_i=e^{-z_i}$ ($z_i$ being the coordinates on $i\t$ associated to the basis $\sigma$),
an iterated residue
 $Res_{{\sigma}}s_T^{\Psi}(\mu,z)$
can be computed  as an iterated residue at $u=1\in T_\C$ of
the function
$\frac{u^{-\mu}}{\prod_{\psi\in \Psi}(1-u^{-\psi})}$.

Similarly the other terms  associated to $\gamma\in \Gamma/q\Gamma$ can be computed as iterated residues at some elements of finite order $q$ in $T_\C$.
Thus Szenes-Vergne theorem is a multi-dimensional residue theorem.
A detailed example will be given in \ref{sub:verysimple}.
$\Box$

\bigskip

\subsubsection{ Multiplicity and Duistermaat-Heckman function for $K$}
We now want to compute the functions $dh^{\CH}_K(\xi)$
and $m_K^{\CH}(\mu)$.

Let $\xi\in i\t^*$. Define the following function of $z\in \t_\C$:
 $$s_K^{\Psi}(\xi,z)=e^{\langle \xi,z\rangle }\frac{\prod_{\alpha\in \Delta_\k^+}\langle\alpha,z\rangle } {\prod_{\psi\in \Psi} \langle\psi,z\rangle }.$$

\begin{definition}\label{DHK}
Let $ \tau$  be a $\Psi$-tope.
Define
$$D_\tau^{\Psi}(\xi)=\sum_{\sigma\in \mathcal{OS}(\Psi,\tau)}  Res_{{\overrightarrow{\sigma}}} s_K^{\Psi}(\xi,z).$$
\end{definition}

\begin{theorem}\label{theo:DHK}
Let $\tau$  be a $\Psi$-tope
and let $\xi\in \overline \tau\cap i\t^*_{\geq 0}$.
If   $\tau\cap \t^*_{>0}$ is contained in the Kirwan cone $C_K(\CH)$, then
for any $\xi\in \overline\tau\cap \t^*_{\geq 0}$,
$$dh^{\CH}_K(\xi)=D_\tau^{\Psi}(\xi)$$
\end{theorem}

\begin{remark} \end{remark}

The function $D_\tau^{\Psi}$ is easier to compute that  $d_{\tau}^{\Psi}$
 as the orders of poles in $z$ of the function $s_K^{\Psi}(\xi,z)$ are smaller that those of $s_T^{\Psi}(\xi,z)$.

$\Box$
\bigskip

In particular, if $\xi\in i\t^*_{>0}$, and is $\Psi$-regular, we can check if
$\xi$ is in $C_K(\CH)$, by computing $D_\tau^{\Psi}(\xi)$ where $\tau$ is the tope containing $\xi.$

We can now state the result to compute $m_K^{\CH}(\lambda).$

 Let $\lambda\in \Lambda_{\geq 0}$,
 and for $z\in \t_\C$ define:$$S_K^{\Psi}(\lambda,z)=e^{\langle\lambda,z\rangle }\frac{\prod_{\alpha\in \Delta_\k^+}(1-e^{\langle\alpha,z\rangle })}{\prod_{\psi\in \Psi} (1-e^{-\langle\psi,z\rangle })}.$$

\begin{definition}
Let $\tau$  be a tope.
Define
$$P_{\tau}(\lambda)=
\sum_{\sigma\in \mathcal{OS}(\Psi,\tau)}
\sum_{\gamma\in \Gamma/q\Gamma} Res_{{\overrightarrow{\sigma}}} S_K^{\Psi}(\lambda, z+\frac{2i\pi}{q}  \gamma).$$

\end{definition}

\begin{theorem}\label{theo:multK}
Let $\tau$ be a tope and let $\lambda \in \Lambda_{\geq 0}$.
Then
\begin{enumerate}
\item
If $\lambda \in \tau$, then
$m_K^{\CH}(\lambda)=P_{\tau}(\lambda).$

\item  if $\lambda \in \overline{\tau}$
and $\tau \cap i\t^*_{\geq 0}\subset C_K(\CH)$ %,
then
$m_K^{\CH}(\lambda)=P_{\tau}(\lambda)$.
\end{enumerate}
\end{theorem}

\begin{proof}
Consider the quasi-polynomial function, def.\ref{defmultT}:
 $$p_\tau^{\Psi}(\mu)=\sum_{\sigma\in \mathcal{OS}(\Psi,\tau)}
\sum_{\gamma\in \Gamma/q\Gamma} Res_\sigma S_T^{\Psi}(\mu, z+\frac{2i\pi}{q} \gamma).$$
The function $\CP_\Psi$ is given on $\tau$
by the quasi polynomial formula, Theorem \ref{multT}
$$\CP_\Psi(\mu)=p_\tau^{\Psi}(\mu).$$
We now use Theorem \ref{theo:reductoT}  which asserts that
$m_K^{\CH}$ is given on the closure of $\tau$ by
$\sum_w\epsilon(w) p_\tau^{\Psi}(\mu+\rho_\k-w\rho_k).$
This is
$$\sum_{w \in \CW_\k} \epsilon(w )\sum_{\sigma\in \mathcal{OS}(\Psi,\tau)}
\sum_{\gamma\in \Gamma/q\Gamma}  \Res_{{\overrightarrow{\sigma}}} S_T^{\Psi}(\mu -w(\rho_\k)+\rho_\k, z+\frac{2i\pi}{q} \gamma)=$$
$$\sum_{\sigma\in \mathcal{OS}(\Psi,\tau)}
\sum_{\gamma\in \Gamma/q\Gamma} \Res_{{\overrightarrow{\sigma}}}\sum_w\epsilon(w) e^{\la-w(\rho_\k)+\rho_\k,z+\frac{2i\pi}{q} \gamma\ra }\frac{e^{\la \mu,z+\frac{2i\pi}{q} \gamma\ra }}{\prod_{\psi \in \Psi}(1-e^{\la -\psi,z+\frac{2i\pi}{q} \gamma\ra })}=$$
$$\sum_{\sigma\in \mathcal{OS}(\Psi,\tau)}\sum_{\gamma\in \Gamma/q\Gamma} Res_{{\overrightarrow{\sigma}}} S_K^{\Psi}(\mu, z+\frac{2i\pi}{q}  \gamma).$$
The last equality follows from the denominator formula $\sum_w\epsilon(w) e^{\rho_\k-w\rho_\k}=\prod_{\alpha>0} (1-e^{\alpha})$.
Thus $m_K^{\CH}(\mu)=P_\tau^{\Psi}(\mu)$ when $\mu$ is in the closure of $\tau$.
\end{proof}

\subsubsection{$OS$ bases}\label{computeOS}
Let $v$ be a $\Psi$-regular vector. Let $n=\dim \t$.
We recall here the algorithmic method based on the method of Maximal Nested Sets (MNS) of De Concini-Procesi, \cite{DCP}, and developed by Baldoni-Beck-Cochet-Vergne in \cite{BBCV} to compute $\mathcal{OS}(\Psi,\tau(v))$.
 Recall that $\mathcal{OS}(\Psi,\tau(v))$ depends only of the tope $\tau(v)$ where $v$ belongs.
We first need to compute the set $\CA$ of admissible
hyperplanes generated by elements of $\Psi$ (we do not have a good method to perform this step).
For each hyperplane, we choose an element $H\in \t$, so that our hyperplane is $H^{\perp}$.

Anyway, assume that we have determined this set $\CA$ of admissible hyperplanes.
Then we order $\Psi$.
We compute recursively a set of OS basis for  $\Psi$  by the following method.
\begin{itemize}
\item We first choose $\psi_0$ the lowest element of $\Psi$, and an hyperplane $H_1$ where $\psi_0$ does not belong:
$\la H_1,\psi_0\ra \neq 0$. Let    $\Psi_1=\Psi\cap H_1^{\perp}$ and $\CA_1=\{H^{\perp}\in \CA , \la H,\psi_0\ra =0\}$, that is $\CA_1$ is the set of hyperplanes in $\CA$ that contain $\psi_0.$
\item We choose $\psi_1$ the lowest element in $\Psi_1$ and $H_2$ in $\CA_1$ , such that $\psi_1$ does not belong to $H_2$,
and such that  $H_1^{\perp}\cap H_2^{\perp}$ is spanned by
$\Psi\cap H_1^{\perp}\cap H_2^{\perp}$.

 \noindent In  practice, we verify  this condition only at the end, but we verify
 that  the set  $\Psi \cap H_1^{\perp}\cap H_2^{\perp}$ has at least $n-2$ elements.

 \noindent We can now define $\Psi_2= H_1^{\perp}\cap H_2^{\perp}\cap \Psi$ and $\CA_2=\{H^{\perp}\in \CA _1, \la H,\psi_1\ra =0\}$, that is $\CA_2$ is the set of hyperplanes in $\CA$ that contain $\psi_0, \psi_1.$
 \item The inductive step is the following:

 \noindent suppose we have constructed  a list  $[\alpha_0,\alpha_1,\ldots, \alpha_s] \subset \Psi$ and a list $[H_1,H_2,\ldots, H_{s+1}]$ of hyperplanes with the property that
 \begin{itemize}
 \item $\alpha_i$ is the lowest element of $\Psi\cap H_1^{\perp}\cap \cdots \cap H_i^{\perp}$.
 \item $\alpha_j \notin H_{j+1}$ and $H^{\perp}_{j+1} \in \CA_{j}$ where $\CA_{j}=\{H^{\perp} \in \CA,
 \la H,\psi_t\ra =0, t=0\cdots j-1\}.$

 \noindent Thus  $[\alpha_0,\alpha_1,\ldots, \alpha_{j-1}]\in H^{\perp}_{j+1}$ and $\alpha_j \notin H_{j+1}.$
 \item $\Psi \cap H_1^{\perp}\cap H_2^{\perp}\cap\cdots\cap H_s^{\perp} \cap H_{s+1}^{\perp}$ has $\geq n-s-1$ elements.
\end{itemize}
\item The next step is as follows:
\begin{itemize}
\item  we choose $\alpha_{s+1}$ the lowest element in $\Psi \cap H_1^{\perp}\cap H_2^{\perp}\cap\cdots\cap H_s^{\perp}\cap H_{s+1}^{\perp}$
\item We choose an hyperplane $H_{s+2}$   containing $[\alpha_0,\alpha_1,\ldots, \alpha_{s}]$, i.e  $H_{s+2} \in \CA_{s+1}, $ but not $\alpha_{s+1}$ and such that
$\Psi \cap H_1^{\perp}\cap H_2^{\perp}\cap\cdots\cap H_s^{\perp} \cap H_{s+2}^{\perp}$ has $\geq n-s-2$ elements.
\end{itemize}
\item we then continue with $[H_1,H_2,\ldots, H_{s+1},H_{s+2}]$
and $[\alpha_0,\alpha_1,\ldots, \alpha_{s+1},\alpha_{s+2}]$
\item At the end, we obtain $n$ elements
$[\alpha_0,\alpha_1,\ldots, \alpha_{n-1}].$

It may occur that we cannot go to the end. This means that at this step, the space
 $H_1^{\perp}\cap H_2^{\perp}\cap \cdots \cap H_k^{\perp}$ was not generated by its intersection with $\Psi$.
However, if we arrive to the end, we obtain a basis, and by \cite{DCP} we obtain a set of OS basis for $\Psi$.
\end{itemize}

We can refine this method to obtain directly the set of OS adapted basis for  a regular vector $v$.
The algorithm we just explained is modified easily.
At each step when we look for  the hyperplane $H_s$  containing $[\alpha_0,\alpha_1,\ldots, \alpha_{s-2}]$, but not $\alpha_{s-1}$,
 we also impose on $H_s$ the condition that
the vectors  $v$ and $\alpha_{s-1}$ lie on the same side of the hyperplane $H_s$.
In this way, we  obtain $\mathcal{OS}(\Psi,\tau(v))$. If $v$ is not in the cone $C(\Psi)$,
 then the algorithm  returns the empty set for $\mathcal{OS}(\Psi,\tau(v)).$

Remark that if we know the equations of the cone $C(\Psi)$, then we could first check if $v$ is in the cone,  before computing $\mathcal{OS}(\Psi,\tau(v))$ to shorten the procedure.
However, we do not use this preliminary step.

\subsubsection{Dilated coefficients, Hilbert series}

Once we know   $\mathcal{OS}(\Psi,\tau)$ the calculation  is not more difficult to  do with symbolic variable $\mu$, and we obtain the periodic polynomial which coincides with $m_K^{\CH}(\mu)$ on $\tau\cap i\t^*_{\geq 0}$
(and the closure if $\tau\cap i\t^*_{\geq 0}$ is contained in $C_K(\CH)$) .

In particular, to compute the function $m_K^{\CH}(k \mu)$ on the line $\N \mu$  as a periodic polynomial function
of $k$ is not more difficult than to compute the numerical value $m_K^{\CH}(\mu).$

A particular interesting example is the computation of Hilbert series.
Assume that $\k=\z \oplus [\k,\k]$, where the center $\z=\R J$  of $\k$ acts by the homothety. Consider $\chi\in \Lambda$ such that
$\chi(iJ)=1$ and $\chi=0$ on $i(\t\cap [\k,\k])$.
Then
$m_K^{\CH}(k\chi)=\dim [S^k(\CH)]^{[K,K]}$.
So
the series $R(t)=\sum_{k=0}^{\infty} m_K^{\CH}(k\chi) t^k$ is the Hilbert series of the ring of invariants polynomials under the action of
$[K_\C,K_\C]$. It is of the form
$\frac{P(t)}{\prod_{j=1}^N (1-t^{a_j})}.$

 The degree of the quasi-polynomial function $k\to m^K(\CH)(k\chi)$ as well as its set of periods gives some information on
 the number of factors and the coefficients $a_j$ in this Hilbert series.

\subsubsection{ Advantages and Difficulties of the method}\label{difficult}

Since one of the objective of this paper is to describe an efficient algorithm to compute multiplicities, at least for low dimension, let's look at the weak and good  points in implementing the above formulae.

$\bullet$ The computation of the set of admissible hyperplanes for a system $\Psi$.
For example, if we consider the space of $6$ qubits, a brute force computation (taking any subset with $6$ elements of $2^6$ elements
 would lead to $7624512$ computations).
We do not know an efficient algorithm for computing the set $\CA(\Psi)$. Moreover, our system $\Psi$ has symmetries coming from the Weyl group action. So it would be good to have an algorithm computing $\CA(\Psi)$ up to Weyl group action.
When $\Psi$ is a root system of rank $r$, this set, up to Weyl group action, is just the set $\{h_i^{\perp},i=1,\ldots,r\}$ where $h_i$
are the fundamental coweights. In the Kronecker examples, we do not know the set
 $\CA(\Psi)$, although the system $\Psi$ is quite simple, see Formula \ref{restricted} inside Example  \ref{exarestricted}.

$\bullet$ The residues calculation for the function $S_K^{\Psi}(\mu,z)$. We do this by power series expansion, and
it leads to multiply polynomials of larger and larger degree.

$\bullet$ The computation of the set $\mathcal{OS}(\Psi,\tau)$.
Once this set is computed, residues contributions are independent of each other. So an advantage of the iterated residue method is that each individual residue computation is easy and does not use much memory.

$\bullet$ The index $q$ of the set $\Psi$. The fact is that for arbitrary systems $\Psi$ with $N$-vectors,
this index can be  large. So Szenes-Vergne formula does not provide a polynomial time algorithm (the dimension of $\CH$ being fixed).

For example, in the case of the knapsack, (Example  \ref{knapsack}), when the $A_i$ are relatively prime, the index $q$ is $A_1\cdots A_n$.
Thus, if $q$ is large,  the function
 $$\sum_{\gamma\in \Gamma/q\Gamma}  S_T^{\Psi}(\mu, z+\frac{2i\pi}{q} \gamma)$$
should be computed in polynomial time  using Barvinok determination of generating functions of cones \cite{Bar}.
This is the method we used in \cite{BBDDKV} to give an efficient computation when $n$ is fixed.
However, for our computation of Kronecker examples,   we just  used the summation over $L^*/qL^*$, as $q$ was relatively small.

\bigskip

Let us explain at this point the method followed by Christand-Doran-Walter \cite{C-D-W}. Their   method
(of polynomial complexity, when the dimension of $\CH$ is fixed)
 computes the numeric multiplicity by the following method.
As we said, the multiplicity function $m_T^{\CH}(\lambda)$ for $T$ is the number of points in the polytope $\Pi_\Psi(\lambda)$.
Thus  \cite{C-D-W} uses the "Barvinok algorithm", as implemented in \cite{Verdo}, to compute
$m_T^{\CH}(\lambda)$.  This computation would  also be possible to do using  the Latte  software package \cite{allourfriends}.
So the above difficulty of the possibly very large index $q$ is resolved (in polynomial time) by the power of Barvinok signed decompositions.
Then they use the relation between $m_K^{\CH}$ and $m_T^{\CH}$ to compute the multiplicity under $K$.
Remark that this last computation can be done on parallel computing. However the number of elements $w$ becomes very large.

 We close this  subsection with some questions.

{\bf Questions}

\noindent Here are the problems that we would like to  have some partial answers in interesting examples:
\begin{itemize}\item Can we understand the image of the moment map, and describe the Kirwan polyhedron by inequalities?
\item  Can we compute the Duistermaat-Heckman measure?
\item  Can we compute the dilated multiplicity $k\to m_K^\CH(k\lambda)=\sum_i^R c_i(k) k^i$,
or at least  say something on the periodicity of the coefficients and the degree  of the function $k\to m_K^\CH(k\lambda)$ ?
In particular when $\lambda$ is a one dimensional representation of $K$, or other interesting
$\lambda$.

\end{itemize}

\subsubsection{A very simple example}\label{sub:verysimple}

Return to Example \ref{ex:Walter}.
We close this part in computing in two ways the multiplicity $m_{T_K}^{\CH}(\lambda)$ in this very simple example.

$\bullet$ The straightforward method of computing the expansion of a Fourier series.
The one dimensional residue formula is used repeatedly  to replace the large poles at $u=0$ in poles on $|u|=1$.
 The branchings appearing  in this straightforward
becomes  soon quite complicated, if the rank of $K$ as well as the dimension of $\CH$ increases.
This is the method followed by \cite{Xin}, with an efficient way to keep track of the branchings.

$\bullet$ The Jeffrey-Kirwan residue method. It gives a (possibly large) number of  residues
 at $z=0$,  or equivalently, in
exponential coordinates $u=\exp z$, at $u=1$, or some finite order elements in $T$.
The  computation for each iterated residue  are independent of each other.

$\bullet$ {\bf The straightforward method.}

We consider the torus $T$ of $U(2)$ parameterized as $\{(u_1,u_2); |u_1|=1; |u_2|=1\}$
and the action of $T$ on $\CH$ with weights $[2,0],[1,1],[0,2]$.
The character of the representation of $T$ in $Sym(\CH)$ is obtained by computing the Fourier expansion of
$$F(u_1,u_2)=\frac{1}{(1-u_1^2)(1-u_1u_2)(1-u_2^2)}$$ for $u_1,u_2$ of modulus strictly less then than $1$.
Let $\lambda=\lambda_1\epsilon_1+\lambda_2\epsilon_2$, with $\lambda_1$, $\lambda_2$ two non negative integers.
Thus, taking for example $r_1=r_2=\frac{1}{2}$, we obtain
the straightforward formula
$$m_T^{\CH}(\lambda)=\int\int_{|u_1|=r_1,|u_2|=r_2}\frac{u_1^{-\lambda_1} u_2^{-\lambda_2}}{(1-u_1^2)(1-u_1u_2)(1-u_2^2)}\frac{du_1}{2i\pi u_1} \frac{du_2}{2i\pi u_2}.$$

Let us fix $u_1$, and integrate on $|u_2|=r_2$.
As $\lambda_2\geq 0$, the poles of the integrand  on $|u_2|\leq r_2$  is the pole at $u_2=0$.
The other poles are $u_2=1, u_1=-1, u_2=1/u_1$, and there are no poles at $\infty$.
 Using the one dimensional residue theorem, we obtain
$m_T^{\CH}(\lambda)=A+B+C$
with $$A=-\frac{1}{2}\int_{|u_1|=r_1}\frac{ u_1^{-\lambda_1}}{(1-u_1)(1-u_1^2)}\frac{du_1}{2i\pi u_1},$$
$$B=-\frac{(-1)^{\lambda_2}}{2}\int_{|u_1|=r_1}\frac{ u_1^{-\lambda_1}}{(1-u_1)(1-u_1^2)} \frac{du_1}{2i\pi u_1},$$
$$C=\int_{|u_1|=r_1}\frac{u_1^{\lambda_2-\lambda_1+2}}{(1-u_1^2)^2} \frac{du_1}{2i\pi u_1}.$$

We now integrate in $u_1$. However, at this step, the computation is different if $\lambda_2\geq \lambda_1$ or not.
Indeed in the case where $\lambda_2\geq \lambda_1$, the integral $C$ will be equal to $0$, as the integrand has no poles inside
$|u_1|\leq \frac{1}{2}$.
So let us assume this is the case.
We then pursue in computing similarly the integrals $A$,$B$ by the one dimensional residue formula, and we obtain
by summing the residues at $u_1=1, u_1=-1$ of the integrands expressions in $A$,$B$, a sum of $4$ terms adding up to
$$\frac{1}{4}(1+(-1)^{\lambda_1+\lambda_2})\lambda_2+ (1+(-1)^{\lambda_1+\lambda_2})\frac{3}{8}+\frac{1}{8}((-1)^{\lambda_1}+(-1)^{\lambda_2}),$$
an expression vanishing if $\lambda_1+\lambda_2$ is odd, as it should, and coinciding with
the given expression
$$m_{T_K}^{\CH}(\lambda)=\frac{1}{2}\lambda_1+\frac{3}{4}+(-1)^{\lambda_1}\frac{1}{4}.$$
on the cone of polynomiality $\c_2$

$\bullet$
Let us now employ  Theorem \ref{multT},  using Jeffrey Kirwan residues.

We consider  the same case
$\lambda_2\geq \lambda_1$.

Consider the tope $\tau_2=\{\lambda; \lambda_1>\lambda_2>0\}$, the interior of $\c_2$.
If we compute $OS(\Psi,\tau_2)$ for the order   $\Psi=[[2,0],[1,1],[0,2]]$,  there is only one adapted basis
$\sigma=[[2,0],[0,2]].$
The index $q$ is equal to $2$ (in the lattice $\Lambda_K$ generated by $\Psi$).
The function $S_T(\lambda,z)$ is equal to
$$\frac{e^{\lambda_1 z_1+\lambda_2 z_2}}{(1-e^{-2z_1})(1-e^{-2z_2})(1-e^{-z_1-z_2})}.$$

A representative of $\Gamma_K/2\Gamma_K$ is $G=[-1/2,0]$, and we obtain
$$\sum_{\gamma\in \Gamma_K/2\Gamma_K}S_T(\lambda,z+i\pi \gamma)=S_1+S_2$$ with
$$S_1=\frac{e^{\lambda_1 z_1+\lambda_2 z_2}}{(1-e^{-2z_1})(1-e^{-2z_2})(1-e^{-z_1-z_2})},$$
$$S_2=(-1)^{\lambda_1}\frac{e^{\lambda_1 z_1+\lambda_2 z_2}}{(1-e^{-2z_1})(1-e^{-2z_2})(1+e^{-z_1-z_2})}.$$

The  iterated residue computation  $Res_{z_1=0} Res_{z_2=0}S_i$ is straightforward and we obtain
$\frac{1}{2}\lambda_1+\frac{3}{4}$ from $S_1$, and $\frac{1}{4}(-1)^{\lambda_1}$ from the second term $S_2$.

It is also worthwhile to remark in this example that
the iterated residue computation depends of the order. The reverse order
$Res_{z_2=0} Res_{z_1=0}S_i$ would have given the formula for the tope
 $\tau_1=\{(\lambda_1,\lambda_2), \lambda_1>\lambda_2>0\}.$

\section{Branching Rules}\label{Branching Rules}
In this section, we will write  a quasi polynomial formula for  the branching coefficients.   This will allow, in combination with the Cauchy formula, to express the Kronecker coefficients in a different way that has at least the advantage of reducing the number of variables in the setting, Section \ref{section:kro}.
\subsection{The Branching Cone}
Consider a pair $K\subset G$ of two compact connected Lie groups, with Lie algebras $\k,\g$ respectively.
Let $\pi:\g^*\to \k^*$ be the projection.

We consider the following action of $G\times K$ on $G$ : $(g,k)\cdot a= g a k^{-1}$.
The manifold  $T^*G$ (the cotangent bundle of  $G$) is a $G\times K$ Hamiltonian manifold.
The geometric quantization of $T^*G$ "is" the space $L^2(G)$. This statement can be justified, but we will not
do it  here.

Let us define $V=R(G)$ to be the subspace of $C^{\infty}(G)$ generated by the coefficients
$\langle g u_1,u_2\rangle$ of finite dimensional representations of $G$
(by Peter-Weyl theorem, the space $L^2(G)$ is the Hilbert completion of $V$).

Assume $G$,$K$ connected, and let $T_G$, $T_K$ be  maximal tori of $G$,$K$.  We may assume, and we do so, that $T_K\subset T_G$.
We choose Cartan subalgebras $\t_\g$, $\t_\k$, Weyl chambers  $i\t^*_{\g,\geq 0}$, $i\t^*_{\k,\geq 0}$,
and  we denote the corresponding cones of dominant weights by $\Lambda_{G,\geq 0}$, $\Lambda_{K,\geq 0}$.
We denote by $i\t^*_{\g,\k,\geq 0}$ the sum $i\t^*_{\g,\geq 0}\oplus i\t^*_{\k,\geq 0}$ of the closed positive Weyl chambers relatives to $G,K$, and by  $i\t^*_{\g,\k,> 0}$ its interior.
We may also choose compatible root systems on $K$, $G$:
If $\lambda$ is dominant for $G$, then the restriction of $\lambda$ to $i\t_{\k}$ is dominant.

For $\lambda\in \Lambda_{G,\geq 0}$ (resp.  $\mu\in \Lambda_{K,\geq 0}$),
denote by $V_{\lambda}^G$ (resp. $V_{\mu}^K$) the irreducible representation of $G$ (resp. $K$) of highest weight
 $\lambda$ (resp. $\mu$).

Under the action of $G\times G$,
$$V=\oplus_{\lambda\in \Lambda_{G,\geq 0}} V_\lambda^G\otimes V_{\lambda^*}^G.$$
So, under the action of $G\times K$,
$$V=\oplus_{\lambda,\mu} m_{G,K}(\lambda,\mu) V_\lambda^G\otimes V_{\mu^*}^K$$
($\lambda$ varies in $\Lambda_{G,\geq 0}$, and  $\mu$ in  $\Lambda_{K,\geq 0}$)
where $m_{G,K}(\lambda,\mu)$ is the multiplicity of the representation $\mu$ in the restriction of $V_\lambda^G$ to $K$; it is also called the branching coefficient.

Let us write coordinates for $T^*G$.
We identify the tangent bundle $T G$ to $G\times\g$ through the left translations: to
$(g,X)\in G\times\g$ we associate $\frac{d}{dt} e^{-tX}g\vert_{t=0}\in T_g G$.
Thus, $T^*G$ is identified to $G\times \g^*$, and  the  moment map relative to the $G\times K$-action is the
map $\Phi_{G} \oplus \Phi_K : T^*G \to \g^*\oplus \k^*$ defined by
$(g,\xi)\to (\xi,-\pi(g^{-1}\cdot\xi)).$
In order to relate the moment map to the branching coefficient $m_{G,K}(\lambda,\mu)$,
we use the slightly modified map: $\Phi:T^*G \to i\g^*\oplus i\k^*$
given by $$\Phi(g,\xi)\to (i\xi,i\pi(g^{-1}\cdot\xi)).$$
Here $g\in G$, and $\xi\in \g^*$.

Let $$C_{G,K}(T^*G)=\Phi(T^*G)\cap i\t^*_{\g,\k,\geq 0}$$
be the Kirwan cone associated to $\Phi$:
$$
C_{G,K}(T^*G)=\left\{(\xi,\eta)\in i\t^*_{\g,\geq 0}\times i\t^*_{\k,\geq 0}; \eta \in \pi(G\cdot \xi)\right\}.
$$

The set $C_{G,K}(T^*G)$ is a polyhedral cone in $i\t^*_{\g,\k,\geq 0}$, and is related to the branching coefficients through the following basic result.
\begin{proposition}\label{Conesupport}
We have $m_{G,K}(\lambda,\mu)=0$ if $(\lambda,\mu)\notin C_{G,K}(T^*G)$.

Conversely, if $(\lambda,\mu)$ is a pair of dominant weights contained in
$C_{G,K}(T^*G)$, there exists an integer $k>0$ such that
$m_{G,K}(k\lambda,k\mu)$  is  non zero.
\end{proposition}

Thus the support of the function $m_{G,K}(\lambda,\mu)$ is contained in the Kirwan polyhedron
$C_{G,K}(T^*G)$ and its asymptotic support is exactly the cone $C_{G,K}(T^*G)$.

Remark that if $G=K$, the cone $C_{G,K}(T^*G)$ is just the diagonal $\{(\xi,\xi), \xi\in i\t^*_{\g,\geq 0}\}$ in $i\t^*_{\g,\g,\geq 0}$.
However, {\bf we assume from now on} that no nonzero ideal of $\k$ is an ideal of $\g$ (this condition excludes the preceding case).
It implies the following result (Duflo, private communication).
\begin{lemma}
The polytope  $C_{G,K}(T^*G)$ is solid.
\end{lemma}

Let $\pi_\lambda: G\lambda \to i\k^*$ be the restriction of $\pi$ to the orbit $G\lambda$.

\begin{definition}
Define the reduced space $M_{red,\mu}^\lambda=\pi_\lambda^{-1}(K\mu)/K$.
\end{definition}
Remark that  $\Phi^{-1}(G\lambda, K\mu)/(G\times K)$ is isomorphic to
 $M_{red,\mu}^\lambda$ and that the reduced space $M_{red,\mu}^\lambda$ is non empty if and only if
$(\lambda,\mu)\in C_{G,K}(T^*G)$.

The $[Q,R]=0$ theorem of Meinrenken-Sjamaar relates
$m_{G,K}(\lambda,\mu)$ to the Riemann-Roch number (suitably defined) of the space $M_{red,\mu}^\lambda$.
As a consequence, we obtain the following theorem.

\begin{theorem}\label{theo:mGK}
There exists a decomposition of the cone
$C_{G,K}(T^*G)=\cup_a \c_a$,  in {\bf closed solid  polyhedral cones} $\c_a$
such that the following property holds.

 For each $a$, there exists a  non zero quasi-polynomial function
$p_a$ on $\Lambda_G\oplus \Lambda_K$  such that
$$ m_{G,K}(\lambda,\mu)=p_a(\lambda,\mu)$$ if
$(\lambda,\mu)\in \c_a\cap (\Lambda_G\oplus \Lambda_K)$.
\end{theorem}

In particular, for any  pair  $(\lambda,\mu)$  of dominant weights contained in
$C_{G,K}(T^*G)$,
the function $k\to m_{G,K}(k\lambda,k\mu)$ is of the form:
$m_{G,K}(k\lambda,k\mu)=\sum_{i=0}^{N}E_i(k) k^i$
where $E_i(k)$ are periodic functions of $k$.
This formula is valid {\bf for any } $k\geq 0$, and in particular $E_0(0)=1$.

An  interesting example is the case of $K$ embedded in $G=K\times K$ by the diagonal.
Recall that $c_{\lambda,\mu}^\nu$ is the multiplicity of $V_\nu$ in the tensor product $V_\lambda\otimes V_\mu$.
Thus we obtain
\begin{corollary}
Let $K$ embedded in $G=K\times K$ by the diagonal.
If $(\lambda,\mu,\nu)$ is in the $C_{G,K}(T^*G)$, the dilated Littlewood-Richarson coefficient
$k\to  c_{k\lambda,k\mu}^{k\nu}$ is a quasi-polynomial function of $k\in \{0,1,2,\ldots\}$.
\end{corollary}

To describe the cone $C_{G,K}(T^*G)$ is difficult,
and has been the object of numerous works, notably
Berenstein-Sjamaar, Belkale-Kumar,  Kumar,  Ressayre.
We refer to the survey article, \cite{Brion-bourbaki}. The complete description of the multiplicity function $m_{G,K}$, in particular the decomposition of  $C_{G,K}(T^*G)$ in $\cup_a \c_a$    is even more so.
However, we will give an algorithm where, given as input $(\lambda,\mu)$, the output is the dilated function
$k\to m_{G,K}(k\lambda,k\mu)$. In particular, we can  test if the point $(\lambda,\mu)$ is in the cone
$C_{G,K}(T^*G)$ or not, according if the output is not zero or zero.

Consider the set $\Psi\subset i \t_\k^*$ of non zero restrictions of the roots $\Delta_\g^+$ to $i\t_\k$.
We say that $\Psi$ is the list of restricted roots.

Recall that an hyperplane in $i\t_\k^*$ is $\Psi$-admissible if it is spanned by
elements of $\Psi$.  The set $\CA=\CA(\Psi)$ of admissible hyperplanes is finite.
For $H\in \CA$, consider  $X\in \t_\k$ such that $H=X^{\perp}$.
Let $\CW_\g$ be the Weyl group of $G$. If $X\in \t_\k\subset \t_\g$,
consider $wX\in \t_\g$ and  the hyperplane
$$H(w)=\{(\xi,\nu) \in i\t_\g^*\oplus i\t_\k^*; \langle \xi,wX\rangle -\langle \nu,X\rangle =0\}.$$
We obtain a finite set of hyperplanes $\CF$ in $i\t_\g^*\oplus i\t_\k^*$.

Consider a connected component $\tau$ of  the complement of the union of the  hyperplanes $H(w)$, where $H$ varies over admissible hyperplanes in $\t_\k^*$, and $w$ in the Weyl group of $G$, in other words  $\tau$ is a tope for
the system of hyperplanes $\CF$.
 So $\tau$ is an open conic  subset of $i\t_\g^*\oplus i\t_\k^*$.
Thus, if  $(\xi,\nu)\in \tau$,
for any admissible hyperplane $H\in \CA$ with equation $X$, and any $w\in \CW_\g$, we have
\begin{equation}
\langle \xi,wX\rangle -\langle \nu,X\rangle \neq 0.
\end{equation}

The following proposition follows from the description of the Duistermaat-Heckman measure \cite{Heck}.

\begin{proposition}
The facets of the cones $\c_a$ generates hyperplanes belonging to the family $\CF$.
\end{proposition}

Thus the following lemma follows, and will be useful.
\begin{lemma}
\begin{itemize}Fix a cone $\c_a$.
\item  If $\tau$ is a tope, then
$\tau\cap i\t^*_{\g,\k,\geq 0}$ is
either contained in $\c_a$, or
is disjoint from $\c_a$.
\item  The closed cone $\c_a$ is the union of
the closures of the sets $\tau\cap i\t^*_{\g,\k,\geq 0}$
 over the $\tau$ such that $\tau\cap \c_a$ is non empty.
 \end{itemize}
\end{lemma}

Remark that there might be several  topes $\tau$ needed to obtain $\c_a$.

We rephrase Theorem \ref{theo:mGK} as follows.
\begin{proposition}[``Continuity property" of $m_{G,K}$]
\begin{itemize}
\item If $\tau$ is a tope, the function $(\lambda,\mu)\to m_{G,K}(\lambda,\mu)$
is given by a quasi polynomial function $n_\tau$ on
$\tau\cap \Lambda_{G,K,\geq 0}.$
\item If $\tau\cap i\t^*_{\g,\k,\geq 0}$
 is contained in  the cone $C_{G\times K}(T^*G)$,
 and if  $(\lambda,\mu)\in \overline \tau \cap \Lambda_{G,K,\geq 0}$, then
 $$m_{G,K}(\lambda,\mu)=n_\tau(\lambda,\mu).$$
\end{itemize}
 \end{proposition}

\bigskip

We finally state a result on the behavior  of the function $m_{G,K}(\lambda,\mu)$ on the boundary of the
polyhedral cone $C_{G\times K}(T^*G)$.

If $X\in i\t_\k$, we have an injection $K(X)\subset G(X)$.
 Let $X\in i\t_\k$, $H=X^{\perp}$, and  $w\in W$  such that
$\langle \xi,wX\rangle -\langle \nu, X\rangle \geq 0$ for all $(\xi,\nu)\in C_{G\times K}(T^*G)$.
We assume that $H(w)$ contains an element $(\lambda,\mu)\in i\t^*_{G,K,>0}$, in other words that $F$ is a regular face.
Let $F=H(w)\cap C_{G\times K}(T^*G)$ be the corresponding face of
the polyhedral cone $C_{G\times K}(T^*G)$.

As $w$ is defined modulo the Weyl group of the stabilizer of $X$,
we may choose $w$ such that, if $\lambda$ is dominant for $G$, then $w^{-1}(\lambda)$ is dominant for $G(X)$.
If $(\lambda,\mu) \in F$, the couple $(w\lambda,\mu)$ is a couple of dominant weights for $(G(X), K(X)).$
The following proposition follows again from the $[Q,R]=0$ theorem of Meinrenken-Sjamaar (see also \cite{Par-Ve}).
\begin{proposition}
For any $(\lambda,\mu) \in F$, we have
$m_{G,K}(\lambda,\mu)=m_{G(X),K(X)}(w\lambda,\mu).$
\end{proposition}
A proof of this theorem using surjectivity of the restriction of homolorphic sections, with given invariance,
 is given in Ressayre \cite{Ressayre-reduction}.

\subsection{Branching theorem: a piecewise quasi-polynomial formula}\label{branch.general}

Our aim is to give an explicit   quasi-polynomial formula for
$m_{G,K}(\lambda,\mu)$ on a tope $\tau$ in terms of iterated residues.

 We assume, for simplicity, that $i\t_\k$ contains a regular (with respect to
$\Delta_\k$) element $X$ which is regular also for $\Delta_\g$ (this is not always the case).
We use this element to define  positive compatible root systems $\Delta^+_\g$ and $\Delta^+_\k$ as follows:
 $\Delta^+_\g:=\{ \alpha \in \Delta_\g , \alpha (X)>0\}$ and $\Delta^+_\k=\{ \alpha \in \Delta_\k, \alpha (X)>0\}.$  Given $\alpha \in \Delta^+_\g$, denote by $\overline {\alpha}$ the restriction  $\alpha_{|_{\t_k}}.$ Then $0 \neq  \overline{\alpha} \in  \Delta_\k^+.$
Thus the system  $\Psi$ in $i\t_\k^*$   consists on  the restrictions of $\Delta_\g^+$ repeated with multiplicities:
 $$\Psi=[
 \overline{\alpha}, \alpha \in \Delta^+_\g].$$
The system $\Psi$ contains $\Delta_\k^+$.
By our construction, all elements $\psi$ in $\Psi$ satisfy $\langle \psi, X\rangle >0$.

We denote by  $\Psi\backslash \Delta_\k^+$ the list where we have removed $\Delta_\k^+$ from $\Psi$.
More precisely, if $\alpha \in \Delta_\k^+$ occur in $\Psi$ with multiplicity $m>0$, then
 $\alpha \in \Delta_\k^+$ occur in
$\Psi\backslash \Delta_\k^+$  with multiplicity $m-1$.

\begin{example} \label{exarestricted}\end{example}Let $G=SU(n)$ and $K=SU(n_1)\times SU(n_2)$ with $n=n_1n_2$.
We consider $\t$ the Cartan subalgebra of $\g$ given by the diagonal matrices of trace zero and $\t_k=\t_1\times\t_2$
the Cartan subalgebras of $\k$ given by the corresponding diagonal matrices.
We take the embedding from $i\t_1\times i\t_2 \rightarrow i\t$
given by $$diag(a_1,\ldots,a_{n_1})\times diag(b_1,\ldots,b_{n_2})
  \rightarrow $$$$diag(a_1+b_1, a_2+b_1,\ldots,a_{n_1}+b_1,a_1+b_2, \ldots, a_{n_1}+b_2, \ldots, a_{1}+b_{n_2},\ldots, a_{n_1}+ b_{n_2}).$$
   We take the lexicographic order. The list of restricted roots is thus the list
   \begin{equation} \label{restricted} \Psi=[(a_i-a_j+b_k-b_\ell)].  \end{equation}There $i,j$  varies between $1$ and $n_1$,
and $k,\ell$ and varies between $1$ and $n_2$.
The couple $(i,k)$ being different from $(j,\ell)$, so all restricted roots are non zero.
This lexicographic order is compatible, as one check easily, that the restrictions are not zero and we get exactly
the order $a_1\geq a_2\geq \cdots \geq a_{n_1}$ on $\t_1$ and similarly for $\t_2.$

Explicitly, we can take the embedding of the element $$X=diag(n_1,\ldots,1)\times diag((n_2-1)n_1+1,\ldots,1).$$

For $n_1=2, n_2=3$ then $$X=diag(2,1),\times diag(5,3,1)$$
 and the embedded element is $diag(7,6,5,4,3,2).$

We do not know how to compute the set  $\CA(\Psi)$ of admissible hyperplanes for the system $\Psi$, for any $(n_1,n_2)$.
We computed it for a few examples (see also \cite{V-W}), but we do not see a general pattern.
Furthermore, the cardinal of the set $\CA(\Psi)$ (up to Weyl group action) seems to grow quickly.
$\Box$

%\end{example}

\bigskip

Let $\lambda\in \Lambda_{G,\geq 0}$
be the highest weight of an irreducible representation of $G$.
The character $\chi_{\lambda}$  of $V_\lambda^G$ is  given by  the Hermann Weyl character formula:
$$\chi_{{\lambda}|_{T_G}}=\sum_{ w \in W_\g}\frac{e^{w(\lambda)}}{\prod_{\alpha \in \Delta_\g^+}(1-e^{-w(\alpha)})}.$$
Restricting on  $T_K$, we obtain:
$$\chi_{{\lambda}|_{T_K}}=\sum_{ w \in \CW_\g}\frac{e^{{\overline{w(\lambda)}}}}{\prod_{\alpha \in \Delta_\g^+}(1-e^{-{\overline{w(\alpha)}}})}.$$

Let $w \in \CW_\g$.
Let $z\in (\t_\k)_\C$. Consider the meromorphic  function of $z$ given by
 $$s^w_{\lambda,\mu}(z)=\frac{e^{\la {\overline{w(\lambda)}-\mu},z\ra}}{\prod_{\alpha \in \Delta_\g^+}(1-e^{-\la\overline{w(\alpha)},z\ra})}.$$

Define
$$S^w_{\lambda,\mu}(z)=\prod_{\beta\in \Delta_\k^+}(1-e^{-\la \beta,z\ra}) s^w_{\lambda,\mu}(z)=
\frac{\prod_{\beta\in \Delta_\k^+}(1-e^{-\la \beta,z\ra})e^{\la \overline{w(\lambda)}-\mu,z\ra}}{\prod_{\alpha \in \Delta_\g^+}
(1-e^{-\la \overline{w(\alpha)}, z\ra })}
.$$

Consider a tope $\tau  \subset i\t_\g^*\oplus i\t_\k^*$ for the system of hyperplanes $\CF$.
  Let $(\xi,\nu)\in \tau$.
  Then for each $w\in  \CW_\g$, the element $\overline{w(\xi)}-\nu$ is $\Psi$-regular.
  Given $w\in \CW_\g$,   we note  $\a(\overline{w(\xi)}-\nu)$  the  tope for $\Psi$,
which contains the element $\overline{w(\xi)}-\nu.$
The tope $\a(\overline{w(\xi)}-\nu)$ depends only from $w$ and $\tau$,
so we denote it by $\a_w^{\tau}$.
We have defined the set ${\mathcal{OS}}(\Psi,\a^\tau_w)$ of adapted basis to the tope $a_w^{\tau}$ (Def. \ref{OS}).
Let $\Gamma_K$ be the dual lattice to $\Lambda_K$, and let $q=q(\Psi)$  be the index of $\Psi.$

\begin{proposition}\label{theo:ptau}
Let $\tau $ be a tope in $i\t^*_\g\oplus i\t^*_\k$ for the system $\CF$.
Define for $(\lambda,\mu)\in \Lambda_G\oplus \Lambda_K$,
$$p_\tau(\lambda,\mu)=\sum_{w \in \CW_\g}\sum_{\gamma\in \Gamma_K/q\Gamma_K}\sum_{\sigma \in  {\mathcal{OS}}(\Psi,\a^\tau_w)} Res_{{\overrightarrow{\sigma}}} S^w_{\lambda,\mu}(z+\frac{2i\pi}{q} \gamma).$$

Then, the function $p_\tau$ is a quasi-polynomial function on $\Lambda_G\oplus \Lambda_K$.

\end{proposition}

\begin{proof}
Explicitly,
$$p_\tau(\lambda,\mu)=$$$$\sum_{w \in \CW_\g}\sum_{\gamma\in \Gamma_K/q\Gamma_K}\sum_{\sigma \in  {\mathcal{OS}}(\Psi,\a^\tau_w)} Res_{{\overrightarrow{\sigma}}}  \frac{e^{\la \overline{w(\lambda)}-\mu,z+\frac{2i\pi}{q} \gamma\ra}}{\prod_{\alpha \in \Delta_\g^+} (1-e^{-\la \overline{w(\alpha)}, z+\frac{2i\pi}{q} \gamma\ra })}\prod_{\beta\in \Delta_\k^+}(1-e^{-\la \beta,z+\frac{2i\pi}{q} \gamma\ra}).$$
\end{proof}

\begin{theorem}
Let $\tau$ be a tope, and let $(\lambda,\mu)\in \overline \tau\cap \Lambda_{G,K,\geq 0}$.
Then
\begin{enumerate}
\item  if $(\lambda,\mu)\notin
C_{G,K}(T^*G)$,

$$m_{G,K}(\lambda, \mu)=p_\tau(\lambda,\mu).$$

\item if $(\lambda,\mu)\in C_{G,K}(T^*G)$, {\bf and} the tope $\tau$ intersect  $C_{G,K}(T^*G)$,
 then
$$m_{G,K}(\lambda, \mu)=p_\tau(\lambda,\mu).$$
\end{enumerate}
\end{theorem}

  To test if a regular element $(\lambda,\mu)$ is in the cone
  $C_{G,K}(T^*G)$. Take the tope $\tau$ containing $(\lambda,\mu)$.
  The point $(\lambda,\mu)$ is in the cone if and only if the quasi polynomial function
  $k\to p_\tau(k\lambda,k\mu)$ is non zero.

Before going into the proof of this theorem, let us make some remarks of how to somewhat reduce the complexity of this formula..

\begin{remark}\end{remark}
Since the system $\Psi$ contains the system $\Delta_\k^+$,   the function
 $S^w_{\lambda,\mu}(z)$ could be written as function with the smaller denominator $\prod_{\psi\in \Psi/\Delta_\k^+}(1-e^{-\la \psi,z\ra})$ and
 the residues could be taken over adapted $OS$ basis of the system $\Psi\backslash \Delta_\k^+$.
Not so many $w$  giving a non zero contribution to the formula.
Indeed, at least $w$ has to be such that $\overline{w\lambda}-\mu$ is in the cone generated by the restricted roots.
This is the so called valid permutations for $(\lambda,\mu)$ defined by Cochet in \cite{C2}, and his algorithm construct them recursively.
For example, as computed by Pamela Harris, \cite{PH}, when $\Psi$ is the system f positive roots of $A_r$, $\lambda$ the highest root and $\mu=0$, the number of $w$ such that $w\lambda$ is in the cone generated by the positive roots is the Fibonacci number $f(r)$, much smaller that $(r+1)!$, the order of the Weyl group.

$\Box$

\begin{proof}
Consider
$$\chi_{{\lambda}|_{T_K}}=\sum_{ w \in \CW_\g}\frac{e^{{\overline{w(\lambda)}}}}{\prod_{\alpha \in \Delta_\g^+}(1-e^{-{\overline{w(\alpha)}}})}.$$

 Using our regular element $X$, we rewrite this formula polarizing the linear form  $\overline{w(\alpha)}$: if
$\la\overline{w(\alpha)},X\ra<0$, we replace $\overline{w(\alpha)}$ by its opposite;
we then make  use of the identity  $\frac{1}{1-e^{-\beta}}=-\frac{e^{\beta}}{1-e^{\beta}}$.

Precisely, write  $\Psi=\Psi^1_{w} \cup \Psi^2_{w}$ as the disjoint union of  the two sets:
\noindent $$\Psi^1_{w}=\{\overline{ w\alpha} , \ \alpha \in \Delta_\g^+, \ \la\overline{w(\alpha)},X\ra >0\}$$
$$\Psi^2_{w}=\{- \overline{w\alpha},\ \alpha  \in \Delta_\g^+,  \ \la\overline{w(\alpha)},X\ra <0\}.$$
If  $s_w=|\Psi^2_{w}|$ and $e^{g_w}=\prod_{\overline{w\alpha} , \la\overline{w\alpha}  ,X\ra <0}  e^{\overline{w(\alpha)}  }$, then we obtain
that
$\chi_{{\lambda}|_{T_K}}$ is equal to

$$\sum_{ w \in \CW_\g}\left(\frac{e^{{\overline{w(\lambda)}}}}{\prod_{\psi  \in  \Psi^1_{w}}(1-e^{-\psi})}\frac{(-1)^{s_w}e^{ g_w}}{\prod_{\psi  \in \Psi^2_{w}}(1-e^{-\psi})}\right)=\sum_{ w \in \CW_\g}\left(\frac{e^{{\overline{w(\lambda)}}}(-1)^{s_w}e^{ g_w}}{\prod_{\psi \in  \Psi}(1-e^{-\psi})}\right).$$

\begin{lemma} The following holds on $T_K$:
$$\chi_\lambda=\sum_{ w \in \CW_\g}\sum_{\mu \in \Lambda^K}(-1)^{s_w} \CP_{\Psi}(\overline{w(\lambda)}+g_w-\mu) e^\mu$$
where $\CP_{\Psi}$ is the partition function determined by the restricted roots $\Psi.$
So
\begin{equation} \label{mult} m_{G,T_K}(\lambda,\mu)=\sum_{ w \in \CW_\g}(-1)^{s_w}\CP_{\Psi}(\overline{w(\lambda)}+g_w-\mu). \end{equation}
\end{lemma}

When $K$ is the maximal torus $T_G$, the formula above is Kostant multiplicity formula for a weight \cite{Kos}.
 The formula (\ref{mult}) is obtained by the same method.

Let us now use the formula
$$ m_{G,K}(\lambda,\mu) =
\sum_{\tilde w \in \CW_\k} \epsilon(\tilde w)m_{G,T_K}(\lambda,\mu -{\tilde w}(\rho_\k)+\rho_\k).$$

We obtain for $(\lambda,\mu)\in  \Lambda_{G,K,\geq 0}$

$$m_{G,K}(\lambda, \mu)=
\sum_{\tilde{w} \in \CW_\k} \epsilon(\tilde{w})\sum_{ w \in \CW_\g}(-1)^{s_w}\CP_{ \Psi}(\overline{w(\lambda)}+g_w -(\mu-\tilde{w}(\rho_\k)+\rho_\k))$$

(we may rewrite this expression as a sum of partition functions for $\Psi\backslash \Delta_\k^+$, obtaining Heckman formula, \cite{Heck},\cite{Knapp}).

Suppose first that $(\lambda,\mu)\in \tau\cap \Lambda_{G,K,\geq 0}$
is regular and "very" far away from all the walls $H(w)$.
So for all $w,\tilde w$,
the element $\overline{w(\lambda)}-\mu+g_w +\tilde{w}(\rho_\k)-\rho_\k$
is still in the tope $\a_{w}^{\tau}$ for $\Psi$ containing  $\overline{w(\lambda)}-\mu.$
We thus can employ the iterated residue formula
(Formula \ref{defres}) for $\CP_{ \Psi}(\overline{w(\lambda)}+g_w -(\mu-\tilde{w}(\rho_\k)+\rho_\k))$ on $\a_w^{\tau}$.
We obtain that $m_{G,K}(\lambda,\mu)$ is equal to
$$\sum_{\tilde{w} \in \CW_\k} \epsilon(\tilde{w})(-1)^{s_w}\sum_{w \in \CW_\g}\sum_{\gamma\in \Gamma^K/q\Gamma^K}
\sum_{\sigma \in   \mathcal{OS}(\Psi,\a^\tau_w)}
Res_{{\overrightarrow{\sigma}}}
\frac{e^{\la \overline{w(\lambda)}+g_w-(\mu-\tilde{w}(\rho_\k)+\rho_\k),
z+\frac{2i\pi}{q} \gamma\ra}}{\prod_{\psi \in \Psi} (1-e^{-\la \psi,z+\frac{2i\pi}{q} \gamma\ra})}.$$

Inverting the polarization process, we rewrite
$$(-1)^{s_w}\frac{e^{\la \overline{w(\lambda)}+g_w,z+\frac{2i\pi}{q} \gamma\ra}}
{\prod_{\psi \in \Psi} (1-e^{-\la \psi,z+\frac{2i\pi}{q} \gamma\ra})}=\frac{e^{\la \overline{w(\lambda)},
z+\frac{2i\pi}{q} \gamma\ra}}
{\prod_{\alpha \in \Delta_\g^+} (1-e^{-\la \overline{w\alpha},z+\frac{2i\pi}{q} \gamma\ra})}.$$

So,  remembering that $\prod_{\beta\in \Delta_\k^+}(1-e^{-\beta})=\sum_{\tilde{w}}e^{-\rho_\k+\tilde{w}(\rho_\k)}$, we obtain
$$m_{G,K}(\lambda, \mu)=p_\tau(\lambda,\mu)$$
when $(\lambda,\mu)\in  \tau\cap \Lambda_{G,K,\geq 0}$
is regular and "very" far away from all the walls $H(w)$.
Theorem \ref{theo:mGK} asserts that $m_{G,K}(\lambda, \mu)$ is given by a quasi polynomial formula
on $\tau\cap  \Lambda_{G,K,\geq 0}$
or even on $\overline\tau\cap \Lambda_{G,K,\geq 0}$
if  $\tau\cap i\t^*_{\g,\k,\geq 0}$ is contained in $C_{G,K}(T^*G)$.

Using Lemma \ref{lem:equaquasipoly} we obtain Theorem \ref{theo:ptau}.
 \end{proof}

\subsection{Singular case}\label{singular}

When  the stabilizer of $\lambda$ is large,
that is $\la\lambda,H_\alpha\ra$ is equal to $0$ for a large number of $\alpha$,
 then we  can rewrite   the Hermann Weyl formula for the character in a way that takes advantage of this.

Fix  a subset $\Sigma$ of the simple roots of $\Delta_\g^+$, and
let $\lp$ be the Levi subalgebra of $\g$, with simple root system $\Sigma$.
Let $\Delta_\lp^+$  be its positive root system.

Let $i \t_{\g,\Sigma}^*$ the set of the elements $\xi\in i\t_\g^*$   such that
$\la \xi,H_\alpha\ra=0$ for all $\alpha\in \Sigma$.
We define consistently
 $\t^*_{\g,\k,\Sigma, \geq 0},$
$\Lambda^{\Sigma}_{G}=\Lambda_G \cap i \t_{\g,\Sigma}^*$, a lattice in $i\t_{\g, \Sigma}^*$, $\Lambda_{G,K,\geq 0}^\Sigma=\Lambda_{G,K,\geq 0}\cap i \t^*_{\g, \Sigma}$,
and similarly.

We  also define
$$C^\Sigma_{G,K}(T^*G)=\{(\xi,\nu)\in C_{G,K}(T^*G); \xi\in i \t_{\g,\Sigma}^*\}.$$
In other words, $\nu$ must belong to the projection on $i\k^*$ of the singular orbit
$G\xi$.

The cone $ C^\Sigma_{G,K}(T^*G)$
is solid in $i \t_{\g, \Sigma}^*$, if and only if there exists $\xi\in i \t_{\g,\Sigma}^*$ such that
the projection  on $i\k^*$ of the singular orbit
$G\xi$ has a non zero interior in $i\k^*$. In other words the Kirwan polytope $\pi(G\xi)\cap i\t^*_\k$ is solid.

\begin{example}\end{example}
Consider the embedding of $K=U(n_2)\times U(n_3)/Z$ in $G=U(n_2n_3)$.
Here $n_2,n_3\geq 2$ and $Z$ is the subgroup
$\{t_2 Id, t_3 Id\}$ of the center of $U(n_2)\times U(n_3)$ with $t_2t_3=1$.
We take $\lambda$  a weight of $G$ with more than two non zero coordinates.
Then $\pi(G\lambda)$ has interior in $i\k^*$.

$\Box$

\bigskip

Let $$\Delta_{\uu}=\Delta_\g^+ \backslash \Delta_\l^+$$
and denote by $\CW_{\lp}$ the Weyl group of $\lp$.

For any $\lambda\in \Lambda_{G,\geq 0} \cap i\t_{\k,\Sigma}^*$,
we can write the character formula on $T_G$ and the restriction on $T_K$ as:
$$\chi_{{\lambda}|_{T_G}}=\sum_{ w \in \CW_\g/\CW_{\lp}}\frac{e^{w(\lambda)}}{\prod_{\alpha \in \Delta_{\uu}}(1-e^{-w(\alpha)})},$$
$$\chi_{{\lambda}|_{T_K}}=\sum_{ w \in \CW_\g/\CW_{\lp}}\frac{e^{{\overline{w(\lambda)}}}}{\prod_{\alpha \in \Delta_{\uu}}(1-e^{-{\overline{w(\alpha)}}})}.$$

\noindent In the regular case, the Levi component is just the Cartan subalgebra , the parabolic is the Borel subalgebra, and $\Delta_\uu=\Delta^+_\g$.

To compute $m_{G,K}(\lambda,\mu)$ for
$\Lambda_{G,\geq 0} \cap i\t_{\k,\Sigma}^*$ by iterated residues, it is tempting to replace the
function
$S^w_
{\lambda,\mu}$ by
the function
$$S^{\Sigma, w}_{\lambda,\mu}(z)=\prod_{\beta\in \Delta_\k^+}(1-e^{-\la \beta,z\ra})
\frac{e^{\la{\overline{w(\lambda)}-\mu,z\ra}}}{\prod_{\alpha \in \Delta_{\uu}}(1-e^{-{\la\overline{w(\alpha)},z\ra}})}.
$$

We consider the system of  hyperplanes $\CF_\Sigma$ in
$i\t^*_{\g, \Sigma}\oplus i\t^*_\k$ defined by the equations
$\langle \xi, wX\rangle -\langle \nu, X\rangle=0$, where $X$ is an equation
for a $\Psi$-admissible hyperplane, and $w\in \CW_\g$.
 If $(\xi,\nu)\in i\t^*_{\g, \Sigma}\oplus i\t^*_\k$ is in  a tope $\tau_\Sigma$ for $\CF_\Sigma$, then $(\xi,\nu)$ is in a unique tope  $\tau$ for $\CF$,
 and determines a tope $\a^{\tau}_w$ in $i\t^*_\k$.

The following proposition is clear.

\begin{proposition}\label{pro:branchfunct}
Let
$$p_\tau^\Sigma(\lambda,\mu)=\sum_{w \in \CW_\g/ \CW_\lp}\sum_{\gamma\in \Gamma_K/q\Gamma_K}
\sum_{\sigma \in  {\mathcal{OS}}(\Psi,\a^{\tau}_w) }
 Res_{{\overrightarrow{\sigma}}} S^{\Sigma,w}_{\lambda,\mu}(z+\frac{2i\pi \gamma}{q}).$$
Then $p_\tau^\Sigma(\lambda,\mu)$ is a quasi polynomial function  on
$\Lambda^\Sigma_{G}\oplus \Lambda_K$.
\end{proposition}

Remark that if  $\tau_\Sigma$ is a tope for $\CF_\Sigma$,  then
$\tau _\Sigma\cap C^\Sigma_{G,K}(T^*G)$ is empty if
$C^\Sigma_{G,K}(T^*G)$ is not solid.
If the cone $ C^\Sigma_{G,K}(T^*G)$ is solid, it is the  union of the closures of the open sets
$\tau_\Sigma\cap \t^{\Sigma,*}_{\g,\k,\geq 0}$
 contained in $C^\Sigma_{G,K}(T^*G)$.

\begin{theorem}\label{theo:branch}
Let  $\tau_\Sigma$  a tope in $i\t^*_{\g, \Sigma}\oplus i\t^*_\k$ for $\CF_\Sigma$,
 and let $(\lambda,\mu)\in \overline \tau_\Sigma\cap \Lambda^\Sigma_{G,K,\geq 0}$.
Then
\begin{enumerate}
\item if $(\lambda,\mu)\notin C^\Sigma_{G,K}(T^*G)$,

$$m_{G,K}(\lambda, \mu)=p^\Sigma_\tau(\lambda,\mu).$$
\item  if $(\lambda,\mu)\in C^\Sigma_{G,K}(T^*G)$, {\bf and} the tope $\tau_\Sigma$ intersect  $C^\Sigma_{G,K}(T^*G)$,
 then
$$m_{G,K}(\lambda, \mu)=p^\Sigma_\tau(\lambda,\mu).$$
\end{enumerate}
\end{theorem}

\begin{proof}
If the set $\tau_\Sigma\cap C_{G,K}^{\Sigma}(T^*G)$,
 is  contained in  $\tau \cap C_{G,K}(T^*G)$,
so we  know (by before)  that on  $\overline \tau\cap  \Lambda_{G,K,_\geq 0}$,
the function $m_{G,K}$ is given by a quasi polynomial formula, so a fortiori its restriction to
$\overline {\tau_\Sigma} \cap  \Lambda^{\Sigma}_{G,K, \geq 0}$.
So it is sufficient to prove that when
$\tau_\Sigma\cap  \Lambda^{\Sigma}_{G,K, \geq 0}$
is sufficiently far away from all walls belonging to $\CF_\Sigma$, then $m_{G,K}(\lambda,\mu)$ coincide with $p^\Sigma_{\tau}$.

The proof is very similar to the preceding proof, so we skip  details.

Let $w\in \CW_\g$.
Because of our assumption on compatible systems, namely the existence of a regular element $X,$
 we can define  $\Psi_{w,\uu}=\Psi^1_{w,\uu} \cup \Psi^2_{w,\uu}$ with
$\Psi^1_{w,\uu}=\{ \overline{w\alpha}, \alpha \in \Delta_\uu, \ \la\overline{w(\alpha)},X\ra >0\}$,
 $\Psi^2_{w,\uu}=\{-\overline{w( \alpha)}, \alpha  \in \Delta_\uu,  \ \la\overline{w(\alpha)},X\ra <0\}.$
Elements in $\Psi_{w,\uu}$ are positive on $X$, so
$\Psi_{w,\uu}$ is contained in $\Psi$.
In contrast to the regular case, $\Psi_{w,\uu}$ depends of $w$
and may not contain $\Delta_\k^+$.

Let  $s^\Sigma_w=|\Psi^2_{w,\uu}|$ and
$e^{g^\Sigma_w}=\prod_{\overline{w(\alpha )},   \ \la \overline{w(a)},X\ra <0}e^{\overline{w(\alpha)}  }$
then we obtain
$$\chi_{{\lambda}|_{T_K}}=
\sum_{ w \in \CW_\g/\CW_{\lp}}\left(\frac{e^{{\overline{w(\lambda)}}}(-1)^{s^\Sigma_w} e^{g^\Sigma_w}}
{\prod_{\psi \in  \Psi_{w,\uu}}(1-e^{-\psi})}\right)$$
and
$$m_{G,K}(\lambda, \mu)=\sum_{\tilde{w} \in \CW_\k} \epsilon(\tilde{w})
\sum_{ w \in \CW_\g/\CW_{\lp}}(-1)^{s^\Sigma_w}\CP_{ \Psi_{w,\uu}}
(\overline{w(\lambda)}+g^\Sigma_w -(\mu-\tilde{w}(\rho_\k)+\rho_\k))
.$$

The point  $(\lambda,\mu)$ being in $\tau_\Sigma$, the point
$\overline{w(\lambda)} -\mu$ is in $\a_w^{\tau}$.
We can assume that $(\lambda, \mu)$ is sufficiently far away from all walls, so that
$\overline{w(\lambda)}+g^\Sigma_w -(\mu-\tilde{w}(\rho_\k)+\rho_\k))$ is also in $\a_w^{\tau}$,
so that we can apply the iterated residue formula
for $$\CP_{ \Psi_{w,\uu}}
(\overline{w(\lambda)}+g^\Sigma_w -(\mu-\tilde{w}(\rho_\k)+\rho_\k))$$
as a sum of iterated residue with respect to adapted OS basis for
$\a^{\tau}_w$.
Then we proceed as in the preceding case, reversing the polarization process, and obtain
$$m_{G,K}(\lambda,\mu)=p^\Sigma_{\tau}(\lambda,\mu)$$
provided $(\lambda,\mu)$ is  in $\tau_\Sigma$ and sufficiently for away from the walls.

If $\overline\tau \cap
C^{\Sigma}_{G,K}(T^*G)\subset \overline\tau \cap
C_{G,K}(T^*G)$
the formula is quasi polynomial, so we obtain our theorem.

\end{proof}
\section {Kronecker coefficients and examples}\label{section:kro}
We describe now our approach  to compute  Kronecker coefficients, the result is summarized in Corollary \ref{reductionKro}.

Consider  $N = n_1\cdots n_s$ and assume that $n_1$ is the maximum of the $n_i$.
\noindent Write $\CH=\C^N=\C^{n_1}\otimes \cdots  \otimes \C^{n_s}$ and  consider the action of $U(n_1)\times \cdots \times U(n_s)$
in $Sym(\CH)$.

Thus $$Sym(\CH)=\sum g(\mu_1\otimes \cdots \otimes \mu_s)V_{\mu_1}^{U(n_1)}\otimes \cdots\otimes V_{\mu_s}^{U(n_s)}.$$
We want to compute
$g(\mu_1,\ldots, \mu_s)$
and the dilated coefficients.

Let $M=n_2\cdots n_s$, and consider the
embedding of $K=U(n_2) \times \cdots \times  U(n_s)$  in $G=U(M)$ determined by: $(k_2,\cdots ,k_s)(v_2\otimes \cdots , v_s)=k_2v_2\otimes \cdots k_sv_s,$  as we explained in Example \ref{exa:Kron}.

Using the Cauchy formula \ref{exa:Cauchy formula} we can write
$$Sym(\C^N)=Sym(\C^{n_1}\otimes \C^M)=\sum_{\mu_1 \in P\Lambda_{U(n_1),\geq 0}} V_{\mu_1}^{U(n_1)} \otimes V_{\tilde\mu_1} ^{U(M)}.$$

 Write the decomposition of $V_{\tilde\mu_1}^{U(M)}$  restricted to $K$:
$$V_{\tilde\mu_1}^{U(M)}=\oplus_{\mu_2\in \hat{U}(n_2), \cdots, \mu_s \in \hat{U}(n_s)} m_{G,K}(\tilde{\mu_1}, \mu_2\otimes\cdots \otimes\mu_s)V_{\mu_2}^{U(n_2)}\otimes \cdots \otimes V_{\mu_s}^{U(n_s)}.$$
Remember that the polynomial irreducible  representation of $U(n_k)$   are parameterized by the highest weight $\gamma=[\gamma_1,\ldots, \gamma_{n_k}]$ with $\gamma_1\geq \gamma_2\geq \cdots \geq \gamma_{n_k}\geq 0$
and that  $|\gamma|=\sum_i\gamma_i.$  Taking care of the fact that if $\mu_1\otimes\cdots \otimes \mu_s$  occur in $Sym(\C^{n_1}\otimes  \cdots \otimes \C ^{n_s})$ we must have that the actions on the centers must be  the same, we obtain:
 \begin{corollary}\label{reductionKro}
$$g(\mu_1\otimes\cdots\otimes\mu_s)=\left \{\begin{array}{ll} m_{G,K} (\tilde{\mu_1}, \mu_2\otimes\cdots \otimes\mu_s) \ \  {\text {if $|\mu_1|=|\mu_2|=\cdots =|\mu_s|$}}\\
  0  \ \ {\text{otherwise}}\end{array} \right.$$
\end{corollary}

We take advantage of the formula to somewhat reduce the computation of a tensor product with $s$ factors
to an analogous computation with $s-1$ factors.
So we will compute some of the Kronecker coefficients as a corollary of  the  branching theorem.

The case
$\C^2\otimes \C^2 \otimes \C^4$  corresponds to give explicit formulae for the decomposition of the representation of the symmetric group  associated to partitions with at most two rows.
 Complete expressions for these functions have already been obtained by
Briand et al.

We list the examples we can compute in Subsection \ref{application}. For example,
we can compute  the dilated Kronecker coefficients for
 $\C^3\otimes \C^3 \otimes \C^3$, as well as some examples for  $\C^3\otimes \C^3 \otimes \C^4$.

\subsection{The algorithm to compute Kronecker coefficients}\label {thealgoKro}
We refer to Section \ref{singular}  and Section \ref{section:kro} for the notation.

We are given  a sequence of $s$ strictly positive  integers $[n_1,\ldots,n_s]$ and for each integer $n_i$ a sequence  $\nu_i$ of integers:  $\nu_i =[\nu_1^i,\cdots,\nu_{n_i}^i]$ with $\nu_1^i\geq \nu_2^i\geq\cdots \geq \nu_{n_i}^i\geq 0.$ Each $\nu_i$ parametrizes  an  irreducible polynomial representation of $U(n_i)$ of highest weight $\nu_i$. Write  $N=\prod_{i=1}^s n_i$ and $M=\prod_{i=2}^sn_i.$
We want to compute the Kronecker coefficients $g(k\nu_1,\cdots, k\nu_{s})$  dilated by an integer $k$ that is
the multiplicity of the tensor product representation $k\nu_1\otimes \cdots \otimes k\nu_{s}$ in $Sym(\C^N).$
Our approach uses Cauchy formula together with the computation of the branching coefficients to reduce the number of parameters.
 We may assume that $n_1\leq M$ and that $|\nu_1|=|\nu_2|=\cdots =|\nu_s|$.

We set $G=U(M)$ and $K=U(n_2) \times \cdots \times  U(n_s).$

The  first reduction step is:

$\bullet$ \   If $|\nu_1|=|\nu_2|=\cdots =|\nu_s|$ then $g(\nu_1\otimes\cdots\otimes\nu_s)=
m_{G,K} (\tilde{\nu_1}, \nu_2\otimes\cdots \otimes\nu_s)   $
where $\tilde{\nu_1}$ is the highest weight representation of $U(M)$ obtained by $\nu$ adding
$M-n_1$ zeros and the branching coefficient $m_{G,K}$ is computed in Theorem \ref{theo:branch}
via the function defined in Proposition \ref{pro:branchfunct}.

Let us write $\lambda= \tilde{\nu_1}$,  $\mu=\nu_2\otimes\cdots\otimes\nu_s$.
If $n_1<M$, $\lambda$ is a singular weight for the group $U(M).$
Denote by $\Sigma$ the set of simple roots $[e_{n_1+2}-e_{n_1+1},\ldots, e_M-e_{M-1}]$
of $U(M)$. Let $\lp=\u(M-n_1)$ the Lie algebra
with this simple root system.
We have $\la \lambda, H_\alpha\ra=0$ for all $\alpha\in \Sigma$.

   Then
$(\lambda,\mu) \in \Lambda_{G,\geq 0}^\Sigma \oplus \Lambda_{K\geq 0}$ with the notations as in \ref{singular}.
Let us review the key steps of the algorithm to compute $m_{G,K}$.
See  the discussion in  \ref{difficult} outlining   the limits of the implementation.

 Given as input  $(\lambda,\mu)$, we wish to compute the branching coefficients.
 Recall that :

\begin{equation} \label{multkro}
 m_{G,K}(\lambda,\mu)=\sum_{w \in \CW_\g/ \CW_\lp}\sum_{\gamma\in \Gamma_K/q\Gamma_K}
\sum_{\sigma \in  {\mathcal{OS}}(\Psi,\a^{\tau}_w) }
 Res_{{\overrightarrow{\sigma}}} S^{\Sigma,w}_{\lambda,\mu}(z+\frac{2i\pi \gamma}{q}).
 \end{equation}
 One of the tricky point in computing the right hand  side  of equation (\ref{multkro})
  is  to find a $\CF_\Sigma$-tope $\tau_\Sigma$ such that $(\lambda,\mu) \in \overline{\tau_\Sigma}.$ We  do this by computing a regular point inside  the Kirwan cone sufficiently small  and deform
  $(\lambda,\mu)$ along the line from $(\lambda,\mu)$ to this interior point.
 \begin{enumerate}
\item We list all the  equations $X$  of  the  $\CF$- admissible hyperplanes.
\item  For each such equation given by $X$ and for $w \in W_\g$,
 we compute $H(X,w)=\la \lambda,wX\ra-\la\mu,X\ra$.  As $(\lambda,\mu)$ is in the lattice of weights,
 and $X$ in the dual lattice, $H(X,w)$ is an integer.

 Remember $\lambda, \mu$ are our fixed input.
\item \label{aa} If $H(X,w)\neq 0  \ \forall w, X$,
then $(\lambda,\mu)$ is  $\CF_\Sigma$-regular and then   it is in a tope $\tau_\Sigma.$ A fortiori it is in a unique $\CF$ tope $\tau$  and therefore  $\overline{w(\lambda)}-\mu$ is in a unique tope  $a_w^\tau\in i\t_\k^*$.

\noindent In conclusion we can  compute $\mathcal{OS}(\Psi,a^\tau_w)$.

\item Else  if  $H(w,X)=0$ for some $X$  and $w$, then we deform as follows:
\begin{enumerate}
\item We find $\epsilon=(\epsilon_1,\epsilon_2)$ in the interior of
$C^\Sigma_{G,K}$ .
We can find  this point, in the cases we treat,
because either we have the equations of the Kirwan cone, either  we know some points in the Kirwan cone by directly computing projections.

\item We rescale $\epsilon$ so that
$|\la wX,\epsilon_1\ra-\la X,\epsilon_2\ra|<1/2,$
so $(\lambda,\mu)+t\epsilon$ stays in the same tope $\tau_\Sigma$ for all $0<t<1$.
\item We define $(\lambda_{def}, \mu_{def})$ as
$(\lambda+\epsilon_1,\nu+\epsilon_2)$.
\end{enumerate}

\item We can now pick  the tope $\tau$ defined by $(\lambda_{def},\mu_{def})$ and  compute  $\mathcal{OS}(\Psi,a^\tau_w)$ as in step \ref{aa}.
\item  Now we compute $$S^{\Sigma, w}_{\lambda,\mu}(z)=\prod_{\beta\in \Delta_\k^+}(1-e^{-\la \beta,z\ra})
\frac{e^{\la{\overline{w(\lambda)}-\mu,z\ra}}}{\prod_{\alpha \in \Delta_{\uu}}(1-e^{-{\la\overline{w(\alpha)},z\ra}})}$$ and the residue along an $\mathcal{OS}$ basis with an appropriate series expansion.
\item  Finally to compute $m_{G,K}(\lambda,\mu),$ we  have to sum  the contribution from   $w \in \CW_\g/ \CW_\lp $,  over the set     $\gamma\in \Gamma_K/q\Gamma_K$  and $\sigma \in \mathcal{OS}(\Psi,a^\tau_w).$  Each individual term of these two sums,
    that is if we fix  a $\gamma$ and a $\sigma,$ is easy to compute,
    but there can be really many of these terms, making possibly the computation very long.
\end{enumerate}
Observe once again that in particular, we can  test if the point $(\lambda,\mu)$ is in the cone
$C_{G,K}^\Sigma(T^*G)$ or not, according if the output is not zero or zero.

It is not more difficult to compute with symbolic variables $(\lambda,\mu)$ belonging to the closure of a tope.
However to describe all possible topes (the chamber decomposition of $C_K(\CH)$)  seems very difficult.
So our input is $(\lambda,\mu)$, the output is either the numeric value, either the dilated coefficient
$k\to m_{G,K}(k\lambda,k\mu)$,
or (in low dimensions), a tope $\tau$ containing
$\lambda,\mu$ in its closure and the quasi-polynomial function in both variables $\lambda,\mu$
coinciding with $m_{G,K}(\lambda,\mu)$ on the tope $\tau$.

\begin{remark} Rectangular tableaux\end{remark}
Remark that if $\mu_1$ is a rectangular tableau, then $\lambda$ is even more singular.
This enable us to compute more easily using a larger set $\Sigma$ (reducing then the number of parabolic roots).
When all $\mu_i$ are rectangular tableaux, this corresponds to the case of Hilbert series.
We list the corresponding results in the last subsection \ref{HS}.

\subsection{Examples}\label{application}
The first two examples have  already been treated in the literature.

\begin{example}{ The case of 3-qbits :  $\C^2\otimes\C^2\otimes\C^2$ } \end{example}
This case has been treated in complete details in  \ref{exa3qubits} and is due to \cite{BOR1}.

\begin{example}{ The case of :  $\C^4\otimes\C^2\otimes\C^2$ } \end{example}
This example has been studied in complete details by \cite{BOR1}.
The number of chambers of polynomiality is $74$ and on each chamber the quasipolynomial is of degree 2.

The multiplicity function $k\to m(k\lambda,k\mu,k\nu)$ is a quasi polynomial function of the form
$$f(k)+(-1)^k c$$ where $f(k)$ is a polynomial of at most degree two and
$c$ is a constant.
Here is an example.
For $\lambda=[5,3,2,1]$, $\mu=\nu=[6,5]$, then we obtain
$$1/4\,{k}^{2}+1/2\,k+5/8+3/8\, \left( -1 \right) ^{k}$$

Remark that all points $(\alpha,\beta,\gamma)$ in the boundary of the Kirwan cone are stable, thus $g(k\alpha,k\beta,k\gamma)$ is $0$ or $1$.

\begin{example}{ The case of 4-qbits :  $\C^2\otimes\C^2\otimes\C^2\otimes \C^2$ } \end{example}

We consider the action of of $K=U(2)\times U(2)\times U(2)\times U(2)$ on  $\CH=\C^2\otimes \C^2\otimes \C^2\otimes \C^2$
with  $(k_1,k_2,k_3,k_4) $ acting as $k_1\otimes k_2\otimes k_3\otimes k_4$.
The Kirwan polytope has been described by  Higuchi-Sudbery-Szulc, \cite{H-Su-Sz}.

We have no idea of the number chambers for  polynomiality. Nevertheless, given  highest weights
$\alpha,\beta,\gamma,\delta$, we can compute
$g(k\alpha,k\beta,k\gamma,k\delta)$ as a periodic polynomial in $k.$

 It is a  polynomial of degree  at most $7$ and period $6.$
Here is an example. When $\alpha=\beta=\gamma=\delta=[2,1]$ then we compute:
$$m(k\alpha,k\beta,k\gamma,k\delta)=$$
$$\frac{23}{241920}k^7+\frac{13}{5760}k^6+\frac{155}{6912}k^5+\frac{139}{1152}k^4
+\left(\frac{81601}{207360}+\frac{1}{1536}(-1)^k\right )k^3+$$$$\left(\frac{9799}{11520}+(-1)^k\frac{ 5}{256}\right)k^2
+\left(\frac{38545}{32256}+(-1)^k \frac{179}{1536}\right )k+P(k)$$
where
$$P(k)= \left( {\frac {5}{243}}+{\frac {1}{243}}\,\theta\right)
 \left( {\theta}^{2} \right) ^{k}+ \left( {\frac {4}{243}}-{\frac {1}{
243}}\,\theta \right) {\theta}^{k}+{\frac {5279}{6912}}+{\frac {
51}{256}}\, \left( -1 \right) ^{k}
$$ is of period 6  and $\theta$ is a primitive root $\theta^3=1$.
The  values of $P(k)$ on $0,1,2,3,4,5$ are
$$[1, \frac{5725}{10368}, \frac{76}{81}, \frac{77}{128}, \frac{77}{81}, \frac{5597}{10368}]$$

Figure \ref{DHM4qu}  shows the Duistermaat-Heckman measure for four  qbits. The drawing is along  the line
from $v_{min}=[[1/2,1/2],[1/2,1/2],[1/2,1/2],[1/2,1/2]]$ to $v_{top}=[[1,0],[1,0],[1,0],[1,0]]$.
The function $t\to DH_K^\CH(tv_{min}+(1-t)v_{top})$  is a  spline of degree $7$ with
singularities at $t=[0,\frac{1}{5},\frac{1}{3},\frac{1}{2},1]$.

\begin{figure}[h]
\centering
\includegraphics[width=1.5 in]{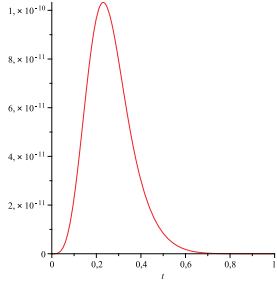}
\caption{Duistermaat-Heckman measure for four  qbits}
\label{DHM4qu}
\end{figure}

\begin{example}{ The case of :  $\C^6\otimes \C^3\otimes \C^2$ }\label{exa:632} \end{example}

When $n_2=3$, $n_3=2$, it is sufficient to consider  the case when  $n_1=6$.

Then the multiplicity function $k\to m(k\lambda,k\mu,k\nu)$ is a quasi polynomial function of the form
$$f(k)+(-1)^k g(k)+h(k)$$ where $f(k)$ is a polynomial of $k$  generally of  degree  $8$, $g(k)$ of degree $2$ and $h(k)$ is a periodic function of $k \mod \  6.$

Here is an example where the degree of the polynome is the expected one.

We fix   $\lambda=[1500,1052,940,492,268,156],\ \mu=[2110,1438,860], \nu=[2748,1660]$  and  compute:
$$m(k\lambda,k\mu,k\nu)=$$$${\frac {20160143036868818273}{540}}\,{k}^{8}+{\frac {
2401100429169038668}{945}}\,{k}^{7}+{
\frac {236265968398572733}{3240}}\,{k}^{6}+$$$${\frac {311654396584249}{270}}\,{k}^{5}+{\frac {
11934644414969}{1080}}\,{k}^{4}+{\frac {53468894201}{810}}
\,{k}^{3}+{\frac {44636639}{180}}\,{k}^{2}+{\frac {378739}{630}}\,k+$$$$
{\frac {232}{243}}+{\theta}^{k} \left( {\frac {2}{81}}+{\frac {1
}{243}}\,\theta \right) + \left( {\theta}^{2} \right) ^{k}
 \left( {\frac {5}{243}}-{\frac {1}{243}}\,\theta \right)
$$
where   $\theta$  is a  third primitive root of 1.
The term of the degree zero in  $m(k\lambda,k\mu,k\nu)$ is a periodic function of $k$   whose values are
$$\left [1,{\frac {25}{27}},{\frac {76}{81}},1,{\frac {25}{27}},{\frac {76}{81
}}\right ]$$

We finish by noticing that for $k=1\ldots 6,$ the values of  $m(k\lambda,k\mu,k\nu)$ are  given by the following list
$$[1, 39948532219001323, 9887333657571493818, 250556011548476811713, $$ $$2488623801870416780185, 14783083490287618355455].$$
We now compute examples of multiplicity on the walls.

The walls of the Kirwan cone  have been described by Klyachko in \cite{Kl}.
Given $\lambda=[\lambda_1, \lambda_2,\lambda_3,\lambda_4,\lambda_5,\lambda_6],\  \mu=[\mu_1,\mu_2,\mu_3]$ and $ \nu=[\nu_1,\nu_2,]$,
 the walls  are by  the following $5$ types of inequalities
in $\lambda,\mu,\nu.$

More precisely, for each of the inequality below, there is a particular non empty subset $S$ of $\Sigma_6\times \Sigma_3 \times \Sigma_2$ (where $\Sigma_k$ is the
of permutations of $k$ elements) computed by Klyachko such that the permuted inequality is a wall of the corresponding Kirwan cone.

\[
\begin{tabular}{|c|c|}\hline
type& equation \\
\hline\hline
I &$\nu_{{1}}-\nu_{{2}}-\lambda_{{1}}-\lambda_{{2}}-\lambda_{{3}}+\lambda_{{4}}+\lambda_{{5}}+
\lambda_{{6}}
  \leq 0 $  \\
  \hline
  II&$\mu_{{1}}+\mu_{{2}}-2\,\mu_{{3}}-\lambda_{{1}}-\lambda_{{2}}-\lambda_{
{3}}-\lambda_{{4}}+2\,\lambda_{{5}}+2\,\lambda_{{6}}\leq 0$
\\ \hline
III&$2\,\mu_{{1}}-2\,\mu_{{3}}+\nu_{{1}}-\nu_{{2}}-3\,\lambda_{{1}}-\lambda
_{{2}}-\lambda_{{3}}+\lambda_{{4}}+\lambda_{{5}}+3\,\lambda_{{6}}\leq 0$\\
\hline
IV&$2\,\mu_{{1}}+2\,\mu_{{2}}-4\,\mu_{{3}}+3\,\nu_{{1}}-3\,\nu
_{{2}}-5\,\lambda_{{1}}-5\,\lambda_{{2}}+\lambda_{{3}}+\lambda_{{4}}+\lambda_{{5}}+7\,\lambda_
{{6}}\leq 0$\\ \hline
V&$4\,\mu_{{1}}-2\,\mu_{{2}}-2\,\mu_{{3}}+3\,\nu_{{1}}-3\,\nu
_{{2}}-7\,\lambda_{{1}}-\lambda_{{2}}-\lambda_{{3}}-\lambda_{{4}}+5\,\lambda_{{5}}+5\,\lambda_
{{6}}\leq 0$\\ \hline
  \end{tabular}
\]

The following table list  elements $v=[\lambda,\mu,\nu]$ in the relative interior of the corresponding wall type facet, together with  the value of  $m(k\lambda,k\mu,k\nu).$
As asserted in Lemma \ref{max},
for the listed value $m(k\lambda,k\mu,k\nu)$ has the maximum possible degree.

\small{\[\begin{tabular}{|l|l|l|}\hline
vector & $[\lambda,\mu,\nu] $&$m(k\lambda,k\mu,k\nu)$\\
\hline\hline
$v_I$&${\small{[[288, 192, 174, 120, 30, 6], [343, 270, 197], [654, 156]]}}$&$1+17k$\\
  \hline
$v_{II}$&$[[600, 372, 300, 156, 96, 12], [876, 552, 108], [930, 606]]$&$1+{\frac {311}{2}}\,k+{\frac {21051}{2}}\,{k}^{2}+242154\,{k}^{3}$\\ \hline
$v_{III}$&$[[188, 140, 92, 52, 20, 4], [304, 152, 40], [340, 156]]$&$1$\\
\hline
$v_{IV}$&$[[276, 204, 120, 66, 30, 6], [351, 273, 78], [552, 150]]$&$1+36k$\\ \hline
$v_V$&$[[276, 198, 126, 66, 48, 6], [406, 201, 113], [536, 184]]$&$1+41k$\\ \hline
  \end{tabular}\]}

In particular, we see that all facets of type $v_{III}$ consists of stable elements. There are $20$ such facets (obtained by considering the special permutations computed by Klyachko).

As we noticed just after Theorem \ref{theo:beforemax} we don't know how to compute the degree when $\mu$ is singular, see  Subsection \ref {HS} for the case of three rectangular tableaux.
 Here is an example for which the degree is smaller. Consider  $\lambda=[9,7,5,3,2,1]$, $\mu=[9,9,9]$, $\nu=[14,13]$,  then

$m(k\lambda,k\mu,k\nu)$ is given by the following formula

$\left( {\frac {13}{64}}\, \left( -1 \right) ^{k}+{\frac {67}{64}}
 \right) k+{\frac {17}{12}}\,{k}^{2}+{\frac {617}{432}}\,{k}^{3}+{
\frac {19}{24}}\,{k}^{4}+{\frac {55}{288}}\,{k}^{5}+{\frac {1}{81}}\,{
\theta}^{k} \left( -2\,\theta+8 \right) +{\frac {1}{81}}\,
{\theta}^{2\,k} \left( 2\,\theta+10 \right) +{\frac {85}{
144}}+\frac{3}{16}\, \left( -1 \right) ^{k}$

 Here $\theta$ is  again a primitive root $\theta^3=1$.
 Thus the  term of degree zero is a periodic function $r$ of $k$ such that
$$[r(0),r(1), r(2),r(3),r(4),r(5)]=[1, 71/216, 17/27, 5/8, 19/27, 55/216]$$
Of course, the value of $m(0,0,0)$ is equal to $1$.
 Here $m(\lambda,\mu,\nu)=5$
 and for instance the value $m(17\lambda, 17\mu, 17\nu)=344715.$

\begin{example}\label{C333}{ The case of 3-quthrits. :  $\C^3\otimes\C^3\otimes\C^3$ } \end{example}
Here  $\CH=\C^3\otimes \C^3\otimes \C^3.$
The multiplicity function $k\to m(k\lambda,k\mu,k\nu)$ is a quasi polynomial function of the form
$$f(k)+(i)^k g(k)+h(k)$$ where $f(k)$ is a polynomial of $k$ of  degree at most $11$ and $h(k)$ is a periodic function of $k\pmod {12}.$
The actual numerical values are computed in a rather quick time.

Let us give an example of the dilated Kronecker coefficient.
The periodic term for the coefficient of degree 0
$m(k\lambda,k\mu,k\nu)$ with $\lambda=\mu=\nu=[4,3,2]$ is given by
{\small{$$\left [\small{1, \frac{1166651}{5308416}, \frac{13403}{20736}, \frac{29899}{65536}, \frac{59}{81}, \frac{1166651}{5308416}, \frac{235}{256}, \frac{980027}{5308416}, \frac{59}{81}, \frac{32203}{65536}, \frac{13403}{20736}, \frac{980027}{5308416}}\right]$$}}
In this case $m(k\lambda,k\mu,k\nu)$ has precisely degree 11.

 Here is a numerical example.  For  $\lambda=[5,2,1]$, $\mu=[4,2,2]$, $\nu=[3,3,2]$ we compute
$m(\lambda,\mu,\nu)=4$.

 We close with one last example that we had promised.
 \begin{example}\label{firstnonzero} \end{example}
Here  $\CH=\C^3\otimes \C^3\otimes \C^2.$
 $$g([k[1,1,1],k[1,1,1],k[2,1]])=\frac{1}{4}\, \left( -1 \right) ^{k}+{\frac {5}{12}}+ \left(1 -\theta
 \right) \frac{1}{9}{\theta}^{2\,k}+ \left( 2+\theta \right) {\frac{1}{9}\theta}^{k}+\frac{1}{6}k$$
 and for instance the list of the values for $0\leq k\leq 30$ are
$$[1, 0, 1, 1, 1, 1, 2, 1, 2, 2, 2, 2, 3, 2, 3, 3, 3, 3, 4, 3, 4, 4, 4, 4, 5, 4, 5, 5, 5, 5, 6]$$
So the saturation factor is $2$.

\subsection{Rectangular tableaux and Hilbert series}\label{HS}

\noindent We give a list of the Kronecker coefficients  for the following situation of rectangular tableaux.
We use the following notations:  $(\C^2)^3=\C^2\otimes\C^2\otimes\C^2$, $[[1,1]]^3=[[1,1],[1,1],[1,1]],$ $\C^{[4,3,3]}=\C^4\otimes \C^3\otimes \C^3$ and similarly.
In  the following Table the second column refers to the choice of the parameters $[\lambda,\mu,\nu]$ and the third column to the value of the Hilbert series $\sum_k m(k)t^k$ where $m(k)$ is the Kronecher coefficient $g(k\lambda, k\mu,k\nu)$.

\noindent \[\begin{tabular}{|l|c|l|}\hline
type& parameters& value\\
\hline\hline
$(\C^2)^3$&$[[1,1]]^3 $ &$\frac{1}{1-t^2}$\\
  \hline
$(\C^2)^4$&$[[1,1]]^4 $ &$\frac{1}{(1-t)(1-t^2)^2(1-t^3)}.$\\
  \hline
  $(\C^2)^5$&$[[1,1]]^5$ &$HS_{22222}$\\
  \hline
$(\C^3)^3$&$[[1,1,1]]^3 $ &$\frac{1}{(1-t^2)(1-t^3)(1-t^4)}.$\\
  \hline
  $\C^{[4,3,3]}$&$[[3,3,3,3], [4,4,4], [4,4,4]]$&$\frac {1+{t}^{9}}{ \left( 1-{t}^{2} \right) ^{2} \left(1- {t}^{4}
 \right)  \left( 1-t\right)  \left(1- {t}^{3} \right) }$\\
 \hline
  \end{tabular}
\]

where
$$HS_{22222}=\sum g(k[1,1], k[1,1], k[1,1], k[1,1], k[1,1]]) t^{k}=P(t)/Q(t)$$ where
$$P(t)=t^{52}+16\, t^{48}+9\, t^{47}+82\, t^{46}+145\, t^{45}+383\, t^{44}\nonumber\\
+770\, t^{43}+$$$$1659\, t^{42}+3024\, t^{41}+5604\, t^{40}+9664\, t^{39}\nonumber\\
+15594\, t^{38}+24659\, t^{37}+36611\, t^{36}+52409\, t^{35}\nonumber\\
+$$$$71847\, t^{34}+95014\, t^{33}+119947\, t^{32}+146849\, t^{31}\nonumber\\
+172742\, t^{30}+195358\, t^{29}+214238\, t^{28}+$$$$225699\, t^{27}\nonumber\\
+229752\, t^{26}+225699\, t^{25}+214238\, t^{24}+195358\, t^{23}\nonumber\\
+172742\, t^{22}+146849\, t^{21}+$$$$119947\, t^{20}+95014\, t^{19}\nonumber\\
+71847\, t^{18}+52409\, t^{17}+36611\, t^{16}+24659\, t^{15}+15594\, t^{14}\nonumber\\
+9664\, t^{13}+$$$$5604\, t^{12}+3024\, t^{11}+1659\, t^{10}+770\, t^{9}\nonumber\\
+383\, t^{8}+145\, t^{7}+82\, t^{6}+9\, t^{5}+16\, t^4+1$$\nonumber
and $$Q(t)=(1-t^2)^5(1-t^3)(1-t^4)^5(1-t^{5})(1-t^{6})^5.$$

We remark that for the $5$-qubits the  result in \cite{LUTHI}  correspond to the series $\sum_k m(k) t^{2k}$ and has a misprint on the value of the coefficient $a_n$ for $n=42$  (corresponding to the coefficient of $t^{21}$ in our formula), as the numerator is not palindromic. So the value $a_n$ for $n=42$  in \cite{LUTHI} has to be replaced by
$146849.$

We report for completeness the value of the Kronecker coefficients in the examples considered, we omit the
actual expression for the Kronecker coefficients in the $5$-qbits case and the one for $g(k[3,3,3,3],k[4,4,4],k[4,4,4])$ because the formula is too long to be reproduced here.
$$g(k[1,1],k[1,1])=\frac{1}{2}+\frac{1}{2}(-1)^k$$
$$g(k[1,1],k[1,1],k[1,1], k[1,1])=
$$$$\frac {23}{36}+\frac{1}{4}\, \left( -1 \right) ^{k}+\frac{1}{27}{\theta}^{k} \left( 2
+\theta \right) + \frac{1}{27}{\theta}^{2\,k} \left(1-\theta \right) +
 \left( {\frac {29}{48}}+\frac{1}{16}\, \left( -1 \right) ^{k} \right) k+\frac{1}{16}\,
{k}^{2}+{\frac {{k}^{3}}{72}}$$
 $$g(k[1,1,1],k[1,1,1],k[1,1,1])=$$
$${\frac {107}{288}}+{\frac {9}{32}}\, \left( -1 \right) ^{k}+\,\left(1 + \left( -1 \right) ^{k}\right)  \frac{1}{16}{i}^{k}+ \left( 1+
 \left( -1 \right) ^{k+1} \right)  \frac{1}{16}{i}^{k+1}+$$$$\frac{1}{9}\,{\theta}^{2\,k}
+\frac{1}{9}\,{\theta}^{k}+ \left( \frac{1}{16}\, \left( -1 \right) ^{k}+\frac{3}{16} \right) k+\frac{1}{48}\,{k}^{2}
$$ where $\theta$ is a third root of unity.
For  $g(k[1,1,1],k[1,1,1],k[1,1,1])$  we report, as an example, the expressions on cosets. We have twelve cosets and thus a sequence of 12 polynomials given by the following list
 {\small{$$ [1+\frac{1}{4}\,k+\frac{1}{48}\,{k}^{2},-{\frac {7}{48}}+\frac{1}{8}\ k+\frac{1}{48}\,{k}^{2},{\frac {5}{12}}+\frac{1}{4}k+\frac{1}{48}\,{k}^{2},{\frac {7}{16}}+\frac{1}{8}\, k+\frac{1}{48}\,{k}^{2},$$
  $$\frac{2}{3}+
\frac{1}{4}\,k+\frac{1}{48}\,{k}^{2},-{\frac {7}{48}}+\frac{1}{8}\,k+\frac{1}{48}\,{k}^{2},\frac{3}{4}+\frac{1}{4}\,k+
\frac{1}{48}\,{k}^{2},{\frac {5}{48}}+\frac{1}{8}\,k+\frac{1}{48}\,{k}^{2},$$$$\frac{2}{3}+\frac{1}{4}\,k+\frac{1}{48}\,{k
}^{2},\frac{3}{16}+\frac{1}{8}\,k+\frac{1}{48}\,{k}^{2},{\frac {5}{12}}+\frac{1}{4}\,k+\frac{1}{48}\,{k}^{2},{
\frac {5}{48}}+\frac{1}{8}\,k+\frac{1}{48}\,{k}^{2}]
$$}}

The following is the list of  values of the Kronecker coefficients computed by the above formula for $0\leq k\leq 20:$
$$[1, 0, 1, 1, 2, 1, 3, 2, 4, 3, 5, 4, 7, 5, 8, 7, 10, 8, 12, 10, 14].$$

\noindent Once again the saturation factor is $2.$

{\bf{Acknowledgments.}}
We  are  grateful  to Michel Duflo and Michael Walter  for
suggestions and comments.   Part of the work for this article was made during the period the authors spent at  the Institute for Mathematical Sciences(IMS) of the National University of Singapore in November/December 2013. The support received is gratefully acknowledged.
The first author was also partially supported by a PRIN2012 grant.

\end{document}